\documentclass[11pt]{article}

\usepackage{graphicx}
\usepackage{color,url,amssymb}
\usepackage{amsmath,mathtools,amsthm}
\usepackage{verbatim,epstopdf}
\usepackage{empheq} 
\usepackage{varwidth}
\usepackage{todonotes}
\usepackage[normalem]{ulem}

\usepackage{hyperref}

\newcommand{\R}{\mathbb{R}}

\newcommand{\T}{\tau}
\newtheorem{theorem}{Theorem}[section]
\newtheorem{proposition}[theorem]{Proposition}
\newtheorem{defn}{Definition}[section]

\newtheorem{rmk}{Remark}[section]

\definecolor{darkgreen}{rgb}{0, 0.5, 0}
\newcommand{\sw}{\color{blue}}

\begin{document}

\pagestyle{myheadings}

\title{
\vspace*{-20mm}
Rate-Induced Tipping: Thresholds, Edge States and Connecting Orbits
}

\author{
Sebastian Wieczorek\thanks{School of Mathematical Sciences, University College Cork, Cork, T12 XF62, Ireland}~\thanks{Joint lead authors},
Chun Xie${}^*$
and
Peter Ashwin\thanks{Centre for Systems, Dynamics and Control, Department of Mathematics and Statistics, University of Exeter, Exeter EX4 4QF, UK}~${}^{\dag}$
}

\date{\today}

\maketitle

\vspace*{0mm}

\section*{Abstract}

Rate-induced tipping (R-tipping) occurs when time-variation of input parameters of a dynamical system interacts with system timescales to give genuine nonautonomous instabilities. Such instabilities appear as the input varies at some {\em critical rates} and cannot, in general, be understood in terms of autonomous bifurcations in the frozen system with a fixed-in-time input. 

This paper develops an accessible mathematical framework for R-tipping in multidimensional nonautonomous dynamical systems with an autonomous future limit. We focus on R-tipping via loss of tracking of base attractors that are equilibria in the frozen system, due to crossing what we call  {\em regular R-tipping thresholds}. These thresholds are anchored at infinity by {\em regular R-tipping edge states}:  compact normally hyperbolic invariant sets of the autonomous future limit system that have one unstable direction, orientable stable manifold, and lie on a basin boundary. We define R-tipping and critical rates for the nonautonomous system in terms of special solutions that limit to a compact invariant set of the autonomous future limit system that is not an attractor. We focus on the case when the limit set is a regular edge state, introduce the concept of {\em edge tails}, and rigorously classify R-tipping into reversible, irreversible, and degenerate cases.
The central idea is to use the autonomous dynamics of the future limit system to analyse R-tipping in the nonautonomous system. We compactify the original nonautonomous system to include the limiting autonomous dynamics. Considering regular R-tipping edge states that are equilibria allows us to prove two results. First, we give sufficient conditions for the occurrence of R-tipping in terms of easily testable properties of the frozen system and input variation. Second, we give  necessary and sufficient conditions for the occurrence of reversible and irreversible R-tipping in terms of computationally verifiable (heteroclinic) connections to regular R-tipping edge states in the autonomous compactified system. 


\vspace{10mm}

\noindent
Mathematics Subject Classification Numbers:\\ 
34C37, 34C45, 34E15, 37B55, 37C29, 37C60, 37M20.

~

\noindent
Email: sebastian.wieczorek@ucc.ie

\tableofcontents

\newpage

\section{Introduction}

Instability in the evolution of an open system subject to time-varying external conditions is a vitally important problem in many areas of applied science, including climate, ecology and biology. In particular, ``tipping points'' or ``critical transitions'' are {\em large, sudden} and often {\em irreversible} changes in the state of the system in response to {\em  small and slow} changes in the external conditions. 
Consider an open system near a stable state
(the base attractor). We might expect 
that, as external conditions change with time, the stable state will change too. We describe this phenomenon as a {\em moving stable state}.
Furthermore, we expect that the boundary of the basin of attraction of this stable state will change too. In many cases the system may adapt to changing external conditions and {\em track} the moving stable state.  However, tracking may not
always be possible. Nonlinearities, competing timescales and feedbacks in the system mean that the stable state may turn unstable or disappear.  Alternatively, the system may 
cross the moving basin boundary and evolve away from the moving stable state.
When this happens, the system {\em tips} to a different state. The different state may be long-lived (another attractor in a multi-stable system) or short-lived (a transient response in an excitable system).

Our focus is on an interesting and relatively new tipping phenomenon,
in which the open system fails to track a moving stable state as external conditions 
vary at some critical rate(s).
Finding these critical rates and characterising what happens when they are exceeded is of great interest in the natural sciences.
From a mathematical viewpoint, 
such tipping corresponds to a {\em genuine nonautonomous instability} in the corresponding 
nonautonomous  dynamical system
with time-varying {input parameters, also referred to as} external inputs.
The two main obstacles to mathematical analysis of such tipping are: (a) inability to explain it in terms of a classical autonomous bifurcation of the  stable state in the {\em frozen system} with fixed-in-time external inputs, and (b) the absence of compact  stable states such as equilibria, limit cycles or tori in the nonautonomous system.
Thus,  it 
requires development of mathematical techniques beyond classical autonomous bifurcation theory~\cite{Kuznetsov2004}. Existing approaches include, for example, identifying  a ``safe region'' about the moving stable state~\cite{Ashwin2012,Bishnani2003,Osinga2014}, using geometric singular perturbations~\cite{Mitry2013,OSullivan2021,Perryman2014,Vanselow2019,Wieczorek2011},
finite-time Lyapunov exponents~\cite{Duc2016,HoyerLeitzel2018,KloedenRasmussen2011,Meyeretal2018},
local pullback attractors~\cite{Alkhayuon2018,Arnold1998,Ashwin2016,KloedenRasmussen2011,Kuehn2021,Longo2021,Potzsche2010,Rasmussen2007} or snapshot attractors~\cite{Drotos2015,Kaszas2019}, Melnikov-like methods~\cite{Kuehn2021},
as well as most likely tipping paths~\cite{Chen2019,Ritchie2016} and tipping probabilities~\cite{Hartl2019} in the presence of noise.

This work overcomes obstacles  (a) and (b) as follows. We relate the actual state of the nonautonomous system to the moving stable state
to develop an accessible mathematical framework for such tipping phenomena. 
Within this framework  we use 
the compactification technique developed in~\cite{Wieczorek2019compact} in combination with geometric singular perturbation theory~\cite{Fenichel1979,Jones1995,Szmolyan2004,Wechselberger2011} to give rigorous results that are both easily verifiable and  relevant for a wide range of applications.
Most importantly,  
we extend a number of key rigorous results from~\cite{Ashwin2016} for  irreversible R-tipping in one-dimensional (scalar) systems to arbitrary dimension and to different cases of R-tipping, including reversible R-tipping that can occur only in higher dimensions. 
Our approach is guided by illustrative numerical examples of different cases of R-tipping in higher dimension, 
which are given in~\cite{Xie2019}.

\subsection{Motivation: Critical Factors and R-tipping}

In applications, it is important to determine {\it critical factors} for tipping~\cite{Ashwin2012}.  The most commonly studied critical factor is a {\it critical level} of the external input at
which the moving stable state of a complex system disappears or
destabilises in a classical dangerous\,\footnote{Dangerous bifurcations have a discontinuity in the parametrised family (or branch) of  attractors at the bifurcation point and include, for example, saddle-node and subcritical Hopf bifurcations~\cite{Thompson1994}.}
bifurcation, causing the system to
suddenly move to a different
state~\cite{Ashwin2016,Kuehn2011,Ott,Ritchie2021,Thompson_Sieber_2010b}.  
Critical levels
have been identified in many different contexts: the collapse of
thermohaline circulation past the critical level of fresh-water influx
into the North Atlantic~\cite{Alkhayuon2019,Dijkstra2008,Lohmann2021}
loss of submerged
vegetation in shallow turbid lakes past the critical level of nutrient
concentration~\cite[Ch.7]{Scheffer2009book}, forest-to-desert
transitions below the critical level of
precipitation~\cite[Ch.11]{Scheffer2009book}, power outage blackouts
past the critical level of power consumption~\cite{Budd2002,Dobson1989},
and in the reports of the {\em Intergovernmental Panel for Climate Change}~\cite{IPCC}
which specify critical levels of atmospheric temperature and CO$_2$
concentration.  The underlying dynamical mechanism is
illustrated in a simple example in Figure~\ref{fig:BR}(a). As the
external input changes in time, the position of the stable state
changes too. The nonautonomous system can track the moving
stable state as long as it persists, provided that the 
external input  varies slowly enough. However, there may be a critical level of the external input at which the
moving stable state disappears or destabilises in a classical
bifurcation~\cite[Lemma 2.3]{Ashwin2016}.  
If the bifurcation is dangerous, there is no nearby stable state to track
beyond the critical level, and the system suddenly moves to
a different state. Note that the critical transition in Figure~\ref{fig:BR}(a):
\begin{itemize}
\item
Requires a critical level of the external input -- a classical dangerous bifurcation of the stable state in the 
{\em frozen system} with fixed-in-time external inputs~\cite{Kuehn2011,OKeeffe2019,Thompson_Sieber_2010b}.
\item Occurs no matter how slowly the external input
passes through the critical level.
\end{itemize}
This nonautonomous instability has been described as a {\em dynamic bifurcation}~\cite{Benoit1990}, {\em adiabatic bifurcation}~\cite{KloedenRasmussen2011} or {\em bifurcation-induced tipping (B-tipping)}~\cite{Ashwin2012}. The key point is that it can be understood in terms of a classical autonomous bifurcation of the moving stable state. In the presence of noise, there may be early warning signals of the impending bifurcation, and there has been much progress in understanding when such signals may be present \cite{boers2021observation,bury2020detecting,Dakos_etal_2015,Dakos_etal_2008,Ditlevsen_etal_2010,Schefferetal2009}.
%
\begin{figure}[t!]
 \includegraphics[width=15.5cm]{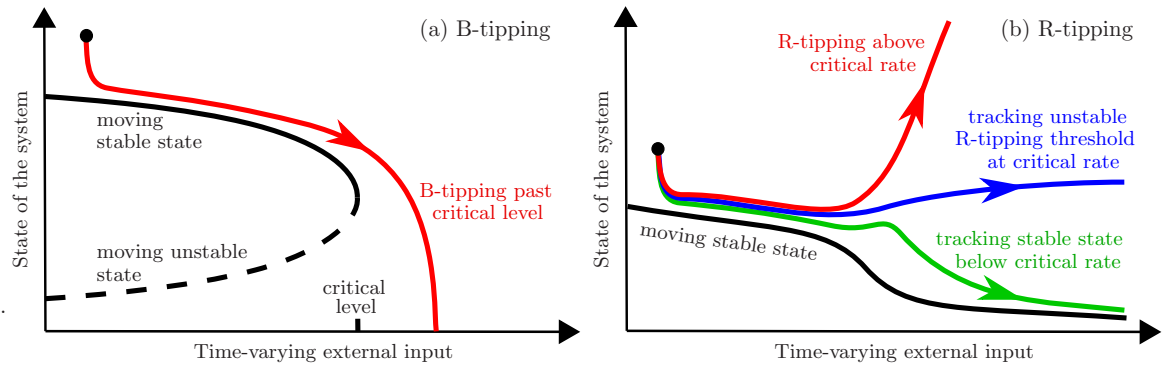}
  \caption{ The conceptual difference between (a) B-tipping and (b) R-tipping for monotonically changing external inputs. The (solid black) moving stable state is a stable state of the frozen system for different values of a fixed-in-time external input. The (colour) trajectories show the system behaviour for a time-varying external input. 
  (a) In B-tipping, there is a {\it critical level} of the external input, and tipping occurs for any rate of passage through the critical level. 
  (b) In R-tipping, there is no critical level, but there is a {\it critical rate} of change of the external input above which the system fails to track the moving stable state and tips. The (blue) special critical-rate trajectory tracks what we define as a repelling R-tipping threshold.
    }
\label{fig:BR}
\end{figure}
%

However, critical levels  of the external input are not the only critical factor for sudden transitions. 
Other factors may arise in a system that is given insufficient time to adapt~\cite{Scheffer2008,stocker1997,Wieczorek2011}, that is subjected to fast fluctuations (noise)~\cite{Ashwin2012,Ditlevsen_etal_2010}, that is close to basin boundary 
and may spend long period of time near (unstable) states of saddle type~\cite{Bezekci2015,Hastings2018,Kuehn2011}, or is sensitive to the  spatial extent, spatial location or spatial change of the external input~\cite{Berestycki2009, hasan2022,Idris2008,Rushton1937,Starmer2007}. 
What is more, real-world tipping phenomena may involve an interplay between 
different critical factors~\cite{OKeeffe2019,Ritchie2021,slyman2022}.

The focus of this work is on systems that are particularly sensitive
to how fast the external input changes~\cite{Jones2021,ritchie2022}. Such systems may not even have
any critical levels of the external input, but they may have {\it critical rates} of change of the external input:
they suddenly and unexpectedly move to a different state if the
external input changes  faster than some critical rate. Although critical rates are less understood than critical levels, they are equally relevant and ubiquitous~\cite{
Alkhayuon2018,
arnscheidt2022,
Arumugam2021,
Ashwin2012,
Berestycki2009,
Clarke2021,
Kaur2022,
Kiers2018,
Lohmann2021,
Merker2021,
Mitry2013,
Morris2002,
OKeeffe2019,
OSullivan2021,
Pierini2021,
Scheffer2008,
Siteur2016,
Suchithra2020,
van2021,
Vanselow2022evo,
vanselow2022,
Vanselow2019,
Wieczorek2011}.
In particular, critical rates
are of special interest in climate science and ecology in the contexts of {\em global warming}, increasing {\em climate variability}, and ensuing
 {\em failure to adapt to changing external conditions}: the moving stable state is continuously available, but the system is unable to adjust to its changing position when the change happens slowly but too fast. This is evidenced by
reports of contemporary and projected climate variability being too
fast for animals and plants to migrate or
adapt~\cite{Berestycki2009,Jezkova2016,Leemans2004,van2021}, critical dependence of
thermohaline circulation on the rate of North-Atlantic fresh-water
influx~\cite{Alkhayuon2018,Lohmann2021,Lucarini2005} and the rate of CO$_2$ emissions~\cite{stocker1997}, sudden release of soil carbon from
peatlands into the atmosphere~\cite{Clarke2021,Luke2011,Wieczorek2011} that can be accompanied by ``zombie wildfires"~\cite{OSullivan2021,Scholten2021} above some critical rate of atmospheric
warming, climate-related ``critical-rate
hypothesis'' in the context of coastal wetlands responding to rising
sea level~\cite{Morris2002} and more generally ecosystems subject to
rapid changes in external conditions such as wet El Ni\~no Southern
Oscillation years, droughts, or disease outbreaks~\cite{OKeeffe2019,Scheffer2008}.

There are many other areas of science where critical
rates are important.
In neuroscience, in addition to type-I or II nerves which ``fire''
above some level of externally applied voltage, there are type-III
excitable nerves that are able to accommodate slow changes in an
externally applied voltage up to very high voltage levels. What is
necessary for
type-III nerves to ``fire'' is a fast enough increase in an externally
applied voltage, rather than a high enough voltage level, and this
rate-dependence allows the brain for accurate coincidence
detection~\cite{Hill1935,Hodgkin1948,Idris2007,Mitry2013,Rubin2021}.
In competitive economy, there is a related ``chasing problem'' in the
context of supply, demand and prices trying to adapt to a changing
equilibrium~\cite{Sheng-YiHsu2014}.

The general concept of rate-induced tipping is illustrated in
Figure~\ref{fig:BR}(b).  When the external input changes in time, the
nonautonomous system tries to track the moving stable state. Tracking
is guaranteed if the external input changes slowly enough~\cite[Lemma2.3]{Ashwin2016}.  However, above some critical rate of the external input change,
the system can no longer track the moving stable state and may
suddenly move to a different state. Note that the critical transition in
Figure~\ref{fig:BR}(b):
\begin{itemize}
\item Does not require any critical level of the external input -- there need not be any classical bifurcation of the stable state in the frozen system with fixed-in-time external inputs.
\item  Occurs only if the external input varies  faster than some critical rate.
\item Can be {\em irreversible}: the system fails to track the
moving stable state, suddenly moves to a different stable state\footnote{Often referred to as an ``alternative stable state".},
and never returns to the original stable state; see for
example~\cite{Kiers2018,OKeeffe2019,Scheffer2008}.
\item Can be {\em reversible}: the system fails to track the moving 
stable state, makes a large excursion away from it, then returns to the original stable state, and this process may happen repeatedly;
see for example~\cite{Mitry2013,OSullivan2021,Perryman2014,Vanselow2019,Wieczorek2011}.
\end{itemize}
We describe such a genuine nonautonomous instability as a {\em  rate-induced tipping} or simply {\em R-tipping}~\cite{Ashwin2012}.
By ``genuine nonautonomous" we mean that, unlike B-tipping, R-tipping cannot, in general, be understood in terms of a classical autonomous bifurcation of a moving stable state. Nonetheless, in the presence of noise, some of the early-warning signals identified for B-tipping may also occur for R-tipping~\cite{Ritchie2016}.

We highlight that R-tipping is somewhat counter-intuitive and difficult to analyse for a number of reasons.
In addition to the fact that R-tipping cannot be simply explained in terms of a classical bifurcation of the stable state in a 
frozen system~\cite{OKeeffe2019}, R-tipping may occur for external input rates that are much slower than the  rate of convergence towards the stable state in the frozen system~\cite{Vanselow2019}. 
The reason is that tracking requires the
convergence rate towards the moving stable state to be faster than the
speed of the moving stable state in the phase space.
Thus, if the position of the stable state in the phase space is sufficiently sensitive to
changes in the  input parameters, then R-tipping may occur for external inputs
varying more slowly than the convergence rate towards the stable
state~\cite{Ashwin2012,Ashwin2012cor}.
Moreover, there may be no obvious R-tipping threshold separating initial
conditions that track the moving stable state from those that
R-tip. The separatrix in the nonautonomous system can be an intricate fractal or a non-obvious quasithreshold~\cite{Mitry2013,OSullivan2021,Perryman2015,Wieczorek2011,Xie2019}.
Lastly, reversible R-tipping poses a mathematical challenge to capture transient 
and thus quantitative instabilities because the system exhibits the same asymptotic (long-term) 
behaviour below and above a critical rate.
This has previously made  reversible R-tipping difficult to define rigorously, 
even using modern concepts from the theory of nonautonomous dynamical systems~\cite{Ashwin2016}. These counter-intuitive properties of R-tipping  further motivate and highlight the need for the development of a mathematical framework that is easily accessible to applied scientists.

\subsection{Summary of Main Results and Outline of Paper}

This paper develops an applicable theory of R-tipping via loss of end-point tracking of a moving sink in multidimensional systems for external inputs that vary smoothly with time and decay exponentially to a constant at
infinity. The theory allows us to extend rigorous criteria from~\cite{Ashwin2016} for irreversible R-tipping in one-dimensional (scalar) systems to arbitrary dimension and to different cases of such R-tipping.

The paper is organized as follows. 
Section~\ref{sec:problemsetting} introduces multidimensional {\em nonautonomous systems} with asymptotically constant {\em input parameters}. Additionally, it introduces a {\em rate parameter} $r>0$ that characterises the `rate' of time variation of the input parameter(s) along some input parameter path. 
Section~\ref{sec:NonautonInstab} defines {\em moving sinks} on a time interval $I$, which are hyperbolic sinks of the frozen system parametrised by time for a given time variation of the input parameter(s). Then, it characterises R-tipping from  base attractors that are hyperbolic sinks as failure of the nonautonomous system to track a moving sink via: (i) loss of end-point tracking and (ii) loss of $\delta$-close tracking.
As a starting point for analysis of R-tipping via los of end-point tracking, 
Section~\ref{sec:ThresholdsEdgeFrozen} develops a theory of {\em regular thresholds} and {\em regular edge states} within the frozen system, and defines {\em moving regular thresholds} and {\em moving regular edge states} in a similar way to moving sinks. Roughly speaking, regular edge states are compact  normally hyperbolic invariant sets with orientable codimension-one stable manifolds (one unstable direction), and regular thresholds are forward invariant subsets of stable manifolds of regular edge states. Crucially, Section~\ref{sec:ThresholdsEdgeFrozen} introduces the important and easily verifiable property of {\em (forward) threshold instability} of a (moving) sink. 
Section~\ref{sec:Rtippingdefs} gives a precise definition 
of R-tipping via loss of end-point tracking in multidimensional nonautonomous systems with asymptotically constant inputs in terms of special solutions that limit to a compact invariant set of the future limit system that is not an attractor.
It identifies {\em R-tipping thresholds}  that are typically responsible for loss of end-point tracking  and separate nonautonomous solutions that R-tip from those that do not in such systems.
Additionally, it defines {\em regular R-tipping edge states} and their {\em edge tails}. Regular R-tipping edge states are examples of non-attracting limit sets that anchor the important  {\em regular R-tipping thresholds} at infinity. The new concept of edge tails allows us to classify R-tipping via loss of end-point tracking into {\em reversible, irreversible and degenerate cases} by focusing on limit sets that are regular R-tipping edge states.
Section~\ref{sec:compact} introduces and summarises results from~\cite{Wieczorek2019compact} on compactification of nonautonomous dynamical systems with exponentially asymptotically constant external inputs. It includes the following key technical results. 
Proposition~\ref{prop:invsete-} relates a local pullback attractor anchored at negative infinity by a hyperbolic sink  to an invariant unstable manifold of a hyperbolic saddle in the compactified system,  a regular R-tipping threshold to an invariant stable manifold of the corresponding regular R-tipping edge state in the compactified system, and each edge tail to one branch of the invariant unstable manifold of the regular R-tipping edge state in the compactified system. Proposition~\ref{prop:edgetails} uses these relations to characterise R-tipping in terms of edge tail behaviour.

The main results of the paper are presented in Section~\ref{sec:gentestcrit} for moving sinks and  regular R-tipping edge states that are hyperbolic equilibria. Theorem~\ref{thm:tracking}  shows that nonautonomous solutions track moving sinks of the frozen system, while Theorem~\ref{thm:trackingthresholds} shows that regular
R-tipping thresholds track moving regular thresholds of the frozen system,
as long as the rate parameter $r$ is sufficiently small.
For moving sinks on $I=\R$, Theorem~\ref{thm:Rtip} gives criteria for the existence of R-tipping via loss of end-point tracking  in the nonautonomous system in terms of: (i) threshold instability of a hyperbolic sink in the frozen system  on a given parameter path, and (ii) forward threshold instability of a moving sink of the frozen system for a given time-varying external input. This theorem generalizes results from~\cite{Ashwin2016} for one-dimensional (scalar) systems in the sense that threshold stability does not guarantee tracking in higher-dimensional systems; see for example~\cite{Kiers2018,Xie2019}.
We finish this section by identifying different cases of R-tipping via loss of end-point tracking in the nonautonomous system with a connecting (heteroclinic) orbit in the autonomous compactified system. 
In particular, Proposition~\ref{prop:rtip_compact} identifies
(non-degenerate) reversible and irreversible R-tipping in the nonautonomous system with presence of a non-degenerate connecting (heteroclinic) orbit\footnote{We give non-degeneracy conditions for connecting orbits in Remark~\ref{rmk:nondeghet}.}
to a regular R-tipping edge state in the compactified system. This means that powerful numerical tools for detection and parameter continuation of connecting (heteroclinic) orbits can be applied to practically find critical rates for R-tipping via loss of end-point tracking. 

Finally, Section~\ref{sec:Conclusions} highlights some open questions associated with extending our results to more general settings. These settings include asymptotically constant external inputs that decay slower than exponentially or are not asymptotically constant, R-tipping from more complicated base attractors, involving more complicated R-tipping edge states,  thresholds that are not regular,
quasithresholds that are typically responsible for R-tipping via loss of $\delta$-close tracking, and R-tipping in spatially extended systems modelled by partial differential equations. This paper is complementary to~\cite{Wieczorek2019compact}  which develops the theory of compactification for asymptotically autonomous dynamical systems, and to~\cite{Xie2019} which presents a number of illustrative numerical examples of R-tipping.

\section{The Problem Setting}
\label{sec:problemsetting}

We consider a nonlinear {\em nonautonomous system}
\begin{equation}
\label{eq:ode}
\dot{x}=f(x,\Lambda(t)),
\end{equation} 
with the state variable $x\in\R^n$, time $t\in \R$, $C^1$-smooth time-varying external input 
$\Lambda:\R\to\R^d$, 
and $C^1$-smooth vector field $f:\R^n\times\R^d\to\R^n$,
where $\dot{x}$ denotes ${\textrm d}x/{\textrm d}t$.

\subsection{Parametrised Nonautonomous System: Rates of Change}

We are interested in understanding nonautonomous R-tipping instabilities that appear on varying the time scale or ``rate of change'' of an external input. To address this question, we extend~\eqref{eq:ode} to a {\em parametrised family of nonautonomous systems}
\begin{equation}
\dot{x}=f(x,\Lambda(rt)),
\label{eq:odewithr}
\end{equation}
where $r >0$ is a constant {\em rate parameter}~\cite{Alkhayuon2018,Ashwin2016,Ashwin2012,Wieczorek2011}. We refer to
$t$ as the {\em time scale of the system}, and to $\T=rt$
as the {\em time scale of the external input}.\,\footnote{Note that if $t$ is in units second and $r$ is in units inverse second then $\tau$ is dimensionless.} It is important to note that, typically, both the external input and solutions of \eqref{eq:odewithr} depend on $r$. Therefore, it will be convenient to analyse R-tipping on the time scale $\T$ of the external input, 
where only solutions to the problem depend on $r$.
To this end, we rewrite~\eqref{eq:odewithr} in terms of $\T$, and consider
\begin{equation}
x' = f(x,\Lambda(\tau))/r,
\label{eq:odewithrs}
\end{equation}
where $x'$ denotes ${\textrm d}x/{\textrm d}\T$.

\subsection{Frozen System}

Although R-tipping is a genuine nonautonomous instability
of the nonautonomous system, much can be
understood about R-tipping 
from properties 
of the autonomous {\em frozen system}
\begin{equation}
\label{eq:odea}
\dot{x}=f(x,\lambda),
\end{equation}
with a fixed-in-time {\em input parameter} $\lambda$ corresponding to a possible value of the external input~\cite{Ashwin2016}. 
The frozen system is sometimes called a {\em quasistatic system} or an {\em instantaneous system}. 
We will be interested in  families of equilibria 
of the frozen system \eqref{eq:odea} that vary $C^1$-smoothly with $\lambda$,
which are also referred to as {\em branches of equilibria}.
 Note that, for fixed $r>0$, one can write \eqref{eq:odea} in the time scale of the external input, namely
\begin{equation}
\label{eq:odears}
x'=f(x,\lambda)/r,
\end{equation}
and that \eqref{eq:odea} and \eqref{eq:odears} clearly have the same invariant sets, qualitative stability and  bifurcations on varying $\lambda$.

\subsection{Asymptotically Constant Inputs: Future and Past Limit Systems}
\label{sec:extinp}

\begin{figure}[t]
  \begin{center}
    \includegraphics[width=15.cm]{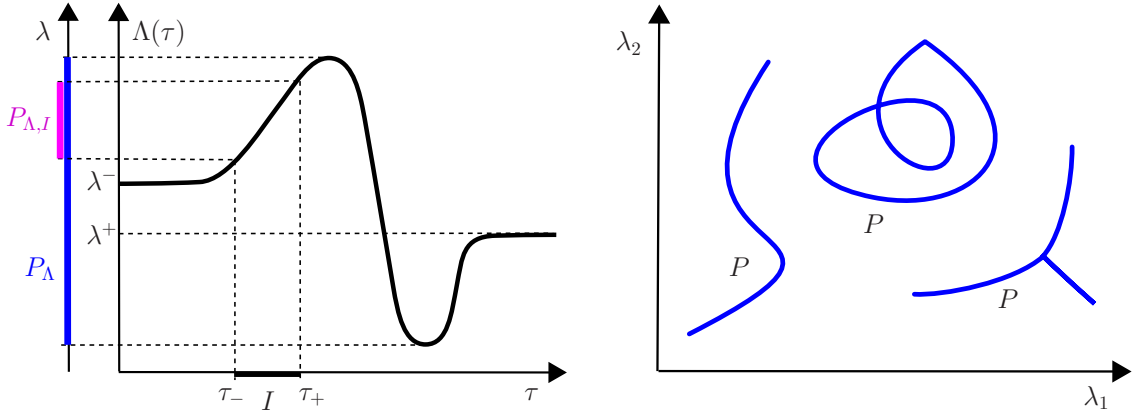}
  \end{center}
  \vspace{-3mm}
  \caption{
  (a) Example of a bi-asymptotically constant (scalar) external input $\Lambda(\T)$  
  with the future limit $\lambda^+$ and the past limit $\lambda^-$, together with two parameter paths: 
   (blue) parameter path $P_\Lambda\subset\R$ traced out by this $\Lambda(\T)$, and (purple)
  parameter path $P_{\Lambda,I}\subset P_\Lambda$
  traced out by this
  $\Lambda(\T)$ on a given time interval $I=(\T_-,\T_+)$. Note that $\lambda^+$ and $\lambda^-$
  do not lie on the boundary of $P_\Lambda$.
  (b) Examples of a parameter path $P$ in $\mathbb{R}^d=\mathbb{R}^2$.
  }
  \label{fig:path}
\end{figure}

When developing a theory of R-tipping, one needs to specify a class of possible external 
inputs $\Lambda(\tau)$. For arbitrary time-dependent inputs, the theory of nonautonomous 
systems \cite{KloedenRasmussen2011} summarises work in this area and gives general results 
on attraction and stability. Here, we focus on a case that is more specific, relevant to applications, and allows us to make further progress on the nonautonomous problem~\eqref{eq:odewithrs}. In particular, it allows us to extend results from~\cite{Ashwin2016} to arbitrary dimension and to different cases of R-tipping. 
To be more precise, we consider 
response of an open system to non-periodic external inputs that 
limit to a constant as time tends to positive and possibly negative infinity:
\begin{defn}
  \label{defn:ac}
We say that $\Lambda(\T)$ is {\em bi-asymptotically constant} with future limit $\lambda^+$ and past limit $\lambda^-$ if 
\begin{equation}
\label{eq:ac1p}
\lim_{\T\to +\infty} \Lambda(\T)= \lambda^+\in\R^d\;\;\mbox{and}\;\; \lim_{\T\to -\infty} \Lambda(\T)= \lambda^-\in\R^d.
\end{equation}
We say $\Lambda(\T)$ is {\em asymptotically constant} if 
one of the limits above exists. 
\end{defn}
\begin{rmk}
  \label{rmk:ac}
A bi-asymptotically constant $\Lambda(\T)$ need not be monotone or indeed one-dimensional,  which is a generalization of the parameter shifts considered in~\cite{Ashwin2016}. For example, for a scalar $\Lambda(\T)$, we do not require the supremum or infimum of $\Lambda(\T)$ to be $\lambda^+$ or $\lambda^-$; see Fig.~\ref{fig:path}(a).
\end{rmk}
Such inputs are used widely in different areas of applied science as mathematical models of finite-time disturbances, saturated growth processes and decay phenomena.  Furthermore, they are a natural choice for defining and analysing R-tipping rigorously: they allow us to identify possible asymptotic states of the system when the disturbance is gone, and discuss changes in the asymptotic state for different rates $r$ of the input.
The main simplification is that  nonautonomous problem~(\ref{eq:odewithrs}) becomes {\em asymptotically autonomous} 
in the terminology of~\cite{Aulbach2006,KloedenRasmussen2011,Markus1956,oljavca2022measure}:
$$
f(x,\Lambda(\T)) \to f(x,\lambda^\pm)\;\;\mbox{as}\;\; \T\to \pm\infty.
$$ 
For the case of bi-asymptotically constant $\Lambda(\T)$ we can define the autonomous {\em future limit system}
\begin{equation}
  \label{eq:odea+}
  \dot{x}=f(x,\lambda^+),
\end{equation} 
and the autonomous {\em past limit system}
\begin{equation}
  \label{eq:odea-}
  \dot{x}=f(x,\lambda^-),
\end{equation} 
which are special cases of the autonomous frozen system~\eqref{eq:odea}.

One of the main contributions of this work is to use autonomous dynamics and compact invariant sets (in particular equilibria and invariant manifolds) of the limit systems~\eqref{eq:odea+} and~\eqref{eq:odea-}
to analyse nonautonomous R-tipping instabilities in system~(\ref{eq:odewithr}) or~(\ref{eq:odewithrs}). 
While  related questions have been  investigated in the past~\cite{Castillo1994,Holmes_Stuart1992,Markus1956,Robinson1996,Thieme1994}, a particular  novelty of our approach is that we relate trajectories of the nonautonomous system~\eqref{eq:odewithrs} and compact invariant sets of the autonomous limit systems~\eqref{eq:odea+} and~\eqref{eq:odea-} to one
autonomous compactified system.
This can be achieved by applying the compactification technique that was developed in~\cite{Wieczorek2019compact} for system~\eqref{eq:ode}  with arbitrary decay of external inputs $\Lambda(t)$. The technique is reviewed in Sec.~\ref{sec:compact} 
from the viewpoint of R-tipping in system~\eqref{eq:odewithrs} and exponentially decaying external inputs $\Lambda(\T)$.

\subsection{Solutions and Trajectories of the Parametrised Nonautonomous System
}
\label{sec:notation}

Throughout the paper, dependence of solutions and trajectories of the nonautonomous system~\eqref{eq:odewithrs} on $r$ is indicated by superscript $[r]$.
For example, we write
$$
 x^{[r]}(\T,x_0,\T_0)\in\mathbb{R}^n,
$$
to denote a solution\,\footnote{This is the flow $x(\T)=\varphi(\T,\T_0,x_0)$ 
written as a process \cite{KloedenRasmussen2011} with the $r$ dependence explicitly shown.
Given a solution $x^{[r]}(\T,x_0,\T_0)$ to system~\eqref{eq:odewithrs}, 
one can easily obtain the corresponding solution to system~\eqref{eq:odewithr} 
by setting $t=\tau/r$ and $t_0=\tau_0/r$. However, it is important to note that, 
for different $r>0$, a fixed initial state $(x_0,\T_0)$ in 
system~\eqref{eq:odewithrs} corresponds to a fixed value of the external 
input $\Lambda(rt_0)$, but different initial states $(x_0,t_0)=(x_0,\T_0/r)$
in system~\eqref{eq:odewithr}.
}
to system~\eqref{eq:odewithrs} at time $\T$ started from $x_0$ at initial time $\T_0$ for a fixed rate $r$. We also write
$$
\mbox{trj}^{[r]}(x_0,\T_0) = \left\{ 
x^{[r]}(\T,x_0,\T_0)\,:\, \T \ge \T_0
\right\} \subset \mathbb{R}^n,
$$
to denote the corresponding trajectory from $(x_0,\T_0)$.
For bi-asymptotically constant inputs $\Lambda(\T)$,
if $e^-$ is a sink for the  autonomous past limit system~\eqref{eq:odea-} and $x^{[r]}(\T,x_0,\T_0)\to e^-$ as $\T\to -\infty$, we write this solution as
\begin{equation}
x^{[r]}(\T,e^-)\in \mathbb{R}^n.\nonumber
\end{equation}
We also write the corresponding trajectory as
$$
\mbox{trj}^{[r]}(e^-) 
= \left\{ 
x^{[r]}(\T,e^-)\,:\, \T \in \R
\right\} \subset \mathbb{R}^n.
$$
If the sink $e^-$ is hyperbolic 
then one can show~\cite{Alkhayuon2018,Ashwin2016} that $x^{[r]}(\T,e^-)$ is unique and can be understood as a {\em local pullback  attractor} for the nonautonomous system~\eqref{eq:odewithrs}.
We sometimes simply write 
\begin{equation}
x^{[r]}(\T)\in \mathbb{R}^n,\nonumber
\end{equation}
to mean either $x^{[r]}(\T,x_0,\T_0)$ or $x^{[r]}(\T,e^-)$, 
and
$$
\mbox{trj}^{[r]} \subset \mathbb{R}^n,
$$
to mean either $\mbox{trj}^{[r]}(x_0,\T_0) $ or $\mbox{trj}^{[r]}(e^-)$,
depending on the context.
Note that solutions $x^{[r]}(\T)$  and trajectories $\mbox{trj}^{[r]}$ started 
from the same initial state $(x_0,\T_0)$, or limiting to the same sink $e^-$, will typically vary nontrivially with $r>0$.

\subsection{Parameter Paths}

To give easily testable criteria for R-tipping, it is convenient to work with a parameter path that is traced out by the external input $\Lambda(\T)$ in the parameter space $\R^d$.  We write $\overline{S}$ to denote the closure\,\footnote{The smallest closed subset of $\R^d$ containing $S$.} of $S$, and define:
\begin{defn}
\label{def:path}
A {\em parameter path}  is a compact subset of the input parameter space $\R^d$, that is the closure of an image of a $C^1$-smooth function from $\R$ to $\R^d$. 
\begin{itemize}
    \item[(a)]
    A given parameter path is denoted by $P$.
     \item[(b)]
     {\em A parameter path traced out by a given external input} $\Lambda(\T)$ is denoted by
    \begin{equation}
    \label{eq:PL}
    P_\Lambda=\overline{\left\{\Lambda(\T)\,:\,\T\in\R \right\}}\subset\R^d.
    \end{equation}
    \item[(c)]
     {\em A parameter path traced out by a given external input $\Lambda(\T)$ on a given time interval} $I=(\T_-,\T_+)$, where $\T_\pm$ may be $\pm\infty$, is denoted by
    \begin{equation}
    \label{eq:PLI}
    P_{\Lambda,I}=\overline{\left\{\Lambda(\T)\,:\,\T\in I \right\}}\subseteq P_\Lambda.
    \end{equation}
\end{itemize}
\end{defn}
\begin{rmk}
\label{rmk:traceout}
Note that 
$P$  can be traced out by (infinitely) many different external inputs,
 $P_{\Lambda,I}$  may be traced out by a given external input $\Lambda$ also on time intervals other than $I$,
$P_\Lambda$ and $P_{\Lambda,I}$ are independent of the rate $r > 0$ of the external input $\Lambda$, and $P_{\Lambda,\R} = P_\Lambda$.
\end{rmk}
Figure~\ref{fig:path} shows examples of (a) $P_\Lambda$
and $P_{\Lambda,I}$ in a one-dimensional parameter space, and (b) examples of $P$ in a two-dimensional 
parameter space~\cite{Alkhayuon2020,OKeeffe2019}. An external input $\Lambda(\T)$ may traverse the 
parameter path $P_\Lambda$ over time in a
complicated manner, for example by moving back and forth along the
path repeatedly, and with a varying speed
$\Vert \Lambda'(\tau)\Vert$, 
as shown in Fig.~\ref{fig:path}(a).
Moreover, the future limit
$\lambda^+$ and, if it exists, the past limit  
$\lambda^-$ of $\Lambda(\T)$ need not lie on the 
boundary of $P_\Lambda$; see also Remark~\ref{rmk:ac}.

\section{ Tracking and Failure to Track of Moving Sinks}
\label{sec:NonautonInstab}

In this section we explore the response of the nonautonomous system~(\ref{eq:odewithr}) or~(\ref{eq:odewithrs}) 
to external inputs $\Lambda$. 
First, we introduce the intuitive concept of a moving sink - 
a smooth family of  instantaneous positions of a hyperbolic sink for the autonomous frozen system~\eqref{eq:odea} that does not depend on the rate parameter $r>0$ when viewed on the external input time scale $\T$. Then, we discuss the relation between the moving sink and rate-dependent solutions $x^{[r]}(\T)$ to system~(\ref{eq:odewithrs}) 
for different but fixed $r>0$. A similar setting was used previously~\cite{Alkhayuon2018,Ashwin2016,Ashwin2012,Longo2021,OKeeffe2019,Perryman2015,Wieczorek2011} 
to understand the dynamical behaviour of~\eqref{eq:odewithrs} in terms of:
\begin{itemize}
    \item[$\bullet$] 
    Tracking of a  moving sink by $x^{[r]}(\T)$ for sufficiently small but non-zero rates $r$.
    \item[$\bullet$] 
    Failure to track a moving sink via a nonautonomous R-tipping instability that can appear at higher  rates $r = r_c$. This includes potential multiple transitions between tracking and tipping as $r$ is increased~\cite{Longo2021,OKeeffe2019,OSullivan2021}.
\end{itemize}

\subsection{Moving Sinks}
\label{sec:movingequ}

We consider a base attractor in the autonomous frozen system~(\ref{eq:odea}) that varies $C^1$-smoothly with $\lambda$. 
Our focus is on a linearly stable equilibrium (a {\em hyperbolic sink}) that continues and does not bifurcate along (some part of) a parameter path traced out by  a given external input $\lambda=\Lambda(\T)$.
We will be interested in how the position of such an equilibrium changes over time. 
\begin{defn}
\label{def:ms}
Suppose the autonomous frozen system~(\ref{eq:odea}) has an equilibrium $e(\lambda)$
for some connected set of values of $\lambda$.  
Consider an external input $\Lambda(\T)$ that traces out a parameter path $P_{\Lambda,I}$ on a time interval $I=(\tau_-,\tau_+)\subseteq\mathbb{R}$, where $\T_\pm$ can be $\pm\infty$.
Then,
\begin{itemize}
\item [(a)] 
We say $e\left(\Lambda(\T)\right)$ is a {\em moving sink on $I$} if $e(\lambda)$ is a hyperbolic sink that varies $C^1$-smoothly with 
$\lambda\in P_{\Lambda,I}$.
\item [(b)] 
If $\Lambda(\T)$ is asymptotically constant to $\lambda^+$ and $e\left(\Lambda(\T)\right)$ is a moving sink on $I = (\tau_-,+\infty)$, we define the {\em future limit} $e^{+}$
 of a moving sink 
$$
e^{+} = e(\lambda^{+}), 
$$
which is a hyperbolic sink for the autonomous future limit system~\eqref{eq:odea+}.\\
If $\Lambda(\T)$ is asymptotically constant to $\lambda^-$ and $e\left(\Lambda(\T)\right)$ is a moving sink on $I = (-\infty,\tau_+)$, we define the {\em past limit} $e^{-}$ of a moving sink 
$$
e^{-} = e(\lambda^{-}),
$$
which is a hyperbolic sink for the autonomous past limit system~\eqref{eq:odea-}.
\end{itemize}
\end{defn}
A moving equilibrium on a time interval $I$ is 
an equilibrium of the autonomous frozen system~(\ref{eq:odea})
parametrised by time $\T\in I$ for a given input $\lambda=\Lambda(\T)$. It is sometimes 
called a {\em quasistatic equilibrium} or an {\em instantaneous equilibrium}.
Guided by the intuition from Fig.~\ref{fig:BR}(b), we often
focus on the special case $I=\R$, namely where moving equilibria continue and do not bifurcate along the whole parameter path $P_\Lambda$ traced out by $\Lambda(\T)$.
Note that the
moving equilibrium $e\left(\Lambda(\T)\right)$  depends on $f$ and  on the 
shape of the external input $\Lambda$, but does not depend on the rate 
parameter $r>0$ (though its eigenvalues do depend on $r$ when viewed 
on the external input timescale $\T$; see Eq.~\eqref{eq:odears}). 
We consider moving equilibria in the phase space $\mathbb{R}^n$ of the nonautonomous system~(\ref{eq:odewithrs}), but note that they are
 not solutions to~(\ref{eq:odewithrs}). However, moving equilibria can serve as a useful point of reference for discussing rate-dependent 
solutions $x^{[r]}(\T)$ to~(\ref{eq:odewithrs}). For example, they can
approximate $x^{[r]}(\T)$  when $r$ is sufficiently small, as we see in Section~\ref{sec:trackingcriteria}.

\subsection{Tracking Moving Sinks}
\label{sec:trackingcriteria}

We will be interested in how  a rate-dependent solution  $x^{[r]}(\T)$ of~(\ref{eq:odewithrs}) changes over time relative to 
a moving sink $e(\Lambda(\T))$  for a given external input $\Lambda(\T)$ and different rates $r>0$.
As noted in~\cite{Alkhayuon2018,Ashwin2016}, there are several ways to 
understand tracking of a moving sink, depending on whether we need closeness at all points in time, or just in the future limit $\T\to +\infty$. The 
following definition formalises this.
\begin{defn} 
 \label{defn:Track}
Consider a nonautonomous system~(\ref{eq:odewithrs}) with an external input $\Lambda(\T)$. Suppose there is a moving  sink $e(\Lambda(\T))$ on $I=(\tau_-,\tau_+)$, where $\T_\pm$ may be $\pm\infty$. For any fixed $\delta>0$ and $r>0$:
\begin{itemize}
    \item[(a)] 
    We say $x^{[r]}(\T)$ {\em $\delta$-close tracks $e(\Lambda(\T))$ on $I$} if
    \begin{equation}
     \label{eq:deltatrack}
     \Vert x^{[r]}(\T) - e(\Lambda(\T))\Vert < \delta\;\;\mbox{for all}\;\;\T\in I.
   \end{equation}
    \item[(b)] Suppose in addition that $\Lambda(\T)$ is asymptotically constant to $\lambda^+$, $e(\Lambda(\T))$ is a moving sink on $I=(\tau_-,+\infty)$, and recall that $e(\Lambda(\T))$ limits to $e^{+}$. Then, we say $x^{[r]}(\T)$ {\em end-point tracks $e(\Lambda(\T))$ on $I$} if $x^{[r]}(\T)$ exists for all $\T\in I$ and
    \begin{equation}
	\label{eq:eptrack}
    x^{[r]}(\T) \to e^+\;\;\mbox{as}\;\; \T\to +\infty.
    \end{equation}
\end{itemize}
\end{defn}
\begin{rmk}
  \label{rmk:ept}
  We  define $\delta$-close tracking for $x^{[r]}(\T)$ on any time interval $I=(\tau_-,\tau_+)$, and end-point tracking for $x^{[r]}(\T)$ on a (semi)infinite time interval $I=(\tau_-,+\infty)$, where $\tau_\pm$ may be $\pm\infty$. This is a generalisation of the $\delta$-close and end-point tracking definitions used in~\cite{Ashwin2016}, which restrict to tracking by pullback attractors $x^{[r]}(\T) = x^{[r]}(e^-,\T)$ on $I=\mathbb{R}$.
\end{rmk}
Theorem~\ref{thm:tracking} gives criteria that sufficiently small rate
parameter $r$ (i.e. slow enough motion of hyperbolic sinks on the system
time scale $t$) will give $\delta$-close and end-point tracking for any $\delta>0$.
Tracking of more complicated attractors\,\footnote{See Appendix~\ref{sec:A4} for the definition of an attractor.} such as limit cycles~\cite{Alkhayuon2020,Alkhayuon2018}, tori and chaotic attractors~\cite{Alkhayuon2020weak,Kaszas2019} is  discussed in Section~\ref{sec:Conclusions} and left for future study.

\subsection{Failure to Track: Nonautonomous R-tipping Instability}
\label{sec:Rtipintro}

We use the notion of R-tipping to describe two types of genuine nonautonomous {\em instabilities} that occur through loss of tracking in the following manner:
\begin{itemize}
    \item[$\bullet$] 
    {\em Loss of end-point tracking:}
    A rate-dependent solution $x^{[r]}(\T)$ fails to end-point track a moving sink $e\left(\Lambda(\T)\right)$ at some 
    rate $r=r_c$~\cite{Ashwin2016,Kuehn2021,Longo2021,OKeeffe2019,Scheffer2008}. This 
    instability
    is a {\em qualitative change}, it can thus be classified as a genuine {\em nonautonomous  bifurcation}.
\item[$\bullet$] 
    {\em Loss of $\delta$-close tracking:}
    For a given $\delta>0$, a rate-dependent solution $x^{[r]}(\T)$ end-point tracks a moving sink $e\left(\Lambda(\T)\right)$ for all $r>0$, but  fails to $\delta$-close track $e\left(\Lambda(\T)\right)$ at some 
    rate $r=r_c(\delta)$ that depends on the choice of $\delta$~\cite{Mitry2013,OSullivan2021,Vanselow2019,Wieczorek2011,Xie2019}. This 
    instability
    is a {\em quantitative change}, but cases of interest may be classified as {\em finite-time bifurcations}~\cite{Rasmussen2010}.
\end{itemize}

This paper gives a rigorous characterisation of R-tipping  that occurs via qualitative ``loss of end-point tracking" in Definition~\ref{defn:Rtip}, and leaves quantitative ``loss of $\delta$-close tracking" for future research.
For example, suppose that $x^{[r]}(\T)\rightarrow e^+$ for $0<r<r_c$ but
$$
x^{[r_c]}(\T)\not\rightarrow e^+\;\;\mbox{as}\;\; \T\to +\infty.
$$
If $x^{[r_c]}(\T)$ remains bounded then  the system undergoes {\em R-tipping} according to our definition. If such an $r_c$ is isolated, we call it a {\em critical rate}.
One aim of this paper is to identify and rigorously define possible cases of such R-tipping.
In doing so, we note that
the critical-rate solution $x^{[r_c]}(\T)$ will typically 
converge to a compact invariant set\,\footnote{Notions of convergence to invariant sets $\eta$ are discussed in Appendix~\ref{sec:A1}.} $\eta^+$:
$$
x^{[r_c]}(\T)\to \eta^+\;\;\mbox{as}\;\; \T\to +\infty,
$$
that is not an attractor, not necessarily an equilibrium, and lies on the basin boundary of a sink $e^+$ in the future limit system~\eqref{eq:odea+}~\cite{OKeeffe2019,Xie2019}. If this set is hyperbolic with one unstable direction and an orientable stable manifold\,\footnote{Note that $\eta^+$ is contained in its stable manifold, that is $\eta^+\subseteq W^s(\eta^+)$.} then we call such an $\eta^+$ a {\em regular R-tipping edge state}. This in turn suggests other important notion:
{\em a regular R-tipping threshold} which contains initial states that converge to the regular R-tipping edge state in the nonautonomous system~\eqref{eq:odewithrs}.

\section{Thresholds and Edge States for Autonomous Frozen Systems}
\label{sec:ThresholdsEdgeFrozen}

We consider thresholds in phase space as invariant sets that have two different sides and, in some sense, give qualitatively  different behaviour for trajectories started on different sides of the threshold.
We introduce different types of threshold, as summarized below:
\begin{itemize}
    \item[$\bullet$] 
    For the autonomous frozen system~(\ref{eq:odea}),  we distinguish in Sec.~\ref{sec:rthr} between {\em regular thresholds} and {\em irregular thresholds}.
    \item[$\bullet$]
    Given a  regular threshold  that varies $C^1$-smoothly with $\lambda$, and a time-varying external input $\Lambda(\T)$, we define in Sec.~\ref{sec:thr} {\em moving regular thresholds} as regular thresholds of the autonomous frozen system~(\ref{eq:odea}) parametrised by time $\T$ for $\lambda=\Lambda(\T)$.
    \item[$\bullet$]
    For the nonautonomous system~\eqref{eq:odewithrs}, we define in Sec.~\ref{sec:Rtip_critrates} {\em regular R-tipping thresholds}. These are nonautonomous forward-invariant sets that separate solutions $x^{[r]}(\T)$ of~\eqref{eq:odewithrs} that R-tip from those that do not. 
\end{itemize}
Definition~\ref{defn:multi_excit} uses regular thresholds to generalise, and in certain sense unify, the concepts of  ``excitability thresholds" for excitable systems~\cite{FitzHugh1955,Izhikevich07,Krauskopf2003,Wieczorek2011} and  ``multi-basin boundaries" for multistable systems~\cite{Pisarchik2014}.

\subsection{Regular Thresholds, Regular Edge States and Excitability}
\label{sec:rthr}

We restrict to thresholds that are 
repelling orientable embedded  manifolds\,\footnote{We recall  some notions used in discussion of differentiable manifolds in Appendix~\ref{sec:A2}.}, which we call {\em regular thresholds}. Thresholds that are repelling 
but not orientable or not embedded manifolds such as the fractal basin boundaries discussed in~\cite{Aguirreetal2009,Kaszas2019,McDonald1985},
we term {\em irregular thresholds} and leave for future study. More precisely:
\begin{defn}
\label{defn:rthr}
In the $n$-dimensional autonomous frozen system~\eqref{eq:odea}, we define a {\em regular threshold} $\theta\subset \R^n$ as a codimension-one  embedded orientable forward-invariant  manifold  that is normally hyperbolic and repelling.
\end{defn}
\begin{rmk}
\label{rmk:edgestatecontraction}
Recall that a forward-invariant manifold $\theta$ is normally hyperbolic 
if perturbations transverse to $\theta$ can be characterised using exponentially growing or decaying modes, and these rates of growth or decay are larger than any rate of growth or decay of perturbations within the manifold. It is normally repelling if all transverse perturbations grow exponentially. More precise statements can be found, for example, in \cite{Fenichel1971,Fenichel1979,Kuehn2011}.
\end{rmk}
\begin{rmk}
\label{rmk:nuqthr}
Any codimension-one forward-invariant subset of a regular threshold is clearly also a regular threshold. In this sense regular thresholds are not unique.
\end{rmk}
\begin{rmk}
\label{rmk:qthr}
There is a close relationship between a regular threshold 
and a basin boundary of an attractor:
\begin{itemize}
\item[(a)]
A regular threshold $\theta$ will
typically be contained in the basin boundary of one or more attractors. 
For example, Fig.~\ref{fig:rthr} depicts regular thresholds $\theta$ that 
lie in the basin boundary of (a) one attractor, (b) two attractors or (c) three attractors.
\item[(b)]
Not all points on the basin boundary need to be in regular thresholds.
In Fig.~\ref{fig:rthr}(a), a regular threshold can be chosen to be any codimension-one connected subset of the stable manifold of $\eta$ containing  $\eta$, in which case there will be parts of the stable manifold that are part of the basin boundary of $e_1$ but not part of the threshold. If a regular threshold is chosen to be the entire stable manifold of $\eta$, as shown (in blue) in the figure, it still has boundary that is a source: this (black dot) source is part of the basin boundary of $e_1$ but not part of the threshold.
\end{itemize}
\end{rmk}

\begin{figure}[t]
  \begin{center}
    \includegraphics[width=16.cm]{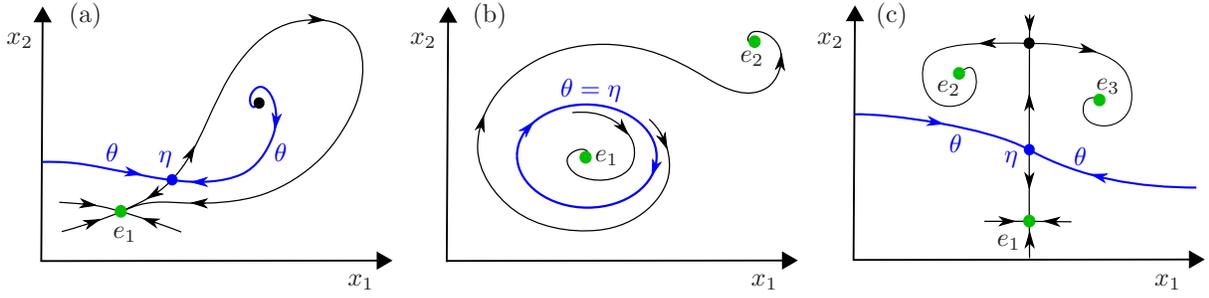}
  \end{center}
  \vspace{-3mm}
  \caption{
Examples of a (blue) regular threshold $\theta$ and the associated regular edge state $\eta$ in a
two-dimensional autonomous frozen system~\eqref{eq:odea}. 
(a) A regular edge state $\eta$ that is a hyperbolic saddle equilibrium. The 
associated regular threshold $\theta$ is any codimension-one forward-invariant 
subset of the stable manifold of $\eta$, so that 
$\eta\subset \theta$. This $\theta$ lies in the basin 
boundary of one attractor,  and the two sides of $\theta$ are in 
the basin of attraction of the same attractor $e_1$.
(b) A regular edge state $\eta$ that is a repelling hyperbolic limit cycle. 
The associated regular threshold is  the same limit cycle,
so that $\eta=\theta$. This $\theta$ lies on the basin boundary of two attractors, and each side of $\theta$ lies in the 
basin of attraction of a different attractor, that is $e_1$ and $e_2$. (c) A regular threshold $\theta$ that lies on the basin boundary of three attractors.
}
\label{fig:rthr}
\end{figure}

%
The assumption of forward invariance means that a regular threshold may contain several invariant sets that are  attractors for the flow restricted to the threshold. Here, we consider compact normally hyperbolic invariant sets $\eta$ that are attracting within $\theta$, together with their stable invariant manifolds\,\footnote{Note that
the stable invariant manifold of $\eta$ contains $\eta$.}, denoted $W^s(\eta)$. 
Using notation inspired by work on fluid instabilities~\cite{Schneideretal07,Schneideretal10,Skufkaetal2006} and climate instabilities~\cite{Ghil2020,Lucarini2017}, we define a regular edge state as follows:
\begin{defn}
\label{defn:edgestate}
In the $n$-dimensional autonomous frozen system~\eqref{eq:odea}, consider a regular threshold $\theta$. We call a compact normally hyperbolic invariant set $\eta\subseteq \theta$ 
a {\em regular edge state} of the regular threshold $\theta$ if
$\eta$ is an attractor\,\footnote{See Appendix~\ref{sec:A4} for the definition of an attractor.} for the flow restricted to $\theta$ and $\theta\subseteq W^s(\eta)$.
\end{defn}
\begin{rmk}
\label{rmk:edgestate0}
Not every regular threshold $\theta$ can be associated with a unique regular edge state $\eta$. For example, points in $\theta$ may limit to a continuum of  equilibria that are neutrally stable within $\theta$, or they may limit to several different attractors within $\theta$ that are regular edge states for forward-invariant subsets of $\theta$.
\end{rmk}
\begin{rmk}
\label{rmk:edgestate}
Recall from Definition~\ref{defn:rthr} that, in an $n$-dimensional frozen system~\eqref{eq:odea}, a regular threshold $\theta$ has dimension $(n-1)$. 
A regular edge state $\eta$ may be of the same or lower dimension than $\theta$.
If  $\eta$ is of the same dimension as $\theta$, then $\eta = \theta = W^s(\eta)$, and $\eta$ is normally repelling.
Examples of such $\eta$ include a repelling equilibrium for $n=1$, a repelling limit cycle for $n=2$ (see Fig.~\ref{fig:rthr}(b)), or more generally a repelling $(n-1)$-torus.
If $\eta$ is of lower dimension than $\theta$, then $\eta\subset\theta\subseteq W^s(\eta)$, and $\eta$
is of saddle type owing to attraction within $\theta$ and normal repulsion of $\theta$.
Examples of such $\eta$ include a saddle equilibrium with one unstable direction  as depicted in Fig.~\ref{fig:rthr}(a) and (c),
a saddle limit cycle with one unstable direction, or a saddle $(n-2)$-torus with one unstable direction.
\end{rmk}

The assumption of normal hyperbolicity implies that it is possible to 
extend regular edge states 
and associated regular thresholds of the frozen system~\eqref{eq:odea} 
to nearby $\lambda$; see~\cite[Theorems 3 and 4]{Fenichel1971}.
We state this rigorously for regular edge states that are hyperbolic
equilibria with precisely one unstable dimension:

\begin{proposition}
\label{prop:edgecontinues}
Suppose that $\eta^*$ is a hyperbolic equilibrium with one unstable direction in the autonomous frozen system \eqref{eq:odea} with $\lambda=\lambda^*$. Then:
\begin{itemize} 
\item[(a)]
The equilibrium $\eta^*$ is a regular edge state. There exists a regular threshold $\theta^*$ that is a forward invariant subset of the stable manifold of $\eta^*$.
\item[(b)]
There is an open neighbourhood $Q$ of $\lambda^*$ in $\R^d$ such that $\eta^*$ can be continued to a family of regular edge states
$\eta(\lambda)$, and $\theta^*$ can be continued to a family of regular thresholds $\theta(\lambda)$ containing $\eta(\lambda)$. These families vary $C^1$-smoothly with $\lambda\in Q$.
\item[(c)]
There is a continuous parametrization of $\theta(\lambda)$ by $x\in \theta^*$ and $\lambda\in Q$. This parameterization can be chosen so that the normal vector $\nu(x,\lambda)$ to $\theta(\lambda)$ varies $C^1$-smoothly with  $\lambda\in Q$.
\end{itemize}
\end{proposition} 

\begin{proof}
Note that $\eta^*$ is an unstable node in $\mathbb{R}$, in which case $W^s(\eta^*)=\eta^*$, or a saddle in $\mathbb{R}^{n\ge 2}$, in which case $\eta^*\in W^s(\eta^*)$.

\noindent
(a)	We choose $\theta^*$ to be a local stable manifold of $\eta^*$, denoted $W^s_{loc}(\eta^*)$, as given by the stable manifold theorem; see e.g.~\cite[Thm 2.1.2]{Kuehn2015}. 
This means that $\theta^*$  is topologically a codimension-one ball that is forward invariant, contractable to $\eta^*$, and one can  choose a normal vector (an orientation) 
corresponding to
the unstable eigenvector of $\eta^*$, which varies smoothly with
$x\in\theta^*$.
Thus, $\eta^*$ is a regular edge state and $\theta^*$ is a regular threshold.

\noindent
(b) Applying results of Fenichel on persistence of   normally hyperbolic invariant manifolds that are compact and embedded (see~\cite[Thm 3]{Fenichel1971} or~\cite[Thm 2.3.5]{Kuehn2015}), there is an open neighbourhood $Q$ of $\lambda^*$ such that $\eta^*$ can be continued to a family of hyperbolic equilibria $\eta(\lambda)$ that varies $C^1$-smoothly with  $\lambda\in Q$.
Similarly, applying results on persistence  of stable/unstable manifolds of normally hyperbolic invariant manifolds
(see e.g.~\cite[Thm 4]{Fenichel1971} or~\cite[Thm 2.3.6]{Kuehn2015}), $W^s_{loc}(\eta^*)$ can be $C^1$-smoothly continued to a family of stable manifolds of $\eta(\lambda)$, and each of these manifolds contains a regular threshold $\theta(\lambda)$ that varies $C^1$-smoothly with  $\lambda\in Q$.

\noindent
(c) The continuous parameterization by $(x,\lambda)$ is a consequence of applying results of~\cite{Fenichel1974} and \cite{Fenichel1977} or~\cite[Thm 2.3.12]{Kuehn2015}. 
Orientability implies that there are two choices
of a normal vector $\pm\nu(x)$ that vary smoothly with $x\in\theta^*$ and $\lambda\in Q$, and  thus a well-defined notion of the two sides (e.g. inside/outside) of a regular threshold. 
\end{proof}

To relate our concept of regular thresholds to 
existing literature~\cite{Krauskopf2003,Pisarchik2014}, we distinguish between notions of ``excitability threshold" for excitable
systems and ``multi-basin boundary" for multistable systems as being
different kinds of thresholds.
\begin{defn}
\label{defn:multi_excit}
Let $\theta(\lambda)$ be a regular threshold 
for the autonomous frozen system (\ref{eq:odea}).
\begin{itemize}
\item[(a)] 
If $\theta(\lambda)$ is contained in the basin boundary of two or more
attractors, we say that the autonomous frozen system~\eqref{eq:odea} is {\em multistable} with {\em multi-basin 
boundary $\theta(\lambda)$}.
\item[(b)] 
If $\theta(\lambda)$ is contained in the basin boundary of a 
single attractor, we say that the autonomous frozen system~\eqref{eq:odea} is  {\em excitable} with {\em excitability threshold $\theta(\lambda)$}. 
\end{itemize}
\end{defn}

\subsection{Moving Regular Thresholds and Moving Regular Edge States}
\label{sec:thr}

It follows from Definition~\ref{defn:edgestate} that, if there is a regular edge state $\eta(\lambda)$, then there is a regular threshold $\theta(\lambda)$ containing $\eta(\lambda)$.
For a given external input $\Lambda(\T)$, we use the notion of a parameter path $P_{\Lambda,I}$ from Definition~\ref{def:path} and define
moving regular edge states and moving regular thresholds analogously to moving sinks, namely
as follows:
\begin{defn}
\label{defn:mth}
Suppose the autonomous frozen system~(\ref{eq:odea}) has a codimension-one 
forward-invariant manifold $\theta(\lambda)$ and a compact invariant set $\eta(\lambda)\subseteq\theta(\lambda)$  for some connected set of values of $\lambda$.
 Consider an external input $\Lambda(\T)$ that traces out a parameter path $P_{\Lambda,I}$ on a time interval $I=(\tau_-,\tau_+)\subseteq\mathbb{R}$, where $\T_\pm$ can be $\pm\infty$. Then,
\begin{itemize}
    \item[(a)]
       We say $\theta\left(\Lambda(\T)\right)$ is a {\em moving regular threshold} on $I$ if $\theta(\lambda)$ is a regular threshold that varies $C^1$-smoothly with $\lambda\in P_{\Lambda,I}$.
   
    \item[(b)]
     We say $\eta\left(\Lambda(\T)\right)$ is a {\em moving regular edge state} on $I$ 
    if $\eta(\lambda)$ is a regular edge state that varies $C^1$-smoothly with $\lambda\in P_{\Lambda,I}$.
     
    \item[(c)]
    Suppose that $\Lambda(\T)$ is asymptotically constant to $\lambda^+$, and $\eta\left(\Lambda(\T)\right)$ is a 
    moving regular edge state on $I=(\tau_-,+\infty)$. Then, we define the {\em future limit} $\eta^+$ of the 
    moving regular edge state by
    $$
    \eta^+=\eta(\lambda^+).
    $$
\end{itemize}
\end{defn}
\begin{rmk}
The assumption in (c) implies that $\eta^+$ is a regular edge state  of a regular threshold 
$$
\theta^+=\theta(\lambda^+),
$$
for the autonomous future limit system~\eqref{eq:odea+}.
\end{rmk}
A moving regular threshold (edge state) is 
a regular threshold (edge state) of the autonomous frozen system~(\ref{eq:odea})
parametrised by time $\T$ for a given input $\lambda=\Lambda(\T)$.
Similar to moving sinks, moving regular thresholds (edge states)
are considered in the  phase space $\mathbb{R}^n$ of the nonautonomous system~(\ref{eq:odewithrs}). They
depend on $f$ and on the shape of the external input $\Lambda$, but do not depend on the rate parameter $r>0$ when viewed on the external input timescale $\tau$. 
Regular edge states $\eta^+$ of the future limit system~\eqref{eq:odea+} are particularly important in our work. This is because regular R-tipping thresholds are anchored at infinity by $\eta^+$.

\subsection{Threshold Instability of a Sink}
\label{sec:thr_inst}

A theory of {\em irreversible R-tipping}  in one-dimensional (scalar) nonautonomous system~\eqref{eq:odewithr} or~\eqref{eq:odewithrs}  presented in~\cite{Ashwin2016}
is based on moving sinks on $I=\R$ and the intuitive concept of {\em forward basin 
stability}\,\footnote{Not to be confused with the `static' notion of ``basin stability"
introduced in~\cite{Menck2013} as a measure related to the volume of the 
basin of attraction. 
}
of a moving  sink; see~\cite[Def.3.3]{Ashwin2016}.
To be more specific, a moving sink $e(\Lambda(\T))$ is {\em forward basin stable} if, at each point in time, $e(\Lambda(\T))$ is contained  in the basin of attraction of its every future position\,\footnote{Equivalently,
a moving sink $e(\Lambda(\T))$ is ``forward basin stable" if, at each point in time, the basin of attraction of $e(\Lambda(\T))$ contains all the previous positions of $e(\Lambda(\T))$.}. This concept was used in~\cite[Th.3.2]{Ashwin2016} to derive easily testable criteria for the absence or presence of irreversible R-tipping for a moving sink on $I=\R$ in one dimension:  forward basin stability in autonomous frozen
system~\eqref{eq:odea} 
with $x\in\mathbb{R}$ excludes R-tipping in nonautonomous system~\eqref{eq:odewithr} or~\eqref{eq:odewithrs}, whereas lack of forward basin stability (plus some additional assumptions) in system~\eqref{eq:odea} 
with $x\in\mathbb{R}$ guarantees R-tipping in system~\eqref{eq:odewithr} or~\eqref{eq:odewithrs}. 
The key point in the derivation of these criteria is that, in $\mathbb{R}$, trajectories started within the basin of attraction approach the attractor monotonically in time. Another point is that, 
in $\mathbb{R}$, a typical basin boundary is a boundary of two attractors unless trajectories on one side of the boundary diverge to infinity. Thus, one typically expects irreversible R-tipping in $\R$.

However, a theory that works in arbitrary dimension and captures both irreversible and reversible R-tipping requires a more 
sophisticated understanding. First, the concept of {\em forward basin
stability} from~\cite{Ashwin2016} is no longer useful. If trajectories
started within the basin of attraction can approach the attractor non-monotonically in time, then forward basin stability in system~(\ref{eq:odea}) with $x\in\mathbb{R}^{n\,\ge\,2}$ no longer excludes R-tipping in system~(\ref{eq:odewithr}) or~\eqref{eq:odewithrs}. This is evidenced by examples of irreversible R-tipping  for a moving sink on $I=\R$ occurring in spite of forward basin stability already in two dimensions~\cite{Kiers2018,Xie2019}.  Second, in two or more dimensions, there can be {\em reversible R-tipping} 
due to crossing a basin boundary 
of a single attractor; see Fig.~\ref{fig:rthr}(a). 
The concept of {\em basin instability} from~\cite{OKeeffe2019} addresses only part of the problem: it gives easily testable criteria for the occurrence of irreversible R-tipping for a moving sink on $I=\R$ in multidimensional systems, but is not useful for reversible R-tipping.

To properly address the problem of different cases of R-tipping in arbitrary dimension, we introduce the more general concepts of {\em threshold instability} and {\em forward threshold instability}.
In short, {\em threshold instability} of a hyperbolic sink on a parameter path describes the position of the sink at some points on the path relative to the position of the threshold at different points on the path. To be specific, we introduce two notions. First, we quantify ``relative position of a sink and a threshold" using the signed distance\,\footnote{The signed distance $d_s(x,S)$ is discussed in Appendix~\ref{sec:A3}} between the point $e(\lambda_1)$ and the set $\theta(\lambda_2)$:
\begin{equation}
\label{eq:sdlambda}
d_s(e(\lambda_1),\theta(\lambda_2)).
\end{equation}
Second, we describe $e(\lambda)$ and $\theta(\lambda)$ at ``different points on the path"
by constructing the subset 
$$
P^2 :=P\times P \subset \mathbb{R}^{d\times d},
$$
and viewing pairs $(\lambda_1,\lambda_2)$ of different 
input parameters as elements of this subset. 
We can then define threshold instability, which generalises the notion of basin instability from~\cite{OKeeffe2019}.
\begin{defn} 
\label{def:thun}
Suppose the autonomous frozen system~(\ref{eq:odea}) has a  hyperbolic sink $e(\lambda)$. Consider a parameter path $P$ such that $e(\lambda)$ varies $C^1$-smoothly with $\lambda\in  P$. 
\begin{itemize}
\item [(a)]
We say $e(\lambda)$ is {\em threshold unstable on $P$} if
there exists a $C^1$-smooth family of regular thresholds $\theta(\lambda)$ and a $(\lambda_{a},\lambda_{b})\in P^2$ such that 
$$
 e(\lambda_{a})\in \theta(\lambda_b) \quad\mbox{i.e.} \quad d_s\left(e(\lambda_a),\theta(\lambda_b)\right) = 0,
$$
and
$d_s(e(\lambda_1),\theta(\lambda_2))$ takes both signs 
in any neighbourhood 
of $(\lambda_a,\lambda_b)$ in $P^2$.
\item[(b)]
We say $e(\lambda)$ is {\em basin unstable on $P$} if  it is threshold unstable on $P$, and the threshold $\theta(\lambda_b)$ is contained in a multi-basin boundary.
\end{itemize}
\end{defn}
\begin{rmk}
Note that, if $e(\lambda)$ is threshold unstable, then there is a 
crossing of the threshold  $\theta(\lambda_2)$ from one side to another by the sink $e(\lambda_1)$,
i.e. a passage through zero with a change in the sign of $d_s(e(\lambda_1),\theta(\lambda_2))$. 
In practice, this could happen when: setting $\lambda_{1(2)}=\lambda_{a(b)}$ and varying $\lambda_{2(1)}$ in a neighbourhood of $\lambda_{b(a)}$ in $P$, or varying $\lambda_{1}$ and $\lambda_2$ near $\lambda_{a}$ and $\lambda_b$, respectively,
both in $P$.
\end{rmk}
Threshold instability on a parameter path $P$  in the autonomous frozen system~\eqref{eq:odea} 
indicates that R-tipping is possible in the nonautonomous system~\eqref{eq:odewithr} or~\eqref{eq:odewithrs} 
given a suitable external input that traces out $P$. To understand which external inputs are ``suitable", we consider $\Lambda(\T)$ for which the moving sink $e(\Lambda(\T))$
crosses some future position of a moving regular threshold $\theta(\Lambda(\T))$ from one side to the other. To this end, we introduce a notation for the signed distance at different points in time:
\begin{equation}
\label{eq:sdtau}
\Delta_{\Lambda}(\T_1,\T_2)=d_s(e(\Lambda(\T_1)),\theta(\Lambda(\T_2))),
\end{equation}
consider pairs $(\T_1,\T_2)$ of different points in time  as elements of $\mathbb{R}^2$,
and define forward threshold instability.
\begin{defn}
\label{def:fthun}
Consider some external input $\Lambda(\T)$ and a moving sink $e(\Lambda(\T))$.
\begin{itemize}
    \item [(a)]
    We say $e(\Lambda(\T))$ is {\em forward threshold unstable for $\Lambda(\T)$} if 
    there exist a moving regular threshold $\theta(\Lambda(\T))$ and finite $\T_{a} < \T_b$ such that
    \begin{equation}
    \label{eq:ftunst}
    e(\Lambda(\T_a)) \in \theta(\Lambda(\T_b))\quad\mbox{i.e.} \quad \Delta_{\Lambda}(\T_a,\T_b) = 0,
    \end{equation}
    and $\Delta_{\Lambda}(\T_1,\T_2)$ takes both signs in any neighbourhood 
    of $(\T_a,\T_b)$ in $\mathbb{R}^2$.
    \item[(b)]
    We say this $e(\Lambda(\T))$ is {\em forward basin unstable for $\Lambda(\T)$} if it is forward threshold unstable 
    for $\Lambda(\T)$,
    and the threshold $\theta(\Lambda(\T_b))$ is contained in a multi-basin boundary.
\end{itemize}
\end{defn}
\begin{rmk}
Note that, if a moving sink $e(\Lambda(\T))$ is forward threshold unstable, then, in some sense, there is
crossing of the moving threshold by 
$e(\Lambda(\T))$ from one side to the other.
\end{rmk}
{\em Forward threshold instability} is a property of the
autonomous frozen system~(\ref{eq:odea}) 
and some external input
$\Lambda(\T)$. {\em Threshold instability} is a property of the
frozen system~(\ref{eq:odea}) on a given parameter path $P$. Threshold
instability on a path  $P$ guarantees existence of some input $\Lambda(\T)$ that
traces out this path, meaning that $P_\Lambda=P$, and gives forward threshold
instability. However, there may be other 
inputs $\tilde{\Lambda}(\T) \ne \Lambda(\T)$ that trace out the same
path, meaning that $P_{\tilde{\Lambda}} = P_\Lambda = P$, but do not give forward threshold instability. 
This is illustrated in Fig.~\ref{fig:ftu}.
%
\begin{figure}[t]
  \begin{center}
    \includegraphics[width=15.cm]{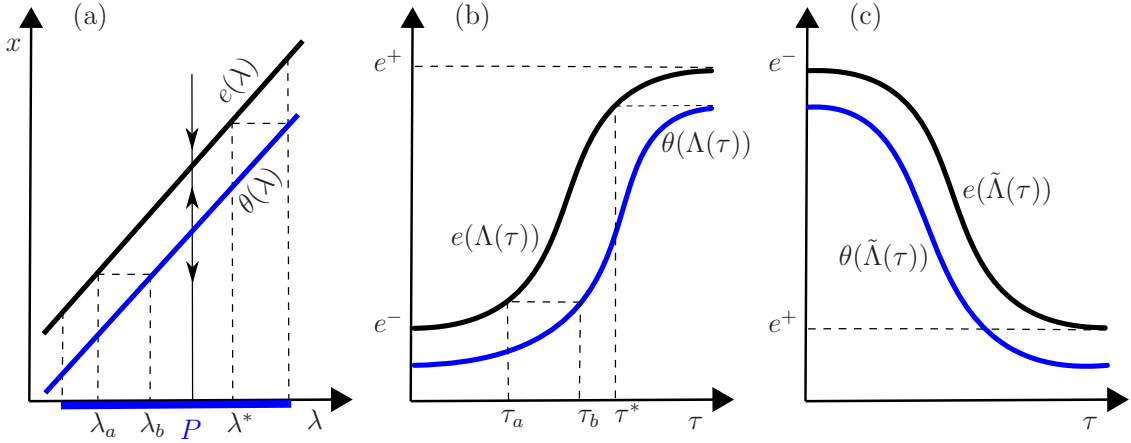}
  \end{center}
  \vspace{-3mm}
  \caption{
    (a) Families (branches) of hyperbolic sinks $e(\lambda)$ and equilibrium regular thresholds $\theta(\lambda)$ for a one-dimensional (scalar) autonomous frozen system~(\ref{eq:odea}), together with a given parameter path $P$. The pair 
    $(\lambda_a,\lambda_b)\in P^2$ indicates threshold instability of $e(\lambda)$ on $P$: 
    for any $\lambda_a\in P$ smaller than $\lambda^*$ there exists a $\lambda_b\in P$ such that $e(\lambda_{a})\in\theta(\lambda_b)$, and $e(\lambda_a)$ can lie on different sides of $\theta(\lambda)$ for $\lambda$ arbitrarily close to $\lambda_b$:
    (b) For a monotone increasing $\Lambda(\T)$ that traces out the path $P$, the moving sink $e(\Lambda(\T))$ is forward threshold unstable 
    because it crosses through future positions of the moving threshold $\theta(\Lambda(\T))$.  For any $\T_a\in(-\infty,\T^*)$ there exist a $\T_b > \T_a$ such that $e(\Lambda(\T))$  at $\T=\T_a$ crosses $\theta(\Lambda(\T_b))$ from one side to the other.
    (c) For a monotone decreasing $\tilde{\Lambda}(\T)$ that traces out the same path $P$, the moving sink $e(\tilde{\Lambda}(\T))$ is forward threshold stable because it never crosses any future position of the moving threshold $\theta(\tilde{\Lambda}(\T))$. There are no finite $\T_a < \T_b$ that can satisfy $e(\tilde{\Lambda}(\T_a))\in\theta(\tilde{\Lambda}(\T_b))$. We say $e(\tilde{\Lambda}(\T))$ is forward threshold stable.
}
  \label{fig:ftu}
\end{figure}

\section{Nonautonomous R-tipping Definitions}
\label{sec:Rtippingdefs}

We now define a {\em nonautonomous R-tipping bifurcation via loss of end-point tracking} in nonautonomous 
system~\eqref{eq:odewithrs}  with asymptotically constant input $\Lambda$,
in a precise yet general context. 
In addition to reversible, irreversible and degenerate cases of R-tipping, we also define critical rates for R-tipping, regular R-tipping edge states and their edge tails, and time-dependent  regular R-tipping thresholds.

\subsection{R-tipping and Critical Rates}
\label{sec:Rtip_critrates}

We start by defining R-tipping and critical rates in terms of limiting behaviour 
of trajectories of the nonautonomous system (\ref{eq:odewithrs}); note that this
generalises the definition of R-tipping in~\cite{Ashwin2016}. 
\begin{defn} 
  \label{defn:Rtip}
    Consider a nonautonomous system~(\ref{eq:odewithrs}) 
    with an external input $\Lambda(\T)$ that is asymptotically constant to $\lambda^+$. 
    Suppose the future limit system~\eqref{eq:odea+} has a compact invariant set 
    $\eta^+$ that is not an attractor\,\footnote{Note that $\eta^+$ is not necessarily  
    a regular edge state from Definition~\ref{defn:mth}(c); it may be a saddle with more than one unstable direction, or even a repeller  of codimension-two or higher, and/or not necessarily hyperbolic.}.
\begin{enumerate}
    \item[(a)] 
    We say the nonautonomous system~(\ref{eq:odewithrs})
    undergoes {\em R-tipping from $(x_0,\tau_0)$} if there are
    $r_1,r_2>0$ 
    such that
    $$
    x^{[r_1]}(\T,x_0,\T_0) \to \eta^+
    \;\;\mbox{and}\;\;
    x^{[r_2]}(\T,x_0,\T_0) \not\to \eta^+
    \;\;\mbox{as}\;\;\T\to+\infty.
    $$
    \item[(b)] 
    Suppose in addition that $\Lambda(\T)$ is bi-asymptotically  constant and the past limit system~(\ref{eq:odea-}) has a hyperbolic sink $e^{-}$. We say the nonautonomous  system~(\ref{eq:odewithrs}) undergoes {\em R-tipping from $e^-$} if there are  $r_1,r_2>0$ such that 
    $$
    x^{[r_1]}(\T,e^-)\to \eta^+
    \mbox{ and }\;\;
    x^{[r_2]}(\T,e^-) \not\to \eta^+\;\;\mbox{as}\;\;\T\to+\infty.
    $$
    \item[(c)] 
    If there is an $r_1>0$ and a $\delta > 0$ 
    such that
    $$
    x^{[r_1]}(\T)\to \eta^+
    \;\;\mbox{and}\;\;
    x^{[r]}(\T)\not\to \eta^+
    \;\;\mbox{as}\;\;\T\to +\infty
    \;\;\mbox{for all}\;\;
    0 < |r - r_1|< \delta,
    $$
    then we say $r_1$ is a {\em critical rate} and denote it with $r_c$.
\end{enumerate}
\end{defn}
\begin{rmk}
\label{rmk:Rtip}
For typical   systems~\eqref{eq:odewithrs} with typical choices of initial condition and the rate parameter $r$, one expects a solution $x^{[r]}(\T)$ that remains bounded to converge to an attractor $a^+$ for the future limit system~\eqref{eq:odea+} rather than converging to something that is not an attractor.

To see this, suppose the future limit system~\eqref{eq:odea+} has a compact invariant set $a^+$ that is an attractor, consider a solution
$$
x^{[r]}(\T) \to a^+\;\;\mbox{as}\;\;\T\to+\infty\;\;\mbox{for some}\;\;r=r_1>0,
$$
and note that this solution can be extended to a family of solutions that is continuous in $\T$, $r$ and initial condition; see for example~\cite[Theorem 3.3]{Robinson1999}. Thus, the same limiting behaviour occurs for an open set of $x$ containing the initial condition and an open set of $r$ containing $r_1$.
A consequence of this robustness to small variations in $r$ is that if the future limit system~\eqref{eq:odea+} has
disjoint compact invariant sets $a_2^+$ and $a_3^+$
that are attractors, and there are rates $0<r_2<r_3$ such that
$$
x^{[r_2]}(\T)\rightarrow a_2^+~\mbox{ and }~x^{[r_3]}(\T)\rightarrow a_3^+~\mbox{ as }~\T\rightarrow +\infty,
$$
then the future limit system~\eqref{eq:odea+} must have at least one compact invariant set $\eta^+$ on the basin boundary of $a_2^+$ and $a_3^+$ that is not an attractor, and there must be at least one rate $r_1\in(r_2,r_3)$ such that  there is R-tipping in the sense of Definition~\ref{defn:Rtip}, namely
$$
x^{[r_1]}(\T)\to\eta^+~\mbox{ as }~\T\rightarrow +\infty.
$$

\end{rmk}

Figure~\ref{fig:Rtip} shows two examples of R-tipping via loss of end-point tracking from Definition~\ref{defn:Rtip} 
for a nonautonomous system~\eqref{eq:odewithrs} on $\R$.\,\footnote{In the one dimensional case, recall that the moving regular threshold and edge state are one and the same.} 
R-tipping from $e^-$ for a moving sink on $I=\R$ that is forward threshold unstable, shown  in Figure~\ref{fig:Rtip}(a), is discussed in~\cite{Ashwin2016} and extended to arbitrary dimension in Section~\ref{sec:Rtippingcriteria}.
However, R-tipping from $e^-$ for a moving sink on a semi-infinite interval $I=(-\infty,\T_+)$ that is forward threshold stable\footnote{We say a moving sink $e(\Lambda(\T))$ on $I$ is forward threshold stable if there are no $\theta(\Lambda(\T))$ and finite $\T_a<\T_b\in I$ that can satisfy condition~\eqref{eq:ftunst}.},
shown in Figure~\ref{fig:Rtip}(b), is not captured by the setting  used in~\cite{Ashwin2016} and Section~\ref{sec:Rtippingcriteria}, which is limited to moving sinks on $I=\R$  that are forward threshold unstable.
To overcome this limitation, we show in Section~\ref{sec:computing} that different R-tipping via loss of end-point tracking, including the example in Figure~\ref{fig:Rtip}(b), can be captured in arbitrary dimension by connecting (heteroclinic) orbits in a suitably compactified system. 

\begin{figure}[t]
  \begin{center}
    \includegraphics[width=14cm]{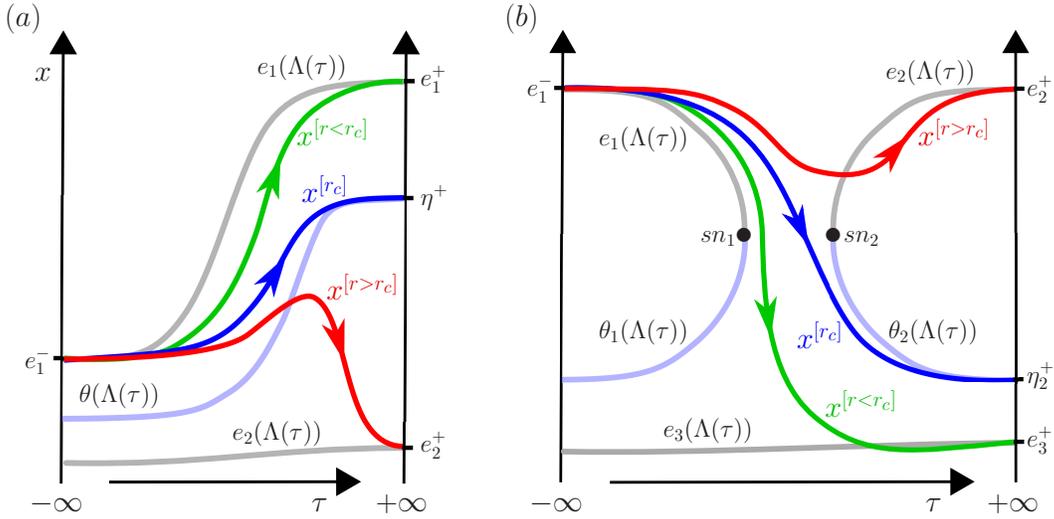}
  \end{center}
  \vspace{-5mm}
  \caption{
  Two examples of R-tipping via loss of end-point tracking from Definition~\ref{defn:Rtip} for the case of a nonautonomous system~\eqref{eq:odewithrs} with  $x\in\R$. As the critical rate crosses $r=r_c$, the trajectory crosses a {\em regular R-tipping threshold} (see Definition~\ref{def:rtipthres}) and limits to an equilibrium that is a {\em regular R-tipping edge state}  (see Definition~\ref{def:rtipedge}). Shown are (grey) moving sinks $e(\Lambda(\T))$, (light blue) moving equilibrium regular thresholds $\theta(\Lambda(\T))$, and trajectories of~\eqref{eq:odewithrs} limiting to a sink $e_1^-$ as $\T\to-\infty$ 
  for different values of the rate parameter: (green) $r<r_c$, (blue) $r=r_c$, and (red) $r>r_c$.
  (a) R-tipping from $e_1^-$ via loss of end-point tracking of $e_1(\Lambda(\T))$, due to crossing the regular R-tipping threshold $\Theta^{[r]}(\T)$ (not shown) anchored at infinity by the equilibrium regular R-tipping  edge state $\eta^+=\theta^+$. 
  Note that $e_1(\Lambda(\T))$ is a moving sink on $I=\R$, 
  that is forward threshold unstable due to $\theta(\Lambda(\T))$.
  (b) R-tipping from $e_1^-$ via loss of end-point tracking of $e_3(\Lambda(\T))$, due to crossing the regular R-tipping threshold $\Theta_2^{[r]}(\T)$ (not shown) anchored at infinity by the equilibrium regular R-tipping  edge state $\eta_2^+=\theta_2^+$. 
  Note that $e_1(\Lambda(\T))$ is a moving sink on a semi-infinite interval $I$, disappears at a finite time via (black dot) a saddle-node bifurcation $sn_1$, and is forward threshold stable, which is different from (a). 
  Furthermore, the saddle-node bifurcation of $e_1(\Lambda(\T))$ gives rise to (green) B-tipping from $e_1^-$ for $r < r_c$~\cite[Definition 3.1]{Ashwin2016}.
  }
  \label{fig:Rtip}
\end{figure}

\subsection{R-tipping Thresholds and R-tipping Edge States}
\label{sec:Rtip_thr}

Next, we recognise the significance of $\eta^+$ that are regular edge states from Definition~\ref{defn:mth}(c).
\begin{defn}
\label{def:rtipedge}
Suppose that a nonautonomous system \eqref{eq:odewithrs} undergoes R-tipping as in Definition~\ref{defn:Rtip}, and $\eta^+$ is a regular edge state  of the future limit system. Then we say $\eta^+$ is a {\em regular R-tipping edge state}.
\end{defn}

Then, we consider  R-tipping thresholds that are anchored at infinity by a regular R-tipping edge state $\eta^+$. These thresholds are regular in the same sense as regular thresholds from Definition~\ref{defn:rthr}.
\begin{defn}
\label{def:rtipthres}
Consider a nonautonomous system~(\ref{eq:odewithrs}) with an 
external input $\Lambda(\T)$ that is asymptotically constant to $\lambda^+$.
Suppose the future limit system~\eqref{eq:odea+} has a regular R-tipping edge state $\eta^{+}$.
We say  $\Theta^{[r]}(\T) \subset\mathbb{R}^n$ is a {\em regular R-tipping threshold} if it is a codimension-one embedded orientable forward-invariant 
subset of the stable set of $\eta^+$.
\end{defn}
By ``forward invariant" we mean that it is forward invariant as a nonautonomous set, i.e.
$$
x_0\in\Theta^{[r]}(\T_0) 
\Rightarrow
x^{[r]}(\T,x_0,\T_0)\in \Theta^{[r]}(\T)\;\;\mbox{for all}\;\;\T>\T_0.
$$
By ``stable set of $\eta^+$" we mean that
\begin{equation}
x_0\in\Theta^{[r]}(\T_0) 
\Rightarrow x^{[r]}(\T,x_0,\T_0)\to \eta^+\;\;\mbox{as}\;\; \T\to +\infty.\label{eq:etastable}
\end{equation}
\begin{rmk}
 Note that:
\begin{itemize}
    \item[(a)] 
    A regular R-tipping threshold $\Theta^{[r]}(\T)$ is a rate and time dependent subset of $\mathbb{R}^n$.
    \item[(b)] 
    We prove existence of regular R-tipping thresholds $\Theta^{[r]}(\T)$ 
    in Proposition~\ref{prop:invsete-}(b1). In particular, we state conditions under which a $\Theta^{[r]}(\T)$ exists for all  $\T > \T_0$ and $r>0$.
    \item[(c)] 
    Any codimension-one forward-invariant subset of a regular R-tipping threshold is clearly also a regular R-tipping threshold. In this sense regular R-tipping thresholds are not unique.
\end{itemize}
\end{rmk}

\subsection{Edge Tails}

We now focus on $\eta^+$ that are regular R-tipping edge states, and introduce for the first time a notion of {\em edge tails} to rigorously classify different cases of R-tipping that may occur via loss of end-point tracking.

Consider a rate-dependent solution $x^{[r]}(\T)$ of the nonautonomous system~\eqref{eq:odewithrs}, started from a fixed $(x_0,\T_0)$ or limiting to a sink $e^-$ as $\T\to -\infty$.
Suppose that end-point tracking of a moving sink $e(\Lambda(\T))$ by 
$x^{[r]}(\T)$
fails for some $r_c >0$ in the sense that
$$
x^{[r_c]}(\T) \to \eta^+\;\;\mbox{as}\;\; \T\to +\infty.
$$
If $\eta^+$ is a regular R-tipping edge state, then the system undergoes R-tipping due to crossing a regular R-tipping threshold $\Theta^{[r]}(\T)$.  If $r_c$ is a critical rate, then
for all $r\neq r_c$ sufficiently close we have
$$
x^{[r]}(\T)\not\to \eta^+\;\;\mbox{as}\;\; \T\to +\infty,
$$
and we generically expect that
$x^{[r<r_c]}(\T)$ and $x^{[r>r_c]}(\T)$ lie on different sides of the regular R-tipping threshold.
To be more precise about 
``lie on different sides of the regular R-tipping threshold'', 
we examine the corresponding trajectory\,\footnote{Recall the notation introduced in Section~\ref{sec:notation}.} $\mbox{trj}^{[r]}$
as the rate parameter $r$ approaches its critical value $r_c$ from above ($r\to r_c^+$) and from below  ($r\to r_c^-$).
The ensuing limit sets\,\footnote{Here, we define
$$
\lim_{r\to r_c^+} \mbox{\normalfont trj}^{[r]}(x_0,\tau_0) = \bigcap_{r>r_c}\; \overline{\bigcup_{r_c < s < r} \mbox{\normalfont trj}^{[s]}(x_0,\tau_0)}
\;\;\mbox{and}\;\;
\lim_{r\to r_c^-} \mbox{\normalfont trj}^{[r]}(x_0,\tau_0) = \bigcap_{r<r_c}\; \overline{\bigcup_{r <s < r_c} \mbox{\normalfont trj}^{[s]}(x_0,\tau_0)}.
$$}
can typically  be decomposed into two components:
\begin{equation}
\lim_{r\to r_c^\pm} \mbox{trj}^{[r]} =
\mbox{trj}^{[r_c]}
\cup x^{[r_c^\pm]}.
\label{eq:trjr}
\end{equation}
The first component, denoted  $\mbox{trj}^{[r_c]}$,
is the trajectory  of the nonautonomous system~\eqref{eq:odewithrs}  
from $x_0$ or $e^-$
to the regular R-tipping edge state $\eta^+$ in $\mathbb{R}^n$, which is common to both limits. 
Note that, being  a projection of a smooth curve from $\mathbb{R}^n\times\mathbb{R}$ onto $\mathbb{R}^n$,
$\mbox{trj}^{[r_c]}$ may intersect itself and  $x^{[r_c^{\pm}]}$.
The second component is either 
$x^{[r_c^+]}$  or $x^{[r_c^-]}$. We define these below as the upper and lower {\em edge tails} of  the regular R-tipping edge state $\eta^+$.   Each edge tail of $\eta^+$ is a (union of) trajectories of the autonomous future limit system~\eqref{eq:odea+} that includes  $\eta^+$ and continues away from 
$\eta^+$ in $\mathbb{R}^n$.
To be more precise, 
\begin{defn}
\label{defn:edgetails}
Consider a nonautonomous system~\eqref{eq:odewithrs}
with an external input $\Lambda(\T)$ that is asymptotically constant to $\lambda^+$.
Suppose the future limit system~\eqref{eq:odea+} has a regular R-tipping edge state $\eta^+$,
and the nonautonomous system~(\ref{eq:odewithrs}) undergoes
 R-tipping for some critical rate $r=r_c>0$
so that
$x^{[r_c]}(\T) \to \eta^+\;\;\mbox{as}\;\; \T\to +\infty$. 
Then, we define the {\em upper edge tail} of $\eta^+$ to be
\begin{align}
&x^{[r_c^+]}=\bigcap_{T>0, ~\delta>0} \overline{\left\{  x^{[r]}(\T)\;:\;\T>T,~r\in(r_c,r_c+\delta)\right\}}\; \subset\mathbb{R}^n,
\end{align}
and the {\em lower edge tail} of $\eta^+$ to be
\begin{align}
&x^{[r_c^-]}=\bigcap_{T>0, ~\delta>0} \overline{\left\{ x^{[r]}(\T)\;:\;\T>T,~r\in(r_c-\delta,r_c)\right\}}\; \subset\mathbb{R}^n.
\end{align}
\end{defn}
Edge tails of $\eta^+$ include $\eta^+$ and trajectories that are contained in the unstable manifold of $\eta^+$, denoted $W^u(\eta^+)$. The upper and lower edge tails are typically different as shown in Fig.~\ref{fig:Rtiptypes}(a) and (b). 
\begin{rmk}
\label{rmk:edgetail}
For an  equilibrium regular R-tipping edge state $\eta^+$, we: 
\begin{itemize}
    \item [(a)]
    Show that each edge tail contains one branch of $W^u(\eta^+)$ in Proposition~\ref{prop:invsete-}(b).
    \item [(b)]
    Relate solutions $x^{[r]}(\T)$ for $r$ on different sides of $r_c$ to the edge tails $x^{[r_c^-]}$ and $x^{[r_c^+]}$ in Proposition~\ref{prop:edgetails}.
\end{itemize}
\end{rmk}

\begin{figure}
  \begin{center}
    \includegraphics[width=11.2cm]{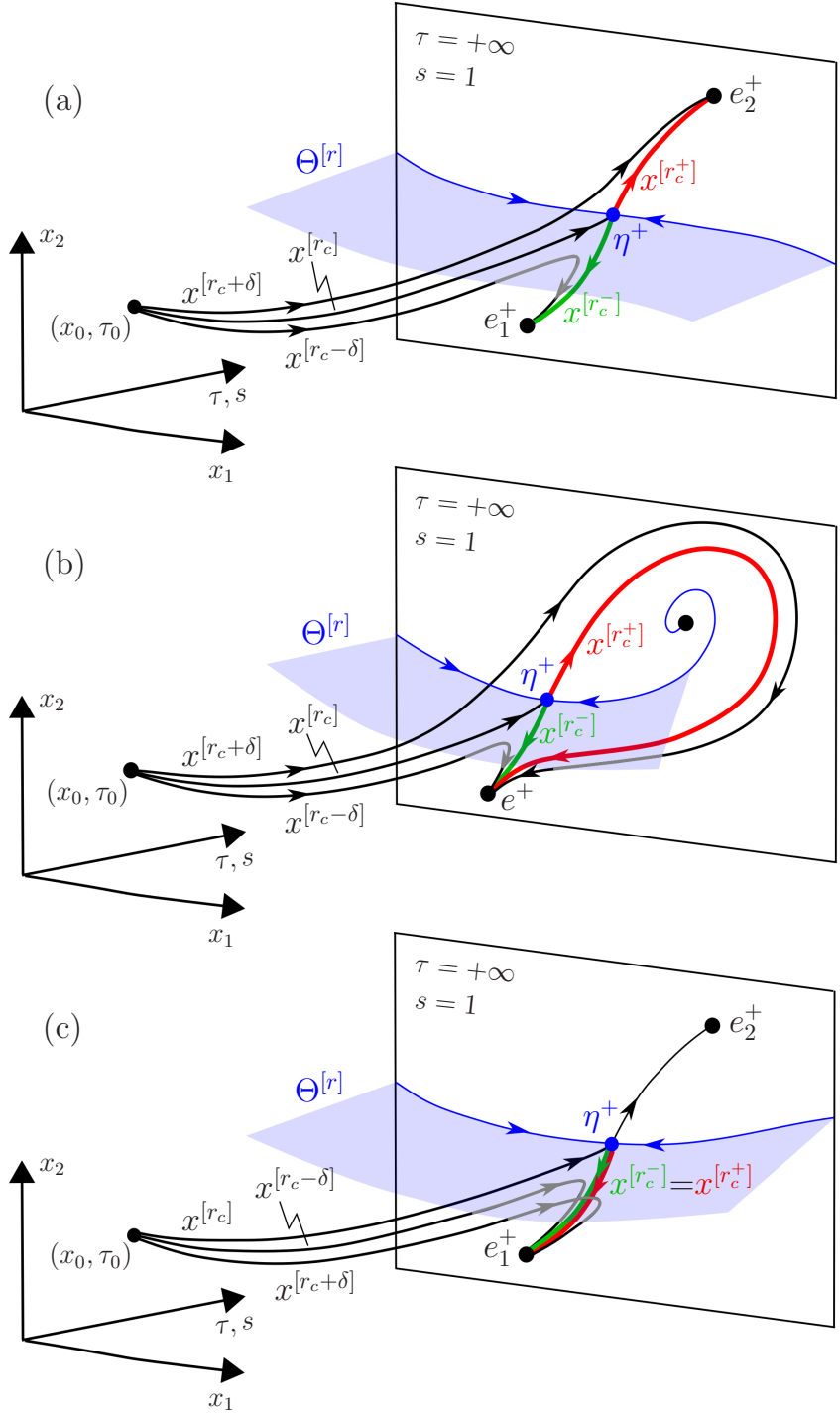}
  \end{center}
  \vspace{-5mm}
  \caption{
  Examples of (a) irreversible, (b) reversible and (c) degenerate  R-tipping via loss of end-point tracking from Definition~\ref{defn:Rtiptypes} for a nonautonomous system~\eqref{eq:odewithrs} on $\mathbb{R}^2$.
  Shown are 
  (thicker black curves) trajectories of~\eqref{eq:odewithrs} started from $(x_0,\T_0)$ for different values of the rate parameter $r=r_c-\delta$, $r=r_c$, and $r=r_c+\delta$,
  (blue dot) the equilibrium regular R-tipping edge state $\eta^+$,  
  the (red) upper $x^{[r_c^+]}$ and (green) lower $x^{[r_c^-]}$ edge tails of $\eta^+$ (note that these contain $\eta^+$),
   (light blue) the rate-dependent family $\Theta^{[r]}$  of time-dependent regular R-tipping thresholds $\Theta^{[r]}(\T)$ defined in~\eqref{eq:Rtipthrfam}, as well as (thinner blue curves) stable 
   and (thinner black curves) unstable manifolds of $\eta^+$ in the future limit system~\eqref{eq:odea+}.
   Note that the projection of $x^{[r^c]}(\T,x_0,\T_0)$ onto the $(x_1,x_2)$ phase plane (not shown in the figure) gives the first component $\mbox{trj}^{[r_c]}(x_0,\T_0)$ in~\eqref{eq:trjr}.
  }
  \label{fig:Rtiptypes}
\end{figure}

\subsection{Reversible, Irreversible and Degenerate R-tipping}
\label{sec:Rtipcasec}

We use the notion of regular R-tipping edge states and their edge tails to classify R-tipping via loss of end-point tracking in nonautonomous system~\eqref{eq:odewithr} or~\eqref{eq:odewithrs} into the following cases. 
\begin{defn} 
 \label{defn:Rtiptypes}
Consider a nonautonomous system~(\ref{eq:odewithrs}) 
with an external input $\Lambda(\T)$ that is asymptotically constant to $\lambda^+$. Suppose the future limit system~\eqref{eq:odea+} has a compact invariant set 
    $\eta^+$ that is not an attractor, and the nonautonomous system~(\ref{eq:odewithrs}) undergoes R-tipping for some $r_1>0$ so that $x^{[r_1]}(\T)\to\eta^+$ as $\T\to +\infty$.
We say this R-tipping is:
\begin{itemize}
    \item[(a)] 
    {\em Non-degenerate}  if  $r_1 = r_c$ is a critical rate,
    $\eta^+$ is a regular R-tipping edge state, the upper and lower edge tails of $\eta^+$ are different: $x^{[r_c^+]}\neq x^{[r_c^-]},$ and each edge tail is a connection from $\eta^+$ to an attractor\,\footnote{See Appendix~\ref{sec:A4} for the definition of an attractor.} for the future limit system~\eqref{eq:odea+}.
    Furthermore, we say non-degenerate R-tipping is
    \begin{itemize}
        \item[$\bullet$]
        {\em Reversible} 
        if each edge tail is a connection from $\eta^+$ to the same attractor.
        \item[$\bullet$]
        {\em Irreversible} if each edge tail is a connection from $\eta^+$ to a different  attractor.
    \end{itemize}
    \item [(b)] 
    {\em Degenerate}  if it is not non-degenerate.
    \end{itemize}
\end{defn}
Examples of different cases of R-tipping are depicted in Fig.~\ref{fig:Rtiptypes}. Only (non-degenerate) irreversible 
and reversible R-tipping,  shown in Fig.~\ref{fig:Rtiptypes}(a) 
and (b), respectively,  are typical in the sense that they are 
generically found at codimension-one in $r$. In other words, 
they are generically found at isolated critical rates $r=r_c$ 
under increasing/decreasing of $r$; see also Remark~\ref{rmk:Rtip}.

Degenerate R-tipping clearly includes many subcases,  even if a regular R-tipping edge state is involved.
One example of degenerate  R-tipping is depicted in 
Fig.~\ref{fig:Rtiptypes}(c), where $\eta^+$ is a regular R-tipping edge state and
the upper and lower edge tails are identical. 
Another example of degenerate R-tipping occurs when 
at least one edge tail is not a connection from $\eta^+$ to an attractor (e.g. an edge tail may connect $\eta^+$ to a saddle, or diverge from $\eta^+$ to infinity; not shown in Fig.~\ref{fig:Rtiptypes}).
Additional examples of degenerate R-tipping involve $\eta^+$ that is not a regular R-tipping edge state.  These include a chaotic saddle  $\eta^+$ with an irregular threshold: a codimension-one stable manifold is not embedded but accumulates on itself, or a repeller  $\eta^+$ of codimension-two (e.g. a source in $\mathbb{R}^2$) or higher that does not have any threshold.
A final example of degenerate R-tipping is the case where there is no critical rate $r_c$: $x^{[r]}(\T)\to\eta^+$ as $\T\to +\infty$ within an interval of $r$.

In the reminder of the paper, we concentrate on R-tipping  due to crossing a regular R-tipping threshold $\Theta^{[r]}(\T)$ anchored at infinity by an equilibrium regular R-tipping edge state $\eta^+$. R-tipping involving more complicated edge states, thresholds that are not regular and quasithresholds, are  discussed in Section~\ref{sec:Conclusions} and left for future study.

\section{Compactification}
\label{sec:compact}

The main obstacle to analysis of genuine nonautonomous R-tipping instabilities in nonautonomous system~(\ref{eq:odewithr})
or~(\ref{eq:odewithrs}) is absence of compact invariant sets such as equilibria, limit cycles or tori. 
We overcome this obstacle by working with  asymptotically constant inputs from Definition~\ref{defn:ac}, $\Lambda(\T)\to\lambda^+$ as $\T\to +\infty$.
Then, the nonautonomous system~(\ref{eq:odewithr})
or~(\ref{eq:odewithrs}) becomes {\em asymptotically autonomous}, and we can define the autonomous {\em future limit system} (\ref{eq:odea+}). More importantly, if 
there is a moving sink $e(\Lambda(\T))$, 
and $e(\Lambda(\T))\to e^+$ as $\T\to +\infty$,
 the future limit system has a hyperbolic sink $e^+$. If there is a moving regular edge state $\eta(\Lambda(\T))$, and $\eta(\Lambda(\T))\to\eta^+$ as $\T\to+\infty$, 
the future limit system has a regular R-tipping edge state $\eta^+$. 
If additionally $\Lambda(\T)\to\lambda^-$ as $\T\to -\infty$,
we can also define the autonomous {\em past limit system}~\eqref{eq:odea-}. If $e(\Lambda(\T))\to e^-$ as $\T\to -\infty$, the past limit system has a hyperbolic sink $e^-$.

Our main idea is to simplify analysis of genuine nonautonomous R-tipping instabilities in system~(\ref{eq:odewithr}) or~(\ref{eq:odewithrs}) by
exploiting the compact invariant sets of interest, such as $e^\pm$ and $\eta^+$, of the autonomous limit systems~\eqref{eq:odea+} and~\eqref{eq:odea-}.
For example, we would like to transform an R-tipping from $e^-$ problem into a heteroclinic $e^--$to$-\eta^+$ orbit problem.
This requires a suitable  compactification of the original nonautonomous system.

In the usual approach~\cite{KloedenRasmussen2011}, the nonautonomous system~(\ref{eq:odewithrs}) 
is augmented with unbounded $\tau\in\mathbb{R}$ as an additional 
dependent variable\,\footnote{By abuse of notation, we use $\T$ to denote both the independent variable and the additional dependent variable.}. This gives the autonomous {\em augmented system}
\begin{align}
\label{eq:odeext0}
\left.
\begin{array}{rl}
 x' &= f(x,\Lambda(\tau))/r\\
\tau'&= 1 
\end{array}\right\},
\end{align} 
defined on $\R^n\times\R$. 
While
the regular R-tipping threshold can nicely be represented in $\R^n\times\R$ as a rate-dependent family of time-dependent subsets of $\mathbb{R}^n$  (see Fig.~\ref{fig:Rtiptypes}):
\begin{equation}
    \label{eq:Rtipthrfam}
    \Theta^{[r]} := \left\{\Theta^{[r]}(\T),\T\right\}_{\T\in\mathbb{R}}
    \subset\mathbb{R}^n\times\mathbb{R},
    \end{equation}
the augmented flow in~\eqref{eq:odeext0} does not contain 
any compact invariant sets
because they only appear as 
$\tau$ tends to positive and negative infinity.

To address this issue, we 
\begin{itemize}
    \item 
    Augment system~(\ref{eq:odewithrs}) with bounded $s\in(-1,1)$ as an additional dependent variable.
    \item 
    Use the compactification technique developed in~\cite{Wieczorek2019compact} to extend the augmented phase space. Specifically, we glue in the limit systems from time infinity ($s=\pm 1$) that carry 
    compact invariant sets such as $e^\pm$ and  $\eta^+$.
\end{itemize}
In short, we require that the additional dependent variable remains within a compact interval.

\subsection{Exponentially Asymptotically Constant Inputs}

Reference~\cite{Wieczorek2019compact} proves existence of a smooth compactification for
nonautonomous system~\eqref{eq:odewithr} or~\eqref{eq:odewithrs} for a wide class of asymptotically constant (possibly non-monotone) inputs\,\footnote{$\Lambda(\tau)$ is denoted $\Gamma(t)$ in~\cite{Wieczorek2019compact}.}  $\Lambda(\tau)$, ranging from super-exponential to sub-logarithmic asymptotic decay (with oscillation). Additionally, it outlines a procedure for constructing suitable examples of time transformation for a given asymptotic decay of $\Lambda(\tau)$.
For simplicity,  we assume here that $\Lambda(\tau)$ decays exponentially, and reformulate the main results from~\cite{Wieczorek2019compact} to account for the presence of the rate parameter $r$. To be precise, 
\begin{defn}
\label{defn:expbac}
We say $\Lambda(\tau)$  is {\em exponentially bi-asymptotically constant} if there is a
{\em decay coefficient} $\rho>0$ such that
\begin{equation}
\label{eq:eas}
\lim_{\tau\to \pm\infty}\frac{\Lambda'(\tau)}{e^{\mp\rho \tau}}\;\;\mbox{exist}.
\end{equation}
We say $\Lambda(\tau)$  is {\em exponentially asymptotically constant} if there is a
 $\rho>0$ such that one of the limits above exists.
\end{defn}

\begin{rmk}
We note that for any bi-asymptotically constant $\Lambda(\tau)$ it is possible to define the slowest rate of exponential approach to a constant as $\T\to\pm\infty$ by
$$
\tilde{\rho}_{\pm} = \lim_{ \T\rightarrow \pm \infty} -\frac{1}{|\T|}\ln \left( \sup_{ u>\T} \|\Lambda'( u)\|\right). 
$$
One can show that $\Lambda$ is exponentially bi-asymptotically constant 
if and only if both of $\tilde{\rho}_{\pm}$ are either positive or $+\infty$. 
Then, a finite decay coefficient $\rho$  in~\eqref{eq:eas} can always be chosen such that $0<\rho<\min(\tilde{\rho}_-,\tilde{\rho}_+)$, and in some special cases\,\footnote{For example, when $\Lambda(\T) \sim C\, e^{\mp\tilde{\rho}\T}\,\T^{n\le0}$ as $\T\to\pm\infty$.} 
such that $0<\rho\le\min(\tilde{\rho}_-,\tilde{\rho}_+)$.
\end{rmk}

\subsection{Autonomous Compactified System}
\label{sec:compautsyst}

Compactification is a three-step process. The first step is 
an $\alpha$-parametrised time transformation that makes the additional dependent variable bounded. Guided by~\cite[Sec.4.2]{Wieczorek2019compact}, we use a transformation designed for exponentially or faster decaying external inputs, and augment the asymptotically autonomous system~(\ref{eq:odewithrs}) with
\begin{equation}
\label{eq:compacttransf}
s = g_{\alpha}(\T) =\tanh\left(\frac{\alpha }{2}\,\T\right) \in (-1,1),
\end{equation}
where $\alpha>0$ is the {\em compactification parameter} that 
is chosen later, in the third step. The inverse is given by
$$
\T = h_{\alpha}(s) = \frac{1}{\alpha}\ln\frac{1+s}{1-s} \in \mathbb{R},
$$
and the augmented component of the vector field is
$$
s' =\alpha\, (1-s^2)/2.
$$
An advantage of the external input time scale $\tau$ is that transformation~\eqref{eq:compacttransf} does not depend on  
the rate parameter $r>0$.
The second step is to make the $s$-interval closed
by including $s=\pm 1$ ($\T=\pm\infty$), and
continuously extend the augmented vector field to $s=\pm 1$.
This gives the autonomous {\em compactified system}
\begin{align}
  \label{eq:odeextbtau}
    \left.
  \begin{array}{rl}
  rx'& = f(x,\Lambda_\alpha(s))\\
  s' &= \alpha(1-s^2)/2
  \end{array}\right\},
\end{align}
with
\begin{align}
  \label{eq:lambda_s+-}
  \Lambda_\alpha(s)&:=
  \left\{
    \begin{array}{rcl}
    \Lambda(h_{\alpha}(s)) &\mbox{for}& s\in(-1,1),\\
    \lambda^+ &\mbox{for}& s = 1,\\
    \lambda^- &\mbox{for}& s = -1,
    \end{array}\right.
\end{align}
that is defined on the {\em extended phase space} $\R^n\times[-1,1]$.
Most importantly, the flow-invariant subspaces 
$$
S^+=\mathbb{R}^n\times\{1\}\;\;\mbox{and}\;\;S^-=\mathbb{R}^n\times\{-1\},
$$
carry the autonomous dynamics and compact invariant sets, such as $e^\pm$ and $\eta^+$,
of the future~\eqref{eq:odea+} and past~\eqref{eq:odea-} limit systems, respectively. 
The third step is to choose the compactification parameter $\alpha$ such that the continuously extended vector field of the compactified system is continuously differentiable ($C^1$-smooth) on $\R^n\times[-1,1]$. This is done in the following proposition.
\begin{proposition}
\label{prop:regular}
Consider a nonautonomous system~\eqref{eq:odewithrs} with exponentially bi-asymptotically constant input $\Lambda(\tau)$ and decay coefficient $\rho>0$. Then, the autonomous compactified  system~\eqref{eq:odeextbtau} is $C^1$-smooth on the extended phase space $\R^n\times[-1,1]$ for any $\alpha\in(0,\rho]$
and all $r >0$.
\end{proposition}

\begin{proof}
For any $r>0$, system~\eqref{eq:odeextbtau} is
a compactification of system~\eqref{eq:odewithrs}.
Thus, we can apply~\cite[Cor.4.1]{Wieczorek2019compact} to~\eqref{eq:odeextbtau} to infer that, for any $\alpha\in(0,\rho]$ and  $r>0$, the compactified system \eqref{eq:odeextbtau} is at least $C^1$-smooth on the compactified phase space $\R^n\times[-1,1]$.
\end{proof}

\subsection{Compactified System as a Singularly Perturbed Fast-Slow System}
\label{sec:csfastslow}

When $0 < r \ll 1$, the compactified system~\eqref{eq:odeextbtau} can be viewed as a singularly perturbed fast-slow system with the small parameter $r$~\cite{Kuehn2015},
where the system time scale $t$ is the {\em fast time}, and the external input time scale $\T=rt$ is the {\em slow time}.
Taking the limit $r\to 0$ in the fast time $t$
in
\begin{align}
\label{eq:odeextb}
    \left.
  \begin{array}{rl}
  \dot{x} &= f(x,\Lambda_\alpha(s))\\
  \dot{s} &= r\,\alpha(1-s^2)/2
  \end{array}\right\},
\end{align}
gives the {\em fast subsystem (the layer problem)}
\begin{equation}
\label{eq:layersyst}
\dot{x} = f(x,\Lambda_\alpha(s)),
\end{equation}
where $s$ becomes a fixed-in-time parameter. Note that this is the frozen system~\eqref{eq:odea} with the input parameter $\lambda = \Lambda_\alpha(s)$.
Taking the limit $r\to 0$ in the slow time  $\tau$ in~\eqref{eq:odeextbtau} gives the {\em slow subsystem (the reduced problem)}
\begin{align}
\label{eq:slowsub}
\left.
\begin{array}{rl}
0  &= f(x,\Lambda_\alpha(s))\\
s' &= \alpha(1-s^2)/2
\end{array}\right\}.
\end{align}
This singular system describes the evolution of $s$ in slow time $\tau$ on
the {\em critical set}
$$
 \tilde{C}^{[0]} = 
\left\{
(x,s)\in \mathbb{R}^n\times[-1,1]
\;:\; 
f(x,\Lambda_\alpha(s)) = 0
\right\},
$$
that consists of all branches of equilibria (critical points) of the fast subsystem~\eqref{eq:layersyst} or~\eqref{eq:odea}.  The critical set $\tilde{C}^{[0]}$ is called the {\em critical manifold} if it is a submanifold of $\mathbb{R}^n\times[-1,1]$. Furthermore, submanifolds of $\tilde{C}^{[0]}$ that consist of hyperbolic equilibria of the fast subsystem~\eqref{eq:layersyst} or~\eqref{eq:odea},
are called {\em normally hyperbolic critical manifolds}~\cite{Fenichel1979,Kuehn2015}.

\subsection{Compact Normally Hyperbolic Critical Manifolds}
\label{sec:NHIM}

The fast-slow viewpoint allows us to represent moving sinks and moving  equilibrium regular edge states as {\em compact normally hyperbolic invariant manifolds} in the extended phase space of the compactified system.
\begin{proposition}
\label{prop:fastslow}
Consider a nonautonomous system~\eqref{eq:odewithrs} 
with exponentially bi-asymptotically constant input $\Lambda(\T)$.
Choose the compactification parameter $\alpha$ that satisfies Proposition~\ref{prop:regular}. 
Consider an interval $I=(\T_-,\T_+)$, let $s_\pm=g_\alpha(\T_\pm)$, and note that  $\T_\pm$ may be $\pm\infty$ in which case $s_\pm=\pm 1$.
Then,
\begin{itemize}
    \item[(a)]
    A moving sink 
    $e(\Lambda(\T))$ on $I=(\T_-,\T_+)$
    in the phase space of the nonautonomous system~\eqref{eq:odewithrs}  can be identified with
    the compact connected normally hyperbolic  attracting critical manifold
    $$
    \tilde{E}_\alpha^{[0]}  =
    \left\{
    \left(e(\Lambda_\alpha(s)),s\right)
    \;:\; s\in[s_-,s_+]
    \right\},
    $$
    in the extended phase space of the compactified system~\eqref{eq:odeextbtau}.
    \item[(b)]
    A moving  equilibrium regular edge state 
    $\eta(\Lambda(\T))$ on $I=(\T_-,\T_+)$ 
    in the phase space of the nonautonomous system~\eqref{eq:odewithrs} 
    can be identified with the compact connected normally hyperbolic critical manifold 
    $$
    \tilde{H}_\alpha^{[0]}  =
    \left\{ \left(\eta(\Lambda_\alpha(s)),s\right) \;:\; s\in[s_-,s_+] \right\}, 
    $$
    in the extended phase space of the compactified system~\eqref{eq:odeextbtau}. $\tilde{H}_\alpha^{[0]}$ is normally repelling if $x\in\mathbb{R}$, or of saddle type with one unstable dimension if $x\in\mathbb{R}^{n\ge 2}$.
\end{itemize}
\end{proposition}
\begin{rmk}
Proposition~\ref{prop:fastslow} allows us to apply Fenichel's theorem~\cite[Thm 9.1]{Fenichel1979} to the compactified system~\eqref{eq:odeextbtau} to give criteria for tracking moving sinks and moving equilibrium regular thresholds in the nonautonomous system~\eqref{eq:odewithrs} in Section~\ref{sec:TrackingProof}.
\end{rmk}

\noindent
{\em Proof of Proposition~\ref{prop:fastslow}.}
(a) Note that $\tilde{E}_\alpha^{[0]}$ is a graph over $s$, and
$$
\frac{\textrm d}{{\textrm d}s}\,e(\Lambda_\alpha(s))=
\frac{\textrm d}{{\textrm d}\lambda}e(\lambda)\,\frac{\textrm d}{{\textrm d}s}\Lambda_\alpha(s).
$$
It then follows from Definition~\ref{def:ms}(a) of a moving sink on $I$, and from Prop.~\ref{prop:regular}, 
that  $\tilde{E}_\alpha^{[0]}$ is at least $C^1$-smooth in $s$ on $[s_-, s_+]$. 
For any fixed $s^*\in[s_-,s_+]$, $\tilde{E}_\alpha^{[0]}$ consists of
an equilibrium (a critical point) of the fast subsystem~\eqref{eq:layersyst}, which is exponentially stable (hyperbolic) within, and neutrally stable transverse to  $\{s=s^*\}$. Hence, $\tilde{E}_\alpha^{[0]}$
is a connected attracting normally hyperbolic critical manifold. It is compact because it is a closed and bounded subset of $\mathbb{R}^n\times[-1,1]$.

\noindent
(b)
Note that $\tilde{H}_\alpha^{[0]}$ is a graph over $s$, and
$$
\frac{\textrm d}{{\textrm d}s}\,\eta(\Lambda_\alpha(s))=
\frac{\textrm d}{{\textrm d}\lambda} \eta(\lambda)\,\frac{\textrm d}{{\textrm d}s}\Lambda_\alpha(s).
$$
It then follows from Definition~\ref{defn:mth}(b) of a moving regular edge state on $I$, and from similar arguments to (a), that $\tilde{H}_\alpha^{[0]}$  is a compact connected normally hyperbolic invariant critical manifold. Normal stability of 
$\tilde{H}_\alpha^{[0]}$ follows from 
Definition~\ref{defn:edgestate} of a regular edge state.
\qed

\subsection{Compactified System Dynamics}
\label{sec:csd}

In this section, we discuss the stability of hyperbolic sinks $e^\pm$ and equilibrium regular R-tipping edge states $\eta^+$ from the limit systems when embedded in the extended phase space of the compactified system~\eqref{eq:odeextbtau}. Additionally, we extrapolate the dynamical structure from these states into the new dependent variable $s$ and characterise their stable and unstable invariant manifolds.
In Section~\ref{sec:Conclusions}, we discuss extensions of some of the results below to non-equilibrium attractors and non-equilibrium regular edge states.

\begin{proposition}
\label{prop:csdyn}
Consider a nonautonomous system~(\ref{eq:odewithrs}) with exponentially bi-asymptotically constant input $\Lambda(\T)$ and decay coefficient $\rho>0$. Choose any compactification parameter $\alpha<\rho$.
\begin{itemize}
    \item[(a)] 
    If $e^+$ is a  hyperbolic sink for the future limit system~\eqref{eq:odea+}, then
    $$
    \tilde{e}^+ =(e^+,1)\in S^+,
    $$
    is also a hyperbolic sink when considered in the extended phase space of the compactified system~\eqref{eq:odeextbtau}. 
    The additional eigenvector of $\tilde{e}^+$, denoted $v_+$,
    exists and is normal to the invariant subspace $S^+$ for any
    $\alpha\in(0,\rho)$ and all  $r>0$. Furthermore, $v_+$ is the leading eigenvector of $\tilde{e}^+$ for any $\alpha\in\left(0,\min\{\rho,-\mathrm{Re}(l_1)/r\}\right)$ 
    and all $r>0$, where $l_1$ is the leading eigenvalue of $e^+$ in 
    the future limit system~\eqref{eq:odea+}.
    \item[(b)] 
    If $\eta^+$ is an equilibrium regular R-tipping edge state, then 
    $$
    \tilde{\eta}^+ =(\eta^+,1)\in S^+,
    $$
    is a hyperbolic saddle
    with a codimension-one stable manifold $W_\alpha^{s,[r]}(\tilde{\eta}^+)$,  a codimension-one embedded orientable  local stable manifold 
    $W^{s,[r]}_{\alpha,loc}(\tilde{\eta}^+) \subseteq W_\alpha^{s,[r]}(\tilde{\eta}^+)$,
    and a one-dimensional unstable manifold $W^u(\tilde{\eta}^+)$, when considered in the extended phase space of the compactified system~\eqref{eq:odeextbtau}. 
    The additional eigenvector of $\tilde{\eta}^+$ is normal to the invariant subspace $S^+$ for any
    $\alpha\in(0,\rho)$ and all  $r>0$.
    \item[(c)] 
    If $e^-$ is a hyperbolic sink for the past limit system~\eqref{eq:odea-}, then
    $$
    \tilde{e}^- =(e^-,-1)\in S^-,
    $$
    is a hyperbolic saddle  with a one-dimensional unstable manifold $W_\alpha^{u,[r]}(\tilde{e}^-)$ when considered in the extended phase space of the compactified system~\eqref{eq:odeextbtau}.
    The additional eigenvector of $\tilde{e}^-$ is normal to the invariant subspace $S^-$ for any
    $\alpha\in(0,\rho)$ and all  $r>0$.
\end{itemize}
\end{proposition}
\begin{rmk}
Note that the shape and relative position of invariant manifolds $W_\alpha^{s,[r]}(\tilde{\eta}^{+})$ and $W_\alpha^{u,[r]}(\tilde{e}^-)$ will typically change with the rate parameter $r$, but these manifolds are guaranteed to respectively meet the invariant subspaces $S^+$ and $S^-$ orthogonally for any $r>0$ if we choose the compactification parameter $\alpha\in(0,\rho)$. The invariant manifold $W^u(\tilde{\eta}^+)$ is independent of $r$ and $\alpha$.
\end{rmk}
\noindent
{\em Proof of Proposition~\ref{prop:csdyn}.}
Note that we impose the limit $0 < \alpha\le \rho$ to ensure the compactified system~\eqref{eq:odeextbtau} is at least $C^1$-smooth; this follows from Prop.~\ref{prop:regular}.
The Jacobian for the compactified system is
\begin{equation}
    \label{eq:jacobian}
J=\left(
\begin{array}{cc}
\frac{1}{r}\left(\frac{\partial f}{\partial x}\right)_{n\times n} & \frac{1}{r}\left(\frac{\partial f}{\partial \Lambda}\,\frac{{\textrm d} \Lambda_\alpha}{{\textrm d} s}\right)_{n\times 1} \\
(0)_{1\times n} & - \alpha s
\end{array}
\right),
\end{equation}
where the subscripts indicate the size of the matrix components of $J$.
Consider linear stability of equilibria $\tilde{e}^\pm$ and $\tilde{\eta}^+$ in the compactified system~\eqref{eq:odeextbtau} on the time scale $\T$.
\\
(a) Equilibrium $\tilde{e}^+$ is a hyperbolic sink.
There are $n$ eigenvalues $q_i = l_i/r$ within $S^+$ that satisfy ${\mathrm Re}(q_n) \le \ldots \le \mathrm{Re}(q_1) < 0$, where $l_i$ are the eigenvalues of ${e}^+$ in the future limit system~\eqref{eq:odea+}, and $S^+$ itself is exponentially attracting, adding one additional negative eigenvalue $q_+ =-\alpha$.
It follows from the structure of the Jacobian~\eqref{eq:jacobian} that the 
additional eigenvector, denoted $v_+$, exists for all $r>0$ if the top $n$ elements in the last column of $J$ are zero
\begin{equation}
\label{eq:orthog}
\frac{\partial f}{\partial \Lambda}({e}^+)\, 
\frac{d\Lambda_\alpha}{ds}(s=1) 
= 
\frac{\partial f}{\partial \Lambda}({e}^+)\, \lim_{\tau\to + \infty}\,\frac{\Lambda'(\tau)}{g'_{\alpha}(\tau)} 
= 0,
\end{equation}
and $v_+$ 
is normal to $S^+$ if and only if~\eqref{eq:orthog} holds.
Noting that $g'_{\alpha}(\tau) \sim 2\alpha \, e^{-\alpha\tau}$ as 
$\tau\to +\infty$, and that $\Lambda(\T)$ decays exponentially with the decay coefficient $\rho>0$, we obtain
$$
\lim_{\tau\to + \infty}\,\frac{\Lambda'(\tau)}{g'_{\alpha}(\tau)} =
\left(\lim_{\tau\to +\infty}\frac{\Lambda'(\tau)}{e^{-\rho \tau}}\right)\,
\left(\lim_{\tau\to +\infty}\frac{e^{-\rho\tau}}{g'_{\alpha}(\tau)}\right)
=
\frac{1}{2\alpha}
\left(
\lim_{\tau\to +\infty}\frac{\Lambda'(\tau)}{e^{-\rho \tau}}
\right)\,
\left(
\lim_{\tau\to +\infty}e^{-(\rho-\alpha) \tau}
\right),
$$
implying that $v_+$ exists  and is normal to $S^+$ 
for any $0<\alpha<\rho$ and all  $r>0$.
Finally, $v_+$ is the leading eigenvector for all $r>0$ if
it exists for all $r>0$, meaning that $0 < \alpha < \rho$, 
and if $-q_+ < -\mathrm{Re}(q_1)$.  
Hence the condition $0<\alpha < \min\{\rho,-\mathrm{Re}(l_1)/r\}$.

\noindent
(b) Equilibrium $\tilde{\eta}^+$ is a hyperbolic saddle with
$n$-dimensional stable eigenspace  $E_\alpha^s(\tilde{\eta}^+)$. 
This is because  $\tilde{\eta}^+$ is either a hyperbolic source ($n=1$) or a hyperbolic saddle with one unstable eigendirection ($n\ge 2$) within $S^+$ by Definition~\ref{defn:edgestate}, and $S^+$ itself is exponentially attracting, adding one (additional) negative eigenvalue $q_+ =-\alpha$.
Note that the additional (generalised) eigenvector is transverse to $S^+$ for all $r>0$. Thus, the stable eigenspace $E_\alpha^s(\tilde{\eta}^+)$ is transverse to $S^+$ for all $r>0$.
It then follows from the stable manifold theorem that, for any $r>0$, there is a unique $C^1$-smooth codimension-one  stable manifold $W_\alpha^{s,[r]}(\tilde{\eta}^+)$ 
that is tangent to $E_\alpha^s(\tilde{\eta}^+)$ at $\tilde{\eta}^+$.
$W_\alpha^{s,[r]}(\tilde{\eta}^+)$ depends on the rate parameter $r$ because the  vector field in~\eqref{eq:odeextbtau} depends on $r$.
Consider a codimension-one forward-invariant local stable manifold $W_{\alpha,loc}^{s,[r]}(\tilde{\eta}^+)$ defined for 
 $s\in(s_0,1]$ with a suitably chosen $s_0$.
It then follows from Definition~\ref{defn:edgestate} of a regular edge state that $W_{\alpha,loc}^{s,[r]}(\tilde{\eta}^+) \cap S^+$ is a codimension-one embedded orientable forward-invariant local stable manifold of $\eta^+$ within $S^+\subseteq \mathbb{R}^n$. 
Since $W_{\alpha,loc}^{s,[r]}(\tilde{\eta}^+)$ intersects $S^+$ transversely, there is an $s_0\in[-1,1)$ such that $W_{\alpha,loc}^{s,[r]}(\tilde{\eta}^+)$ is a graph over $s$ on $(s_0,1]$.
Thus, 
the embedding and orientability properties carry over from $S^+$ to the entire $W_{\alpha,loc}^{s,[r]}(\tilde{\eta}^+)$.
The condition for the stable eigenspace $E_\alpha^s(\tilde{\eta}^{+})$ to be normal to $S^+$ follows from (a). 

\noindent
(c) For any $r>0$, equilibrium $\tilde{e}^{-}$ is a  hyperbolic saddle with 
one-dimensional unstable eigenspace $E_\alpha^u(\tilde{e}^{-})$. This is because $\tilde{e}^{-}$ is a hyperbolic sink within $S^-$, and $S^-$ itself is exponentially repelling, adding one and the only unstable eigendirection with positive eigenvalue $q_- = \alpha$. For any $r>0$, existence of the one-dimensional  unstable manifold $W_\alpha^{u,[r]}(\tilde{e}^-)$ follows from the unstable manifold theorem. 
$W_\alpha^{u,[r]}(\tilde{e}^-)$ depends on  the rate parameter $r$ because the compactified vector field in~\eqref{eq:odeextbtau} depends on $r$.
The condition for the unstable eigendirection $E_\alpha^u(\tilde{e}^{-})$ to be normal to $S^-$ follows from a similar argument to (a). 
\qed

\subsection{Relating Nonautonomous and Compactified System Dynamics}
\label{sec:compactdyns}

We now examine the relationship between:
\begin{itemize}
    \item[(i)]
    Solutions, regular R-tipping thresholds and edge tails in the  nonautonomous system~\eqref{eq:odewithrs},  and 
    \item[(ii)]
    Equilibria $\tilde{e}^-$ and $\tilde{\eta}^+$ as well as their invariant manifolds
    in the autonomous compactified system~\eqref{eq:odeextbtau}.
    \end{itemize}

First, we relate the local pullback attractor $x^{[r]}(\T,e^-)$
to the rate-dependent unstable manifold of $\tilde{e}^-$, the time and
rate dependent  R-tipping threshold $\Theta^{[r]}(\T)$ anchored at infinity by an equilibrium regular R-tipping edge state $\eta^+$ to the rate-dependent
local stable manifold of $\tilde{\eta}^+$,  and associate each edge tail
$x^{[r_c^+]}$ and  $x^{[r_c^-]}$  of $\eta^+$ to a branch of the unstable manifold of $\tilde{\eta}^+$.
\begin{proposition}
\label{prop:invsete-}
Consider a nonautonomous system~(\ref{eq:odewithrs}) with exponentially bi-asymptotically constant input $\Lambda(\T)$ and decay coefficient $\rho$. Choose any compactification parameter $\alpha\in (0,\rho)$.
\begin{itemize}
    \item[(a)]
    Suppose the past limit system~\eqref{eq:odea-} has
    a  sink $e^{-}$. Then,
    \begin{itemize}
        \item[$\bullet$]
        There is a $\tau_0$ such that, for any $r>0$ and all $\T < \tau_0$, there exists a unique local pullback attractor
        $x^{[r]}(\T,e^-)$ in nonautonomous system~\eqref{eq:odewithrs}.
        Note that $\tau_0$ may be $+\infty$.
        \item[$\bullet$]
        For any $r>0$, the local pullback attractor
        $x^{[r]}(\T,e^-)$ in nonautonomous system~\eqref{eq:odewithrs} 
        can be identified with sections of the one-dimensional  unstable manifold 
        $$
        W_\alpha^{u,[r]}(\tilde{e}^{-}) 
        \supset 
        \left\{(x,s)\;\;:\;\; x {\sw \,=\,}  x^{[r]}(\T,e^-),~s=g_{\alpha}(\T) \right\}_{\T<\tau_0},
        $$
        of the saddle $\tilde{e}^- =(e^-,-1)$ in the extended phase space of the compactified system~\eqref{eq:odeextbtau}.
        \end{itemize}
    \item[(b)]
    Suppose the future limit system~\eqref{eq:odea+} has an equilibrium regular  R-tipping edge state $\eta^+$. Then,
    \begin{itemize}
        \item[$\bullet$] 
        There is a $\tau_0$ such that, for any $r>0$ and all $\T>\tau_0$, there exists an R-tipping threshold $\Theta^{[r]}(\T)$ anchored  at infinity by $\eta^+$ in nonautonomous system~\eqref{eq:odewithrs}.
         Note that $\tau_0$ may be $-\infty$.
        \item[$\bullet$]
        For any $r>0$, the R-tipping threshold $\Theta^{[r]}(\T)$ in nonautonomous system~\eqref{eq:odewithrs} can be identified with  sections of the  codimension-one stable manifold
        \begin{align}
        \label{eq:compthr}
        W_\alpha^{s,[r]}(\tilde{\eta}^+) \supset \tilde{\Theta}_\alpha^{[r]}:=
        \left\{ \left(x,s\right)~:~x\in \Theta^{[r]}(\T),~s=g_{\alpha}(\T) \right\}_{\T > \tau_0},
        \end{align}
        of the saddle $\tilde{\eta}^+ =(\eta^+,1)$ in the extended phase space of the compactified system~\eqref{eq:odeextbtau}.
        \item[$\bullet$]
        Each edge tail of $\eta^+$ embedded in the compactified phase space of~\eqref{eq:odeextbtau}, namely
        \begin{equation}
        \label{eq:etsembedded}
        \tilde{x}^{[r_c^+]}=
        \left\{
        (x,1)~:~x\in x^{[r_c^+]}
        \right\}
        \quad\mbox{and}\quad
        \tilde{x}^{[r_c^-]}=
        \left\{
        (x,1)~:~x\in x^{[r_c^-]}
        \right\},
        \end{equation}
        contains one  branch of the  unstable manifold $W^{u}(\tilde{\eta}^+)$ of the saddle $\tilde{\eta}^+ =(\eta^+,1)$.
    \end{itemize}
\end{itemize}
\end{proposition}

\begin{rmk}
These relations between nonautonomous and compactified system dynamics are the main advantages of the compactification. They show that the temporal shape of the external input $\Lambda(\T)$ and the magnitude of the rate parameter $r>0$ are in a certain sense `encoded' in the geometric shape of the invariant manifolds $W_\alpha^{u,[r]}(\tilde{e}^-)$ and $W_\alpha^{s,[r]}(\tilde{\eta}^{+})$ for the autonomous compactified system~\eqref{eq:odeextbtau}. This observation allows us to use existing numerical methods from~\cite{Krauskopf2006} to compute families of regular R-tipping thresholds in low-dimensional nonautonomous systems~\eqref{eq:odewithrs} as local stable manifolds of saddles $\tilde{\eta}^+$ in the extended phase space of the compactified system~\eqref{eq:odeextbtau}.
\end{rmk}
\noindent
{\em Proof of Proposition~\ref{prop:invsete-}.}
The assumption of $\alpha$ means that the conclusion of Proposition~\ref{prop:regular} holds.\\
(a) In the nonautonomous system~\eqref{eq:odewithrs}, existence of a unique local pullback 
point attractor $x^{[r]}(\T,e^-)$ that limits to 
$e^-$ as $\T\to +\infty$ for any $r>0$ follows from~\cite[Thm.~2.2]{Ashwin2016}.
In the compactified system~\eqref{eq:odeextbtau}, existence of a unique  one-dimensional unstable manifold $W_\alpha^{u,[r]}(\tilde{e}^-)$
for any $r>0$ follows from Proposition~\ref{prop:csdyn}(c). 
These may exist for all $\T\in\mathbb{R}$ and $s\in(-1,1)$, 
respectively, but this is not guaranteed.
Noting that $\left\{ (x^{[r]}(\T),g_{\alpha}(\T))\;:\;\T<\tau_0 \right\}$ is
the trajectory of the compactified system that corresponds to a solution $x^{[r]}(\T)$ 
of the nonautonomous system gives the result.

\noindent
(b) 
 We prove existence of  a regular R-tipping threshold anchored 
 at infinity by an equilibrium regular R-tipping edge state $\eta^+$ by construction, 
using sections of a suitably chosen subset of $W^{s,[r]}_\alpha(\tilde{\eta}^{+})$ at fixed values of $s$.
Existence of a codimension-one embedded orientable forward-invariant local stable manifold $W_{\alpha,loc}^{s,[r]}(\tilde{\eta}^+) \subseteq W_{\alpha}^{s,[r]}(\tilde{\eta}^+)$ that 
is a graph over $s$ for  $s\in(s_0,1]$ follows from Proposition~\ref{prop:csdyn}(b).  
Keeping in mind that $s = g_\alpha(\T)$, and setting $\tau_0 = h_\alpha(s_0)$, we construct
\begin{equation}
\label{eq:thetaWs}
\Theta^{[r]}(\T) :=\{x\;:\;(x,s)\in W_{\alpha,loc}^{s,[r]}(\tilde{\eta}^+)\} \subset \mathbb{R}^n,
\end{equation}
for any $r>0$ and all  $\T \in (\tau_0,+\infty)$. Note that $\tau_0$ is $-\infty$ if $s_0=-1$.
Such $\Theta^{[r]}(\T)$ is a codimension-one embedded orientable forward-invariant nonautonomous 
set by construction, and has the property~\eqref{eq:etastable}.
Thus, $\Theta^{[r]}(\T)$ is a regular R-tipping threshold. Note that $\Theta^{[r]}(\T)$ 
is not unique in the sense that there is a different $\Theta^{[r]}(\T)$ for every 
different codimension-one forward-invariant subset of $W_{\alpha,loc}^{s,[r]}(\tilde{\eta}^+)$.  Relation~\eqref{eq:compthr} follows from construction of $\Theta^{[r]}(\T)$ in~\eqref{eq:thetaWs}.
To prove the last bullet point in (b), recall from Definition~\ref{defn:edgetails} that each edge tail of $\eta^+$ contains a trajectory
of the future limit system~\eqref{eq:odea+} that limits to $\eta^+$ in backwards time and does not depend on $r$ or $\alpha$. It follows from Proposition~\ref{prop:csdyn}(b) 
that $\tilde{\eta}^+$ is a hyperbolic saddle with one-dimensional unstable manifold $W^{u}(\tilde{\eta}^+)\subset S^+$. This means that this unstable manifold contains precisely two trajectories (the branches of $W^{u}(\tilde{\eta}^+)$) and hence each edge tail must contain one of these.
\qed\\

Next, we state three relations between solutions $x^{[r]}(\T)$ of the nonautonomous system~\eqref{eq:odewithrs} for $r$ 
on different sides of a critical rate $r_c$, and the upper $x^{[r_c^+]}$ and lower $x^{[r_c^-]}$  edge tails of an equilibrium regular R-tipping edge state $\eta^+$. 
\begin{proposition}
\label{prop:edgetails}
Consider a solution $x^{[r]}(\T)$ to a nonautonomous system~\eqref{eq:odewithrs}
with an external input $\Lambda(\T)$ that is asymptotically constant to $\lambda^+$.
Suppose there is a regular R-tipping threshold $\Theta^{[r]}(\T)$ anchored at infinity by an equilibrium regular  R-tipping edge state $\eta^+$, and there is R-tipping for some critical rate $r=r_c>0$ so that
$x^{[r_c]}(\T) \to \eta^+\;\;\mbox{as}\;\; \T\to +\infty$. 
\begin{itemize}
    \item [(a)]
    If there is a $\delta>0$ such that $x^{[r]}(\T)$ lies on different sides of  $\Theta^{[r_c]}(\T)$ for $r\in(r_c-\delta,r_c)$ and $r\in(r_c, r_c+\delta)$,
    then the upper and lower edge tails of $\eta^+$ are different: $x^{[r_c^+]}\neq x^{[r_c^-]}$.
    \item[(b)]
    If each edge tail of $\eta^+$ is a connection from $\eta^+$ to an attractor, then there is a $\delta>0$ such that  $x^{[r]}(\T)$ converges to an attractor for $0<|r-r_c|<\delta$.
    \item[(c)]
    If each edge tail of $\eta^+$ is a different connection from $\eta^+$ to a (possibly different) attractor, then there is a $\delta>0$ such that $x^{[r]}(\T)$ lies on different sides of $\Theta^{[r_c]}(\T)$ and converges to the corresponding attractor for $r\in(r_c-\delta,r_c)$ and $r\in(r_c,r_c+\delta)$. 
\end{itemize}
\end{proposition}
\begin{rmk}
\label{rmk:edgetails}
Note that different edge tails of $\eta^+$ do not imply that each edge tail is a different connection from $\eta^+$ to an attractor. For example, different composite edge tails, each of which consists of different connected trajectories or components, may have a common first component that connects $\eta^+$ to a saddle,  and a different second component that continues away from this saddle. Another example are different non-composite edge tails that diverge from $\eta^+$ to infinity.
\end{rmk}

\noindent
{\em Proof of Proposition~\ref{prop:edgetails}.}
Choose the compactification parameter $\alpha$  that satisfies Proposition~\ref{prop:regular}. Recall from Section~\ref{sec:compautsyst} that $s(\T)=g_\alpha(\T)$, and use
$$
\tilde{x}_\alpha^{[r]}(\T) = 
\left( 
{x}^{[r]}(\T), s(\T)
\right),
$$
to denote the solution of~\eqref{eq:odeextbtau} corresponding to a solution
$x^{[r]}(\T)$ of the nonautonmous system~\eqref{eq:odewithrs}
with a fixed $r$, and refer to $\tilde{x}^{[r_c^+]}$ and  $\tilde{x}^{[r_c^-]}$ from~\eqref{eq:etsembedded} as {\em embeded edge tails}.
Recall from Proposition~\ref{prop:invsete-}(b) that $W_\alpha^{s,[r_c]}(\tilde{\eta}^{+})$ 
contains a family of regular R-tipping thresholds $\Theta^{[r]}(\T)$, and each embedded edge tail contains one  branch of the unstable manifold $W^{u}(\tilde{\eta}^+)$.

For (a), assume that $x^{[r]}(\T)$ is on different sides of $\Theta^{[r]}(\T)$ for $r\in(r_c-\delta,r_c)$ and $r\in(r_c+\delta,r_c)$ in the nonautonomous system~\eqref{eq:odewithrs}, and consider where the corresponding
$\tilde{x}_\alpha^{[r]}(\T)$ intersects the two branches of $W^u(\tilde{\eta}^+)$ in the extended phase space of the compactified system~\eqref{eq:odeextbtau}. This intersection changes sides of $W_\alpha^{s,[r_c]}(\tilde{\eta}^{+})$ as $r$ passes through $r_c$.
Thus, each embedded edge tail contains a different branch of $W^u(\tilde{\eta}^+)$, meaning that the edge tails
$x^{[r_c^+]}$ and $x^{[r_c^-]}$ are different.

For (b), it follows from~\cite[Prop.3.1]{Wieczorek2019compact} that 
if $a^+$ is an attractor for the future limit system~\eqref{eq:odea+}, 
then $\tilde{a}^+ =\{(x,1):x\in a^+\}\subset S^+$ is an attractor for 
the compactified system~\eqref{eq:odeextbtau}.
Thus, the assumption that each edge tail is a connection from $\eta^+$ 
to an attractor implies that each embedded edge tail lies in the basin 
of attraction of an attractor. This, in turn, implies  that each  section that transversely intersects an embedded edge tail has an open neighbourhood  that 
lies in the basin of attraction of an attractor. We choose $\delta>0$ small 
enough so that $\tilde{x}_\alpha^{[r]}(\T)$ enters this neighbourhood
for all $0<|r-r_c|<\delta$. This implies that the corresponding $x^{[r]}(\T)$ 
converges to an attractor for all $0<|r-r_c|<\delta$.

In (c), the assumption that each edge tail of $\eta^+$ is a different connection from $\eta^+$ to a possibly different attractor implies that each embedded edge tail contains a different branch of  $W^u(\tilde{\eta}^+)$ and thus lies on a different sides of $W_\alpha^{s,[r_c]}(\tilde{\eta}^{+})$. This, in turn, implies that
${x}^{[r]}(\T)$ lies on different sides of $\Theta^{[r_c]}(\T)$ for $r\in(r_c-\delta,r_c)$ and $r\in(r_c,r_c+\delta)$.
It follows from (b) that $x^{[r]}(\T)$ converges to the corresponding attractor for  $0<|r-r_c|<\delta$. 
\qed

\section{Criteria for Tracking and R-tipping  with Regular Thresholds}
\label{sec:gentestcrit}

In this section, we give the main results on R-tipping via loss of end-point tracking in nonautonomous system~\eqref{eq:odewithr} or~\eqref{eq:odewithrs} with  asymptotically
constant inputs $\Lambda$. Our focus is on non-degenerate (reversible and irreversible) cases of R-tipping, due to crossing regular R-tipping thresholds anchored at infinity by an equlibrium regular R-tipping edge state. Specifically, we use the compactification technique together with  relations between nonautonomous~\eqref{eq:odewithrs} and compactified~\eqref{eq:odeextbtau} system dynamics given in Section~\ref{sec:compact} to:

\begin{itemize}
    \item
    Give rigorous testable criteria for tracking of moving sinks, and tracking of moving regular thresholds in arbitrary dimension in Section~\ref{sec:TrackingProof}.
    \item
     Use the concept of threshold instability to generalise  sufficient conditions from~\cite{Ashwin2016} for the occurrence of
     irreversible  R-tipping for moving sinks on $I=\R$ in one dimension  to 
     different cases of R-tipping for moving sinks on $I=\R$ in arbitrary dimension  in Section~\ref{sec:Rtippingcriteria}.
    \item
     Relax the assumption of moving sinks on $I=\R$ and associate different R-tipping in~\eqref{eq:odewithrs} with a connecting (heteroclinic) orbit in~\eqref{eq:odeextbtau}.
    Give necessary and sufficient conditions for the occurrence of non-degenerate R-tipping in~\eqref{eq:odewithrs} in terms of non-degeneracy criteria for connecting (heteroclinic) orbits in~\eqref{eq:odeextbtau}.
    Use this result to give general methods for computing  critical rates for R-tipping in arbitrary dimension in Section~\ref{sec:computing}.
\end{itemize}

\subsection{Criteria for Tracking Moving Sinks and Moving Regular Thresholds}
\label{sec:TrackingProof}

We now demonstrate that a moving sink will be tracked by a 
solution of the nonautonomous system if the rate parameter 
$r$ is small enough. We also demonstrate that a (normally repelling) moving regular 
threshold will be tracked by an R-tipping threshold if $r$ is 
small enough. 
To prove these results, we consider the compactified system~\eqref{eq:odeextbtau} as a singularly perturbed fast-slow system.
 This allows us to use results from geometric singular perturbation 
theory on the compactified system~\eqref{eq:odeextbtau} with small parameter 
$0< r \ll 1$ from Section~\ref{sec:NHIM}, together with relations between nonautonomous~\eqref{eq:odewithrs} and compactified~\eqref{eq:odeextbtau} 
system dynamics from Section~\ref{sec:compactdyns}.

The first result states a sufficient condition that moving sinks are tracked. It reformulates~\cite[Lemma 2.3]{Ashwin2016}
for more general external inputs 
$\Lambda(\tau)$ that are arbitrary dimensional and not necessarily bounded between $\lambda^-$ and $\lambda^+$, and for solutions $x^{[r]}(\T,x_0,\T_0)$ that are not necessarily pullback attractors. The stronger assumption of exponentially asymptotically constant $\Lambda(\tau)$ is made here for simplicity, and the results can easily be extended to any asymptotically constant $\Lambda(\tau)$ by using~\cite[Definition 2.2]{Wieczorek2019compact}; see also Section~\ref{sec:Conclusions}.
\begin{theorem}
\label{thm:tracking}
 Consider a nonautonomous system~(\ref{eq:odewithrs}) with an  input $\Lambda(\T)$ that is exponentially asymptotically constant
 to $\lambda^+$.  
 Suppose there is a moving sink $e(\Lambda(\T))$ on $I=(\T_0,+\infty)$, and recall that $e(\Lambda(\T)) \to e^+$ as $\T \to +\infty$. 
 Fix any $\delta>0$.
\begin{itemize}
\item[(a)] 
For any solution $x^{[r]}(\T,x_0,\T_0)$ with $x_0$ in the basin of attraction of $e(\Lambda(\T_0))$, there is an $r^*(\delta)>0$ and a $\T^*(r,\delta)\ge \T_0$, such that $x^{[r]}(\T,x_0,\T_0)$ 
 {\em $\delta$-close and end-point tracks} the moving sink $e(\Lambda(\T))$ on $(\T^*,+\infty)\subseteq I$ for any $r\in(0,r^*)$.
\item[(b)]   
Suppose in addition that $\Lambda(\T)$ is exponentially 
bi-asymptotically constant, $e(\Lambda(\T))$ is a moving sink on $I=\R$, and recall that  $e(\Lambda(\T))\to e^-$ as $\T \to -\infty$. Then, there is an $r^*(\delta)>0$ such that
\begin{itemize}
\item[$\bullet$]
 The unique local pullback attractor $x^{[r]}(\T,e^-)$ from Proposition~\ref{prop:invsete-}(a)  exists for   any $r\in(0,r^*)$ and all $\T\in\mathbb{R}$.
 \item[$\bullet$]
 The local pullback attractor $x^{[r]}(\T,e^-)$ 
{\em $\delta$-close and end-point tracks} the moving sink $e(\Lambda(\T))$ on $I=\R$ for any $r\in(0,r^*)$. 
\end{itemize}
\end{itemize}
\end{theorem}
\begin{rmk}
The compactification from Section~\ref{sec:compact} allows us to prove Theorem~\ref{thm:tracking} using Fenichel's theorem~\cite[Thm 9.1]{Fenichel1979} on persistence of compact normally hyperbolic invariant  manifolds.
Alternative approaches that may give results similar to Theorem~\ref{thm:tracking}(b) include:
\cite[Theorem III.1]{Alkhayuon2018} which uses results from~\cite{Aulbach2006}, 
\cite[Lemma 2.3]{Ashwin2016} which uses results from~\cite{Eldering} on persistence of non-compact normally hyperbolic invariant manifolds in bounded geometry, 
\cite{Kuehn2021} which uses a Melnikov integral approach, and
\cite{Longo2021} which uses the hull construction,   
although the last two examples are for one-dimensional (scalar) systems.
\end{rmk}

\noindent
{\em Proof of Theorem\ref{thm:tracking}.}
Choose the compactification parameter $\alpha$  that satisfies Proposition~\ref{prop:regular} for any $r>0$.\\
\noindent
(a)
Recall from Proposition~\ref{prop:fastslow}(a) that the moving sink $e(\Lambda(\T))$ on $I=(\T_0,+\infty)$ corresponds to a  one-dimensional compact connected attracting normally hyperbolic critical manifold
$$
\tilde{E}_\alpha^{[0]}  =
\left\{
\left(e(\Lambda_\alpha(s)),s\right)
\;:\; s\in[s_0,1]
\right\},
$$
in the extended phase space of the  compactified
system~\eqref{eq:odeextbtau},
where $s_0=g_\alpha(\tau_0)$. 
It then follows from~\cite{Fenichel1979} that,
for $r>0$ sufficiently small, $\tilde{E}_\alpha^{[0]}$ perturbs to a one-dimensional connected attracting normally hyperbolic  invariant manifold $\tilde{E}_\alpha^{[r]}$ that lies $C^1$-close to $\tilde{E}_\alpha^{[0]}$ and, as $\tilde{e}^+$ is isolated, contains $\tilde{e}^+$.
Thus, for any $\delta>0$
and initial condition $(x_0,s_0)$ in the basin of attraction of $e(\Lambda_\alpha(s_0))$,
we can choose  $r^*$ small enough so that:
(i) $\tilde{E}_\alpha^{[r]}$ is normally hyperbolic ($v_+$ from Proposition~\ref{prop:csdyn}(a) is the leading eigenvector), attracting, and lies $\delta$-close to $\tilde{E}_\alpha^{[0]}$ for  any $r\in(0,r^*)$
and all $s\in[s_0,1]$, and
(ii) $(x_0,s_0)$ is in the basin of attraction of $\tilde{E}_\alpha^{[r]}$ for any $r\in(0,r^*)$.
Thus, $x^{[r]}(\T,x_0,\T_0)$ will be attracted to the solution
of~\eqref{eq:odewithrs} corresponding to $\tilde{E}_\alpha^{[r]}$, 
and $\delta$-close and end-point track $e(\Lambda(\T))$ on $(\T^*,+\infty)$ for any $r\in(0,r^*)$ and sufficiently large $\T^*\ge\T_0$.\\
(b)  
In this case, we have
$$
\tilde{E}_\alpha^{[0]}  =
\left\{
\left(e(\Lambda_\alpha(s)),s\right)
\;:\; s\in[-1,1]
\right\},
$$
so that $\tilde{E}_\alpha^{[r]}$ is connected, attracting and normally hyperbolic, contains $\tilde{e}^-$ and $\tilde{e}^+$, and lies $\delta$-close to $\tilde{E}_\alpha^{[0]}$ for any $r\in(0,r^*)$ and all $s\in[-1,1]$.
Since $\tilde{e}^-$ is a hyperbolic equilibrium with one unstable direction, $\tilde{E}_\alpha^{[r]}$ contains the branch of the unique one-dimensional unstable manifold of $\tilde{e}^-$ in the compactified system~\eqref{eq:odeextbtau}. Hence, by Proposition~\ref{prop:invsete-}(a), $\tilde{E}_\alpha^{[r]}$ corresponds 
to a unique local pullback attractor 
$x^{[r]}(\T,e^-)$ that limits to $e^-$ as $\T\rightarrow -\infty$
in the nonautonomous system~\eqref{eq:odewithrs}. 
It then follows from the properties of $\tilde{E}_\alpha^{[r]}$ that  $x^{[r]}(\T,e^-)$ exists
for all $\T\in\mathbb{R}$, and $\delta$-close and end-point tracks $e(\Lambda(\T))$ on $\mathbb{R}$ for any $r\in(0,r^*)$.
\qed

The next result is an analogous to Theorem~\ref{thm:tracking}, but for moving thresholds.
\begin{theorem}
\label{thm:trackingthresholds}
Consider a nonautonomous system~(\ref{eq:odewithrs}) with an  input $\Lambda(\T)$ that is exponentially asymptotically constant to $\lambda^+$. 
Suppose the future limit system~\eqref{eq:odea+} has an equilibrium regular  R-tipping edge state $\eta^+$.
 Then, 
\begin{itemize}
    \item[(a)] 
    There is a $\T_0$ (that may be $-\infty$), and
    \begin{itemize}
        \item [$\bullet$]
        A moving equilibrium regular edge state $\eta(\Lambda(\T))$ on $I=(\T_0,+\infty)$ that limits to $\eta^+$.
        \item [$\bullet$]
        A moving regular threshold $\theta(\Lambda(\T))$ on $I=(\T_0,+\infty)$ that contains $\eta(\Lambda(\T))$. 
    \end{itemize}
    \item[(b)] 
    Additionally, there is an R-tipping threshold $\Theta^{[r]}(\T)$ anchored at infinity by $\eta^+$. Furthermore, for any $\delta>0$ there is an $r^*(\delta)>0$ such that the R-tipping threshold $\Theta^{[r]}(\T)$ lies $\delta$-close\,\footnote{The notion of Hausdorff distance $d_H$ is discussed in Appendix~\ref{sec:A1}.} to the moving threshold $\theta(\Lambda(\T))$:
    \begin{equation}
    d_H\left(\Theta^{[r]}(\T),\theta(\Lambda(\T))\right)<\delta,
    \label{eq:threshtrack}
    \end{equation}
    for any $r\in(0,r^*)$ and all $\T > \T_0$.
\end{itemize}
\end{theorem}

\proof 
(a) 
Note that, from Definitions~\ref{defn:rthr} and~\ref{defn:edgestate}, the future limit system~\eqref{eq:odea+} has a regular threshold
$\theta^+$ containing $\eta^+$, and  $\theta^+$ and $\eta^+$ are normally hyperbolic. On applying Proposition~\ref{prop:edgecontinues} for the case $\lambda^*=\lambda^+$,  they
can be continued on some neighbourhood $Q$ of $\lambda^+$ to families of equilibrium regular edge states $\eta(\lambda)$ and regular thresholds $\theta(\lambda)$ that vary $C^1$-smoothly with $\lambda \in Q$. 
Pick any such $Q\subseteq P_\Lambda$ together with a $\T_0$ such that
$Q= \overline{\left\{\Lambda(\T): \T\in (\T_0,+\infty)\right\}}$.
This gives a moving equilibrium regular edge state $\eta(\Lambda(\T)$ on $I=(\T_0,+\infty)$ that limits to $\eta^+$, and a moving regular threshold $\theta(\Lambda(\T))$ on $I=(\T_0,+\infty)$ that limits to $\theta^+$ and contains $\eta(\Lambda(\T)$.\\
(b)
Choose the compactification parameter $\alpha$  that satisfies Proposition~\ref{prop:regular} for any $r>0$.
Let $s_0=g_\alpha(\T_0)$, and note that
$$
\tilde{\Theta}^{[0]}_\alpha:=\{(\theta(\Lambda_\alpha(s)),s)~:~s \in [s_0, 1]\},
$$
is a normally hyperbolic forward-invariant manifold in the extended phase space of the $r=0$ compactified system~\eqref{eq:odeextbtau}, that corresponds to the moving regular threshold $\theta(\Lambda(\T))$ on $I=(\T_0,+\infty)$. Note that $\tilde{\Theta}^{[0]}_\alpha$ contains $\tilde{\theta}^+$ and $\tilde{\eta}^+$.
It then follows from~\cite[Thm 9.1]{Fenichel1979} that,
for any $\delta>0$, we can choose a sufficiently small $r^*>0$, so that there is a perturbed normally hyperbolic  manifold $\tilde{\Theta}^{[r]}_\alpha$ that lies 
$\delta$-close to $\tilde{\Theta}^{[0]}_\alpha$ in the sense of~\eqref{eq:threshtrack}
for  any $r\in(0,r^*)$ and all $s\in[s_0,1]$
in the compactified system~\eqref{eq:odeextbtau}.
Furthermore, $\tilde{\Theta}^{[r]}_\alpha$ contains $\tilde{\eta}^+$, meaning that it is contained
within the stable manifold of $\tilde{\eta}^+$.
For any $r\in(0,r^*)$, pick a forward-invariant subset of 
$\tilde{\Theta}^{[r]}_\alpha$ on $[s_0,1]$. On applying Proposition~\ref{prop:invsete-}(b), this forward-invariant subset corresponds to an R-tipping threshold $\Theta^{[r]}(\T)$ that is anchored at infinity by $\eta^+$ and lies $\delta$-close to the moving threshold $\theta(\Lambda(\T))$ for all $\T >\T_0$ in the nonautonomous system~\eqref{eq:odewithrs}.
\qed

\subsection{Threshold Instability as a Criterion for R-tipping}
\label{sec:Rtippingcriteria}

This section maintains our goal of a mathematical framework that is applicable, and follows the approach in~\cite{Ashwin2016}.
Specifically, we use simple properties of the autonomous frozen system~\eqref{eq:odea}, and the external input $\Lambda$, to give rigorous yet easily
testable criteria for R-tipping in the nonautonomous 
system~\eqref{eq:odewithr} or~\eqref{eq:odewithrs}.
These criteria are  for moving sinks on $I=\R$ and R-tipping from $e^-$ via loss of end-point tracking,  due to crossing regular R-tipping thresholds anchored at infinity by an equilibrium regular R-tipping edge state.

Reference~\cite[Theorem 3.2]{Ashwin2016} uses the notion of  ``forward basin 
stability" to give sufficient conditions for such
R-tipping to occur, and to be excluded, in one-dimensional (scalar) systems. Recent work~\cite{Kiers2018,Xie2019} suggests that simple 
testable criteria to exclude  such R-tipping will be much less easy to formulate for 
higher dimensional systems unless there are additional constraints. 
The main reason is that, in higher dimensions, forward basin stability does not exclude the possibility of R-tipping.

Below, we use the notion of ``(forward) threshold instability" introduced in 
Sec.~\ref{sec:thr_inst} to give sufficient conditions for the occurrence of such R-tipping 
in arbitrary dimension.
In case (a), we give a sufficient condition to identify autonomous frozen systems that can exhibit  such R-tipping for suitably chosen external inputs $\Lambda$. 
In case (b), we give a sufficient condition for  such R-tipping to occur in a nonautonomous system with a (possibly reparametrized)  given external input $\Lambda$. This case is a generalization of \cite[Theorem 3.2 part 2]{Ashwin2016}.
\begin{theorem}
  \label{thm:Rtip}
  Consider a nonautonomous system~(\ref{eq:odewithrs}) with a parameter path $P$. Suppose the autonomous frozen system~\eqref{eq:odea} has a hyperbolic sink $e(\lambda)$ that  varies $C^1$-smoothly with 
$\lambda\in P$, and an equilibrium regular  edge state $\eta(\lambda)$ with a regular threshold $\theta(\lambda)$.
 \begin{enumerate}
 	\item[(a)] 
 	If $e(\lambda)$ is threshold unstable on $P$ due to $\theta(\lambda)$, then there is an exponentially
     bi-asymptotically constant input $\Lambda(\T)$ that
 	traces out $P_\Lambda = P$ and gives R-tipping  from $e^-$
 	in the nonautonomous system~(\ref{eq:odewithrs}).
    \item[(b)]
    Consider a given  exponentially bi-asymptotically constant input $\Lambda(\T)$ tracing out $P_\Lambda = P$ such that $e(\Lambda(\T))$ is forward 
    threshold unstable due to $\theta(\Lambda(\T))$, and
    $\eta(\Lambda(\T))$ limits to $\eta^+$. Then, there is R-tipping  from $e^-$ in the nonautonomous system~\eqref{eq:odewithrs} for  $\Lambda$ with suitably reparametrised time, i.e. for some $\tilde{\Lambda}(\T) = \Lambda(\sigma(\T))$ tracing out the same path $P_{\tilde{\Lambda}}=P_{\Lambda}=P$, where $\sigma$ is a strictly monotonic increasing function.
  \end{enumerate}
\end{theorem}
\begin{rmk}
Note that:
\begin{itemize}
    \item
    The R-tipping criteria in Theorem~\ref{thm:Rtip} are sufficient but not necessary: there are examples of (non-degenerate)
    R-tipping for a moving sink on $I=\R$ in the absence of forward threshold instability and presence of forward basin stability~\cite{Kiers2018,Xie2019}.
     \item
    The conditions in Theorem~\ref{thm:Rtip} do not necessary imply that the R-tipping is non-degenerate. 
    Nonetheless, we expect that a  solution $x^{[r]}(e^-)$ and the codimension-one  R-tipping threshold $\Theta^{[r]}(\T)$ will cross transversely on varying $r$, and suggest that ``(forward) threshold instability" will typically give non-degenerate R-tipping.
    \item
     The R-tipping in Theorem~\ref{thm:Rtip} is from $e^-$ and for a moving sink on $I=\R$, which is often the case of interest.
    We discuss generalisations of Theorem~\ref{thm:Rtip} to 
    R-tipping from a fixed $(x_0,\T_0)$ and/or for a moving sink on a finite or semi-infinite time interval $I$
    in Section~\ref{sec:Conclusions}.
    \item
    In the simplest cases in Theorem~\ref{thm:Rtip}(b),
    we may be able to choose $\tilde{\Lambda} = {\Lambda}$ and obtain R-tipping for a suitable choice of the rate parameter  $r=r^*$~\cite{Xie2019}, but more generally, $\tilde{\Lambda}$ is a time reparametrisation of $\Lambda$  with the same limiting behaviour.
    In other words, we can ensure that  the pullback attractor is on different sides of the R-tipping threshold for  a fixed $r$ and different $\tilde{\Lambda}$, but it is more complex to ensure that
    this occurs for a fixed $\tilde{\Lambda}=\Lambda$ and different $r$. 
\end{itemize}
\end{rmk}

The proof of Theorem~\ref{thm:Rtip} is given in Appendix~\ref{app:Rtipproof}.

\subsection{Connecting Orbit as a General Criterion for R-tipping and a General Method for Computing Critical Rates}
\label{sec:computing}

While B-tipping can be found and continued in system parameters on applying tools from theory of autonomous bifurcations~\cite{matcont,auto,Kuznetsov2004} to the autonomous frozen system~\eqref{eq:odea}, this is not the case for nonautonomous R-tipping. Furthermore, whereas Section~\ref{sec:Rtippingcriteria} considers R-tipping  for moving sinks on $I=\R$ (e.g. see Figure~\ref{fig:Rtip}(a)), some R-tipping occur from moving sinks on a  semi-infinite or even finite time interval $I\subset\R$ (e.g. see Figure~\ref{fig:Rtip}(b)). Therefore, there is a need for
general criteria and methods to find different nonautonomous R-tipping and continue them in system parameters.

To address this need, in this section we continue with an applicable mathematical framework.
Our focus remains on R-tipping  via loss of end-point tracking, due to crossing regular R-tipping thresholds anchored at infinity by an equilibrium regular R-tipping edge state.
However, there are two differences from Section~\ref{sec:Rtippingcriteria}. First, we relax the assumption of moving sinks on $I=\R$. Second,
we use properties of the autonomous compactified system~\eqref{eq:odeextbtau} to give rigorous criteria for R-tipping and critical rates in the nonautonomous 
system~\eqref{eq:odewithr} or~\eqref{eq:odewithrs}.

The proof of Theorem~\ref{thm:Rtip} used the 
compactification  technique of Section~\ref{sec:compact} to
show there is R-tipping in the nonautonomous system~\eqref{eq:odewithrs} by computing codimension-one  
heteroclinic connections  in the compactified system~\eqref{eq:odeextbtau}. 
A similar approach     has previously been used  on a case-by-case basis to compute critical rates  in specific examples of R-tipping~\cite{Alkhayuon2018,Ashwin2016,Ashwin2012,Perryman2014,Xie2019}. We show here that  connecting (heteroclinic) orbits of~\eqref{eq:odeextbtau} can be used to:
\begin{itemize}
    \item 
    Give necessary and sufficient conditions for the occurrence of non-degenerate
    R-tipping  from $e^-$ or $(x_0,\T_0)$ for moving sinks on any time interval $I\subseteq\R$.
    \item
    Give a general method for computing critical rates  for R-tipping.  This method also applies to more complicated regular R-tipping edge states such as limit cycles or quasiperiodic tori.
\end{itemize}
To be more specific, recall the notation from Section~\ref{sec:compact} for equilibria of the limit systems embedded in the extended phase space of the compactified system
$$
\tilde{e}^\pm=(e^\pm,\pm 1)
\quad\mbox{and}\quad
\tilde{\eta}^+=(\eta^+,1),
$$
and keep in mind that $s_0=g_\alpha(\tau_0)$.
In the case of asymptotic constant input with a future limit $\lambda^+$, R-tipping from a fixed $(x_0,\tau_0)$ in  nonautonomous system~\eqref{eq:odewithrs} depends on where $(x_0,s_0)$ lies in relation to the 
stable manifold $W_\alpha^{s,[r]}(\tilde{\eta}^{+})$ in the extended phase space of the autonomous compactified system~\eqref{eq:odeextbtau}. Here, $(x_0,s_0$) is 
fixed\,\footnote{Note that a fixed $(x_0,t_0)$ in nonautonomous system~\eqref{eq:odewithr} gives a rate-dependent 
$(x_0,s_0^{[r]})$ in the compactified system~\eqref{eq:odeextb}.}, but the position of $W_\alpha^{s,[r]}(\tilde{\eta}^{+})$ typically changes with $r$. 
R-tipping from  $(x_0,\tau_0)$ occurs when there is 
a {\em connecting orbit} from $(x_0,s_0)$ to 
$\tilde{\eta}^{+}$  
in the compactified system. Such connecting orbits arise when $(x_0,s_0)$ crosses 
$W_\alpha^{s,[r]}(\tilde{\eta}^{+})$ under varying $r$.
In the bi-asymptotic constant input case, R-tipping from $e^-$
in nonautonomous system~\eqref{eq:odewithrs} depends on where the one-dimensional unstable manifold 
$W_\alpha^{u,[r]}(\tilde{e}^{-})$ lies in relation to 
$W_\alpha^{s,[r]}(\tilde{\eta}^{+})$
in the extended phase space of the compactified system~\eqref{eq:odeextbtau}. 
 Here, the  positions of both
$W_\alpha^{u,[r]}(\tilde{e}^{-})$ and $W_\alpha^{s,[r]}(\tilde{\eta}^{+})$
typically change with $r$.
R-tipping from $e^-$ occurs when there is a 
{\em  connecting heteroclinic orbit}   from $\tilde{e}^-$ to $\tilde\eta^+$
in the compactified system. Such connecting orbits arise when $W_\alpha^{u,[r]}(\tilde{e}^{-})$ and $W_\alpha^{s,[r]}(\tilde{\eta}^{+})$ cross each other  under varying $r$. 
These observations allow us to state necessary and sufficient conditions for the occurrence of certain non-degenerate R-tipping in~\eqref{eq:odewithrs} in terms of non-degeneracy criteria for connecting (heteroclinic) orbits in~\eqref{eq:odeextbtau}.
To formulate these criteria in a Proposition, we use
$$
\mbox{trj}^{[r]}_\alpha(x_0,s_0)\subset\R^n\times[-1,1],
$$
to denote a trajectory started from $(x_0,s_0)$ in the phase space of the compactified system~\eqref{eq:odeextbtau} parametrised by the rate $r>0$. 
If this trajectory converges to $\tilde{e}^-$ backward in time, we write
$$
\mbox{trj}^{[r]}_\alpha(\tilde{e}^-) \subset W_\alpha^{u,[r]}(\tilde{e}^{-}),
$$
using the relation from Proposition~\ref{prop:invsete-}(a).
We also write
$$
\mbox{trj}^{[r]}_\alpha \subset\R^n\times[-1,1],
$$
to mean either $\mbox{trj}^{[r]}_\alpha(x_0,s_0)$ or $\mbox{trj}^{[r]}_\alpha(\tilde{e}^-)$, depending on the context.
\begin{proposition}
\label{prop:rtip_compact}
Consider the nonautonomous system~(\ref{eq:odewithrs}) with an input $\Lambda(\T)$ satisfying either of the following conditions:
\begin{itemize}
    \item [1.] 
    $\Lambda(\T)$  is exponentially asymptotically constant to $\lambda^{+}$. The future limit system~\eqref{eq:odea+} has an equilibrium regular  R-tipping edge state $\eta^+$.
    \item[2.]
    $\Lambda(\T)$  is bi-exponentially asymptotically constant to $\lambda^{-}$ and $\lambda^+$. In addition to  condition 1, the past limit system~\eqref{eq:odea-} has a hyperbolic sink $e^-$.
\end{itemize}
Let $\mbox{trj}^{[r]}_\alpha = \mbox{trj}^{[r]}_\alpha(x_0,s_0)$ in cases 1 or 2, or $\mbox{trj}^{[r]}_\alpha = \mbox{trj}^{[r]}_\alpha(\tilde{e}^-)\subset W_\alpha^{u,[r]}(\tilde{e}^{-})$ in case 2.
The nonautonomous system~\eqref{eq:odewithrs} undergoes non-degenerate R-tipping at $\eta^+$ with critical rate $r_c > 0$  if and only if, in the compactified system~\eqref{eq:odeextbtau}: 
\begin{itemize}
    \item [(a)]
    For $r=r_c$, $\mbox{trj}^{[r_c]}_\alpha$ is a (heteroclinic) connection to the regular R-tipping edge state $\tilde{\eta}^+$:
    $$
    \mbox{trj}^{[r_c]}_\alpha
    \subset W_\alpha^{s,[r_c]}(\tilde{\eta}^{+}).
    $$
    \item [(b)]
    There is a $\delta>0$ such that for $r\in(r_c-\delta,r_c)$ and  $r\in(r_c, r_c+\delta)$, 
    $\mbox{trj}^{[r]}_\alpha$
    lies on different sides of $W_\alpha^{s,[r]}(\tilde{\eta}^{+})$.
    \item[(c)]
     Each branch of $W^u(\tilde{\eta}^+)$ is a connection from $\tilde{\eta}^+$ to an attractor.
\end{itemize}
\end{proposition}
\begin{rmk}
\label{rmk:nondeghet}
Various conditions are usually proposed for a heteroclinic orbit to be considered as non-degenerate. These are typically assumptions about the orbit and the limiting states as well as more subtle assumptions on parameter variation and the geometry of linearised behaviour; see for example~\cite{homburg2010}. We consider the connecting (heteroclinic) orbit to $\tilde{\eta}^+$ in system~\eqref{eq:odeextbtau} to be non-degenerate if:
\begin{itemize}
    \item [(i)] 
    It is found at codimension one in $r$. 
    \item [(ii)]
    The trajectory of interest $\mbox{trj}^{[r]}_\alpha$ crosses from one side of $W_\alpha^{s,[r]}(\tilde{\eta}^{+})$ to the other at $r=r_c$.
    We do not require that the crossing occurs with non-zero speed in $r$, though this is likely to be typically the case.
    \item [(iii)] There are no homoclinic connections from $\tilde{\eta}^+$ to itself or heteroclinic connections from $\tilde{\eta}^+$ to other saddle(s). 
    Note that this assumption about  $W^u(\tilde{\eta}^+)$ is not explicitly about the connecting orbit  of interest or its limiting state(s).
\end{itemize}
Then, Proposition~\ref{prop:rtip_compact} says that there is a non-degenerate R-tipping in system~\eqref{eq:odeextbtau} from Definition~\ref{defn:Rtiptypes}(a) if and only if there is a non-degenerate connecting (heteroclinic) orbit to $\tilde{\eta}^+$ in system~\eqref{eq:odeextbtau}.
\end{rmk}
%
\vspace{5mm}
\noindent
{\em Proof of Proposition~\ref{prop:rtip_compact}.}
Choose any compactification parameter $\alpha$ such that Proposition~\ref{prop:regular} applies.
Recall from Section~\ref{sec:compautsyst} that $s(\T)=g_\alpha(\T)$, and relate $\mbox{trj}^{[r]}_\alpha$
to a solution $x^{[r]}(\T)$  of the nonautonomous system~\eqref{eq:odewithrs} with fixed $r>0$:
$$
\mbox{trj}^{[r]}_\alpha = 
\left\{
\left(x^{[r]}(\T),s(\T)\right)
\right\}_{\T\in\mathbb{R}}.
$$
Recall from Proposition~\ref{prop:invsete-}(b) that $W_\alpha^{s,[r_c]}(\tilde{\eta}^{+})$ 
contains a family of regular R-tipping thresholds $\Theta^{[r]}(\T)$, and each embedded edge tail contains one  branch of the unstable manifold $W^{u}(\tilde{\eta}^+)$.
Thus, conditions (a) and (b) imply that the nonautonomous system~(\ref{eq:odewithrs}) undergoes R-tipping: there are $r_c, r_2 >0$ such that 
$x^{[r_c]}(\T)\to \eta^+$ and $x^{[r_2]}(\T)\not\to \eta^+$ as $\T\to +\infty$. 
Condition (b) also implies that the rate $r_c$ is isolated in the sense that 
$x^{[r]}(\T)\not\to \eta^+$ for $0<|r-r_c|<\delta$. Hence $r_c$ is a 
critical rate.
Condition (b) together with Proposition~\ref{prop:edgetails}(a) imply that the lower and upper edge tails of $\eta^+$ are different. 
Finally, condition (c) implies that each edge tail connects $\eta^+$ to an attractor. Hence R-tipping is non-degenerate.
Conversely, non-degenerate R-tipping implies conditions (a), (b) and (c). Specifically, R-tipping implies (a). 
Each edge tail of $\eta^+$ being a different connection from $\eta^+$ to an attractor, together with Proposition~\ref{prop:edgetails}(c), imply (b).
Each edge tail of $\eta^+$ being a connection from $\eta^+$
to an attractor implies (c).
\qed
\\

In consequence, critical rates for R-tipping in nonautonomous system~\eqref{eq:odewithrs} can be found by finding $r$ that give codimension-one or higher connecting (heteroclinic) orbits to $\tilde{\eta}^+$ in the compactified system~\eqref{eq:odeextbtau}.
It is important to note that, unlike R-tipping thresholds, these connecting orbits are one-dimensional curves, which makes them relatively easy to detect in $r$, and then continue in other parameters to obtain curves or even hypersurfaces of critical rates.
This allows us to produce nonautonomous {\em R-tipping diagrams}~\cite{OKeeffe2019,OSullivan2021,Xie2019} akin to classical autonomous bifurcation diagrams.
Non-degenerate R-tipping has  additional requirements that $\tilde{\eta}^+$ is a regular edge state, the edge tails  of $\tilde{\eta}^+$ are different, and each edge tail is a connection from $\tilde{\eta}^+$ to an attractor. 
This means that parameter continuation of critical rates may give continuation of
non-degenerate R-tipping, at least in cases where different edge tails connect to attractors that are simple enough (e.g. an equilibrium or a limit cycle) to continue as attractors in these other parameters.  
In practice, critical rates and non-degenerate R-tipping can always be computed using a shooting method.  In cases where $\tilde{\eta}^+$ is an equilibrium, a limit cycle, or possibly a quasiperiodic torus, parameter continuation can be done using numerical implementations of detection and continuation methods such as that of  Beyn~\cite{Beyn1990numerical} and Lin~\cite{krauskopf2008,Lin}, or numerical software packages such as HOMCONT~\cite{champneys1996numerical} or MATCONT~\cite{matcont} based on these methods. 

 Finally, we point out that our approach  relating R-tipping in the nonautonomous system~\eqref{eq:ode} to an $\tilde{e}^-$-to-$\tilde{\eta}^+$ heteroclinic connection in the compactified autonomous 
system~\eqref{eq:odeextbtau}
has strong parallels with an alternative approach relating R-tipping to a collision (loss of uniform asymptotic stability) of a pullback attractor that limits to $e^-$ and a pullback repeller that  limits to $\eta^+$ in the one-dimensional (scalar)  nonautonomous system~\eqref{eq:ode}~\cite{Kuehn2021,Longo2021}.

\section{Summary and Open Questions}
\label{sec:Conclusions}

This paper describes nonlinear dynamics of a multidimensional nonautonomous system~\eqref{eq:ode} (or equivalently system~\eqref{eq:odewithrs}) for quite a general class of asymptotically constant external inputs, or parameter shifts, that vary with time at a rate $r$ and decay exponentially at infinity. It uses extension to the compactified autonomous 
system~\eqref{eq:odeextbtau}--\eqref{eq:lambda_s+-} by including autonomous dynamics of the future~\eqref{eq:odea+} and past~\eqref{eq:odea-} limit systems from infinity. This approach allows us to understand the dynamics of the nonautonomous system~\eqref{eq:ode}
in terms of compact invariant sets of the autonomous future limit system.
The focus is on genuine nonautonomous R-tipping instabilities that can occur at critical rates $r=r_c$.
Asymptotically autonomous systems have been studied in the past in terms of asymptotic equivalence of two separate systems: the nonautonomous system~\eqref{eq:ode} and the future limit system~\eqref{eq:odea+}~\cite{Castillo1994,Holmes_Stuart1992,Markus1956,Robinson1996,Thieme1994}. 
A particular advantage of our approach is that all invariant sets, including  trajectories of the nonautonomous system~\eqref{eq:ode} as well as compact invariant sets of the autonomous limit systems~\eqref{eq:odea+} and~\eqref{eq:odea-}, can be related to the one autonomous compactified system~\eqref{eq:odeextbtau}--\eqref{eq:lambda_s+-}. 

Our strategy is to define R-tipping in the nonautonomous system, introduce the key concepts of R-tipping thresholds as well as R-tipping edge states and their edge tails also in the nonautonomous system, and derive the main results using the compactified system. As a starting point, Proposition~\ref{prop:csdyn} uses results from~\cite{Wieczorek2019compact} to show for exponentially bi-asymptotically constant inputs that the compactified system is in standard format for a $C^1$ smooth slow-fast system,
where the rate parameter $r$ is the timescale separation. Small $r$ corresponds to quasistatic approximation, giving rise to tracking of a branch of base attractors for the frozen system~\eqref{eq:odea}. R-tipping can be understood as a breakdown of the quasistatic approximation, giving rise to loss of tracking (i.e. moving away from the branch of base attractors) due to crossing an R-tipping threshold for some larger $r$.

We give methods to identify, classify and understand R-tipping in a wide variety of ODE models from applications.
In other words, we  generalise and extend results from~\cite{Ashwin2016} on irreversible R-tipping in one dimension  to arbitrary dimensions and to different cases of R-tipping, some of which can occur only in higher dimensional systems. 
In particular, we give tools for a fairly complete understanding of systems with equilibrium base attractors whose basin boundaries consist of regular thresholds anchored by regular equilibrium edge states. 
This culminates in two results. Theorem~\ref{thm:Rtip} gives an easily verifiable set of sufficient conditions for R-tipping to be present in a multidimensional nonautonomous system~\eqref{eq:ode} for some choice of the external input. Proposition~\ref{prop:rtip_compact} shows 
how R-tipping in the nonautonomous system corresponds to a (heteroclinic) connection to an R-tipping edge state in the compactified autonomous system, and thus gives a numerical tool for quantifying R-tipping and computing critical rates in quite general cases.

A challenge for the future is to understand and classify R-tipping   for more complicated cases such as:
\begin{itemize}
    \item [(a)]
    {\em  R-tipping  for a moving sink on a semi-infinite or finite time interval $I$.} 
    Theorem~\ref{thm:Rtip} considers R-tipping from an equilibrium attractor $e^-$ for moving sinks on $I=\R$.
    The result in cases of  R-tipping from a fixed $(x_0,\T_0)$ for a moving sink on an infinite $I=\R$ or semi-infinite  $I=(\T_-,+\infty)\subset\R$  will follow from a simple generalisation of Theorem~\ref{thm:Rtip}, that is on considering the trajectory from $(x_0,\T_0)$ rather than the one that limits to $e^-$.
    Additionally, there are cases of R-tipping from $e^-$ for a moving sink on a semi-infinite $I=(-\infty,\T_+)\subset\R$, or from a fixed $(x_0,\T_0)$ for a moving sink on a finite $I=(\T_-,\T_+)\subset\R$ or semi-infinite $I=(-\infty,\T_+)\subset\R$. In such cases, the moving sink bifurcates or disappears at some finite time, and need not even be forward threshold unstable (e.g. see Figure~\ref{fig:Rtip}(b)). Thus, such cases will require a more extensive generalisation of Theorem~\ref{thm:Rtip}.
    \item [(b)]
    {\em R-tipping from non-equilibrium attractors $\gamma^-$.}
    For systems with phase space of dimension higher than one, 
    there can be 
    R-tipping from more general attractors $\gamma^-$ including limit cycles~\cite{Alkhayuon2020,Alkhayuon2018}, quasiperiodic tori, and chaotic attractors~\cite{Alkhayuon2020weak,Kaszas2019,Lohmann2021}. 
    It is interesting to note that results on R-tipping in such cases will depend to some extent on the approach taken. For example, 
    non-degenerate 
    R-tipping according to Definition~\ref{defn:Rtiptypes} can be generically found at  codimension one or zero, depending on whether we take the pointwise or setwise approach. 
    In the pointwise approach, where one considers a single solution that limits to $\gamma^-$ as $\T\to -\infty$, non-degenerate R-tipping 
    can be generically found only at codimension-one in $r$, as explained 
    in Section~\ref{sec:Rtipcasec}.
    By contrast, in the setwise approach, one considers the set of all solutions that limit to $\gamma^-$ as $\T\to -\infty$.
    In this case, it is possible that  non-degenerate R-tipping can be found at codimension-zero in $r$: there can be an interval of $r$ such that non-degenerate R-tipping is found for any value of  $r$ within the interval and some solution in the set of solutions that limit to $\gamma^-$ as $\T\to -\infty$. 
    Furthermore,
    non-equilibrium attractors  can give rise to additional cases of R-tipping, such as ``partial R-tipping" from  a limit cycle $\gamma^-$ described in~\cite{Alkhayuon2018} (see also~\cite{Alkhayuon2020}), and to additional cases of tracking, such as ``weak tracking"~\cite{Alkhayuon2020weak} where the pullback attractor limits to an unstable subset of a chaotic attractor $\gamma^-$ as $\T\to -\infty$.  A physical measure on $\gamma^-$ can be used to quantify the probability that R-tipping takes place \cite{ashwin2021physical,newman2022physical}.
\item [(c)]
    {\em R-tipping without crossing regular thresholds.}
    For systems with phase space of dimension higher than one, it is possible 
    to have R-tipping where, as $\T\to +\infty$, the solution limits to a compact invariant set $\eta^+$ on the boundary of a basin of attraction that is not a regular edge state. Such
    $\eta^+$ may be associated with a threshold that is irregular, or with no threshold at all, for one of several possible reasons. More precisely, the boundary of a basin of attraction may include any of:
\begin{itemize}
    \item[(i)] 
    Saddle periodic orbits $\eta^+$ with codimension-one stable manifolds that are not orientable (irregular thresholds).
    \item[(ii)] 
    Chaotic saddles $\eta^+$ with codimension-one stable manifolds that are not embedded (irregular thresholds).
    \item[(iii)] 
    Compact invariant sets $\eta^+$ with stable invariant manifolds of codimension two or higher, for example a source in $\R^2$ (no thresholds).
\end{itemize}  
In all three cases, R-tipping will occur without crossing a regular threshold. 
Case (i) leads to R-tipping that does not give a change in the system behaviour.
 Case (ii) can generate basin boundaries with highly nontrivial fractal structure~\cite{McDonald1985}. 
This means that R-tipping may occur not only at isolated values of $r_c$, but also at sets of $r_c$ with nontrivial accumulation points. Case (iii) generically will not be of codimension one in $r$, but it can be for R-tipping from non-equilibrium attractors $\gamma^-$; see point (b) above. For example, in the setwise approach, trajectories from a limit cycle attractor $\gamma^-$ may interact with an equilibrium $\eta^+$ with two unstable directions at codimension one in $r$. Such ``invisible R-tipping" is documented in~\cite{Alkhayuon2018}.
\item [(d)]
    {\em R-tipping due to crossing quasithresholds.}
    In any dimension, it is possible for 
     so-called ``quasithresholds''~\cite{FitzHugh1955}
     to be present in system~\eqref{eq:ode}.
     The key difference from regular thresholds is that quasithresholds do not contain an R-tipping edge state $\eta^+$. 
     Therefore,
     \begin{itemize}
         \item[(i)]
         Quasithresholds cannot give rise to qualitative {\em R-tipping via loss of end-point tracking.} They can only give rise to quantitative {\em R-tipping via loss of $\delta$-close tracking}; see Section~\ref{sec:Rtipintro}.
         \item[(ii)]
         Rigorous definitions of quasithresholds and R-tipping via loss of $\delta$-close tracking that are relevant for applications still remain a challenge~\cite{OSullivan2021,Perryman2014}.
     \end{itemize}
     Quasithresholds can arise when a moving regular edge state disappears at some finite time~\cite[Sec.4.9]{Xie2019}, or when the frozen system is slow-fast~\cite{FitzHugh1955,OSullivan2021,Vanselow2019,Wieczorek2011}.
     Recent examples of R-tipping due to crossing quasithresholds in slow-fast systems show that {\em singular R-tipping edge states} may appear in the limit of infinite time scale separation; see e.g.~\cite{OSullivan2021}.
\item [(e)]
    {\em R-tipping for asymptotically constant external inputs with non-exponential asymptotic decay.}
    Our results assume asymptotically constant external inputs with exponential decay. This ensures that (normally) hyperbolic compact invariant sets of the autonomous limit systems remain (normally) hyperbolic when embedded in the extended phase space of the compactified system. It should be possible to generalise our results to  asymptotically constant external inputs with slower than exponential decay, provided they are ``normal" in the sense of~\cite[Definition 2.2]{Wieczorek2019compact}. Although such inputs give rise to a centre direction in the compactified system, one can show that both the ensuing centre manifold of $\tilde{e}^-$ and the centre-stable manifold of a regular equilibrium R-tipping edge state $\tilde{\eta}^+$ are unique~\cite[Theorems 3.3 and 3.4]{Wieczorek2019compact}.
    \item [(f)]
    {\em R-tipping for external inputs that are not asymptotically constant.}
    While we focus here on asymptotically constant external inputs, more complex external inputs represent another interesting direction of generalization.
    In particular, one could consider external inputs  that are asymptotically periodic or quasiperiodic. One proposed definition for R-tipping in this general case is suggested in \cite{hoyer2021rethinking,Kuehn2021,Longo2021} as a bifurcation of a pullback attractor. There are many parallels with our work on relating R-tipping to a heteroclinic connection in the compactified system (see also the last paragraph in Section~\ref{sec:computing}), but obtaining general results without imposing stringent hypotheses is likely to be a challenge. Also note that R-tipping due to crossing a quasithreshold may not correspond to a bifurcation of a pullback attractor.
    \item [(g)]
     {\em R-tipping in nonautonomous partial differential equations (PDEs).}
     So far, analysis of R-tipping have focused on nonautonomous ordinary differential equation models~\eqref{eq:ode}. However, there are important examples of R-tipping  in spatially-extended systems modelled by nonautonomous PDEs~\cite{chen2015,siteur2014} including heterogenous reaction-diffusion systems~\cite{Berestycki2009,hasan2022}. Analysis of R-tipping in PDEs is more challenging, will likely involve new critical factors such as critical spatial extent of the external input, and requires development of alternative mathematical techniques; e.g. see~\cite{hasan2022}.   
     \item [(h)]
     {\em R-tipping and Control Theory.}
     The  R-tipping framework presented here gives rigorous results about asymptotic behaviour of a nonlinear system for a given external input, in the spirit of dynamical systems theory. This approach is motivated by applications  where given inputs may be difficult to alter or control (e.g. climate, ecology, earthquakes or neuroscience). 
     An alternative approach is to specify the desired asymptotic state, and use ideas from control theory to make rigorous statements about the class of `optimal' external inputs.
     This  interesting direction of future research on R-tipping is of interest in applications where one has control over the external inputs (e.g. control engineering, climate change mitigation strategy, disease treatment or epidemiological intervention strategy).
\end{itemize}

\subsection*{Acknowledgements}
The initial stages of this research was supported by the CRITICS Innovative Training Network, funded by the European Union's Horizon 2020 research and innovation programme under the Marie Sk\l{}odowska-Curie Grant Agreement No. 643073.
The research of SW was partially supported by the EvoGamesPlus Innovative Training Network funded by the European Union’s Horizon 2020 research and innovation programme under the Marie Skłodowska-Curie grant agreement No 955708, and the Enterprise Ireland Innovative Partnership Programme project IP20190771.
The research of PA was partially supported through funding from EPSRC project EP/T018178/1. This is TiPES contribution \#135. This project received funding from the European Union’s Horizon 2020 research and innovation programme under grant agreement No 820970 (TiPES). We thank the reviewers and the following for their perceptive comments on this research: Ulrike Feudel, Justin Finkel, Chris K.R.T. Jones, Bernd Krauskopf, Martin Rasmussen, Jan Sieber, Mary Silber, Katherine Slyman, Alex Strang.


\bibliographystyle{plain}
\bibliography{rtip_Dec_2022}


\appendix

\section{Appendix: Some geometric background}

\subsection{Hausdorff Distance Functions}
\label{sec:A1}

Recall the {\em Hausdorff semi-distance} between a 
point $x$ and a compact set $A$ of a normed space is given by
\begin{equation}
\label{eq:Hsdist}
    d(x,A) = \inf_{y\in{A}}\,\Vert x - y\Vert,   
\end{equation}
For simplicity, we write
\begin{align}
&x^{[r]}\left(\T\right)\to A\;\;\mbox{as}\;\; \T\to +\infty
\quad\mbox{to denote}\quad
d\left(x^{[r]}(\T),A\right)\rightarrow 0\;\;\mbox{as}\;\; \T\to +\infty,\;\;\mbox{and}\nonumber\\
&x^{[r]}\left(\T\right)\not\rightarrow A\;\;\mbox{as}\;\; \T\to +\infty
\quad\mbox{to denote}\quad
d\left(x^{[r]}(\T),A\right)\not\rightarrow 0\;\;\mbox{as}\;\; \T\to +\infty.\nonumber
\end{align}
In Theorem~\ref{thm:trackingthresholds}  we use the
{\em Hausdorff distance} between compact sets $A$ and $B$:
\begin{equation}
\label{eq:Hdist}
    d_H(A,B)=\max\left(d(A,B),d(B,A)\right),
\end{equation}
where 
$$
d(A,B)=\sup_{x\in A}\left[\inf_{y\in B} \Vert x-y\Vert\right].
$$
Note that that $d(A,B)=0$ if and only if $A\subset B$ and $d_H$ is a metric on the space of compact subsets.

\subsection{Embedded and Orientable Manifolds}
\label{sec:A2}

In order to define regular thresholds in Sec.~\ref{sec:rthr}, we recall 
some properties of invariant manifolds, and refer to~\cite{Robinson1999} 
for a more general discussion.
A set $S\in\R^n$ is an {\em immersed codimension-one manifold} if there 
is an $(n-1)$-dimensional manifold $V$ and a smooth map 
$$
F:V\to \R^n,
$$
such that $F(V)=S$ and $DF(v)$ has maximal rank at all $v\in V$. 
The immersed manifold $S$ is {\em embedded} if $F$ can be chosen such 
that $F$ is a homeomorphism onto its image.  For the particular case 
of an embedded codimension-one manifold, $F(V)=S\subset \R^n$ is {\em orientable} if there is a normal unit vector $\nu(x)$ that varies 
smoothly with $x\in S$. Note that $\nu(x)$ is normal to the tangent 
space $T_xS$, and $\mu\in T_xS$ if and only if $\nu(x)\cdot\mu=0$. 
In such case, there are two choices for a normal unit vector 
corresponding to $\pm \nu(x)$.
We say an embedded manifold $S$ {\em varies continuously (or smoothly)
with} $\lambda$ if the embedding map $F$ can be chosen to be continuous 
(or smooth) in $\lambda$. 

Suppose that $S$ is a codimension-one invariant stable manifold of a (normally) hyperbolic compact invariant set defined to contain the set.
Then, $S$ is an injectively immersed repelling manifold~\cite{Robinson1999}
and thus a candidate for a threshold. 
However, $S$ need not be orientable.
For example, if $S$ is the stable manifold of a saddle limit cycle with a
real negative Floquet multiplier~\cite{Osinga2003}, or the stable manifold
of a saddle equilibrium that undergoes a non-orientable homoclinic bifurcation~\cite{Aguirre2013}, then $S$ is non-orientable.
Moreover, an orientable $S$ need not be embedded: it may be remarkably complex, locally disconnected and even fractal in structure; see~\cite{Aguirreetal2009} for a review. In case $S$ is non-orientable 
or not embedded, one may be able to restrict to an orientable embedded submanifold  of $S$, though this is not possible in general.

\subsection{Signed distance near a threshold}
\label{sec:A3}

Near an embedded orientable codimension-one manifold $S$, one can define a {\em signed distance} between a point $x$ and $S$. We choose an open set $N$ such that $S$ divides $N$ into two components which we call (arbitrarily) $N_{-}$ and $N_{+}$, and use $N^c$ to denote the complement of $N$. We then define
\begin{equation}
\label{eq:sdist}
\begin{split}
d_s(x,S) = 
  \left\{\begin{array}{rcl}
      d(x,S) &\mbox{if} & x\in N_+,\\
      0 &\mbox{if} & x\in S,\\
      -d(x,S) &\mbox{if} & x\in N_-.\\
      \infty &\mbox{if}& x\in N^c
    \end{array}\right.
    \end{split}
\end{equation}
Note that there is a choice of $N$ such that $d_s$ is a smooth function of $x\in N$~\cite[Lemma 14.16]{gilbarg}. 

\subsection{Attractors and boundary of a basin of attraction}
\label{sec:A4}

Suppose that 
$\psi(\T,x_0)$ is the solution to the autonomous frozen system~\eqref{eq:odears} 
at time $\T$ started from the initial condition $x=x_0$ at $\T=0$. Consider any set $D$ and define
$\psi(\T,D)=\{\psi(\T,x)\,:\, x\in D \}$. Then the $\omega-$limit set of $D$ is
$$
\omega (D)= \bigcap_{T>0}\,\overline{\{\psi(\T,D)\,:\,\T>T\}}.
$$
We define an attractor as follows~\cite{Milnor2006}:
\begin{defn}
We say that a compact invariant set $A\subset\R^n$
is an attractor for the autonomous frozen system
if:
\begin{itemize}
    \item[(i)] $A$ is the $\omega$-limit set of a neighbourhood of itself.
    \item[(ii)]
    $A$ does not contain any proper subsets that satisfy (i).
\end{itemize}
\end{defn}

The basin of attraction of $A$ is $$
B(A)=\{x\,:\,\omega(x)\subset A\},
$$ 
and its 
boundary is  
$$\partial B(A) = \overline{B(A)}\setminus B(A),$$
where $\overline{B(A)}$ is the basin closure.
Note that, in general, a codimension-one basin boundary need not divide the phase space into different basins of attraction. Indeed, the basin boundary need not be connected or even locally connected.


\section{Proof of Theorem~\ref{thm:Rtip}}
\label{app:Rtipproof}

We give a detailed proof of statements (a) and (b).

(a) Threshold instability of $e(\lambda)$  on $P$ due to $\theta(\lambda)$  implies that there is a $C^1$-smooth family of $\theta(\lambda)$, as well as $\lambda_a$ and $\lambda_b$ in $P$ and in the domain of existence of $\theta(\lambda)$, such that 
$d_s(e(\lambda_a),\theta(\lambda_b))=0$ and 
$d_s(e(\lambda_1),\theta(\lambda_2))$ takes both
signs in any neighbourhood 
of $(\lambda_a,\lambda_b)$ in $P^2$.
Recall from~\eqref{eq:sdtau} the signed distance notation $\Delta_{\Lambda}(\T_1,\T_2)$ at different points in time, and from
Appendix~\ref{sec:A3} that $\Delta_{\Lambda}(\T_1,\T_2)$ is smooth and well defined near $(\T_a,\T_b)$.
Now choose any $C^1$-smooth function $\Lambda(\tau)$ such that:
\begin{itemize}
    \item 
    $\Lambda(\tau)$ traces out $P_{\Lambda}=P$ and is exponentially bi-asymptotically constant to $\lambda^\pm$. 
    \item 
    There are $\T_a<\T_b$ such that\,\footnote{Note that $\Lambda(\T)$ passes through $\lambda_{a}$ before $\lambda_b$, though it may pass through either or both of these values several times.}
    $\Lambda(\T_a)=\lambda_{a}$ and $\Lambda(\T_b)=\lambda_{b}$,
    \item 
     $\Delta_{\Lambda}(\T_a,\T_b) = 0$, and
    $\Delta_{\Lambda}(\T_1,\T_2)$ 
    takes both signs in any neighbourhood
    of $(\T_a,\T_b)\in\R^2$.
    \item 
    The future limit $\lambda^+$ of $\Lambda(\T)$ 
    is in the domain of existence of $\eta(\lambda)\in\theta(\lambda)$.
\end{itemize} 
For such external input $\Lambda(\T)$, the moving sink $e(\Lambda(\T))$ is forward threshold unstable due to a moving regular threshold $\theta(\Lambda(\T))$ with a moving equilibrium regular edge state $\eta(\Lambda(\T))$ that limits to  
an equilibrium  regular R-tipping edge state 
$\eta^+$. We then apply case (b) of this theorem for this $\Lambda(\T)$ to obtain the result.

(b) Choose any convex neighbourhood $\mathcal{N}$ of $(\T_a,\T_b)$ in $\mathbb{R}^2$.
Forward threshold instability of $e(\Lambda(\T))$ due to $\theta(\Lambda(\T))$ means that
 we can choose a small enough $\delta>0$, as well as the time pairs
$(\T_a^-,\T_b^-)$ and $(\T_a^+,\T_b^+)$ in $\mathcal{N}$, such that
\begin{equation}
\Delta_{\Lambda}(\T_a^+,\T_b^+) = \delta > \,0
\quad\mbox{and}\quad
\Delta_{\Lambda}(\T_a^-,\T_b^-) =  -\delta < \,0;
\label{eq:Delta_delta}
\end{equation}
see Figure~\ref{fig:N} for an illustration of this.

\begin{figure}[t]
  \begin{center}
    \includegraphics[width=8cm]{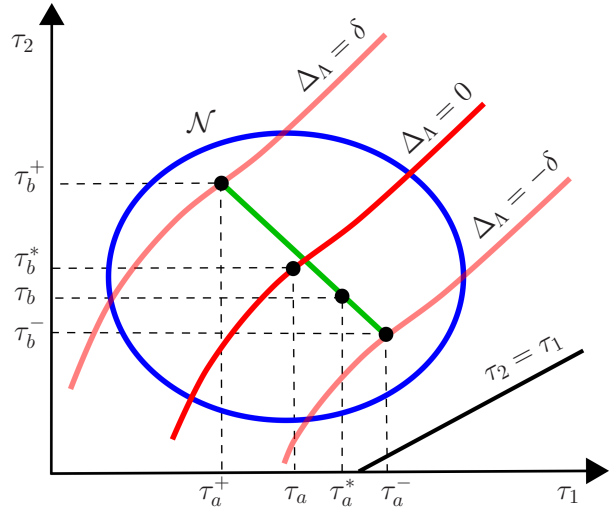}
  \end{center}
  \vspace{-5mm}
  \caption{An illustration of the threshold in the proof of Theorem~\ref{thm:Rtip}: the $(\T_1,\T_2)$-plane  with a (blue) convex neighbourhood ${\cal{N}}$ of $(\T_a,\T_b)$ in the region where $\T_2>\T_1$. Shown are examples of time pairs: $(\T_a,\T_b)$ where $\Delta_{\Lambda}(\T_a,\T_b) =0$,  $(\T_a^+,\T_b^+)$ where $\Delta_{\Lambda}(\T_a^+,\T_b^+) =\delta > 0$, and $(\T_a^-,\T_b^-)$ where $\Delta_{\Lambda}(\T_a^-,\T_b^-) =-\delta <0$. A time pair $(\T_a^*,\T_b^*)$, where $d_s(x^{[r^*,\tilde\Lambda]}(\epsilon,e^-),\Theta^{[r^*,\tilde\Lambda]}(\epsilon))=0$, is guaranteed to lie somewhere on the (green) line from $(\T_a^+,\T_b^+)$ to $(\T_a^-,\T_b^-)$.
  }
  \label{fig:N}
\end{figure}

Next, consider a time reparametrisation of the prescribed external input $\Lambda(\T)$:
\begin{equation}
\tilde\Lambda(\T)=\Lambda(\sigma_{\T_\alpha,\T_\beta,\epsilon}(\T)),
\label{eq:lambdaconstruct}
\end{equation}
using a parametrised family of strictly monotone increasing functions $\sigma_{\T_\alpha,\T_\beta,\epsilon}(\T)$ with range $\R$ and three parameters\,\footnote{Note that the subscript in $\T_\alpha$ is not related to the compactification parameter $\alpha$.} $\epsilon>0$ and  $\T_\alpha<\T_\beta\in{\cal{N}}$.
We define  this reparameterisation of  time
by means of a function
\begin{equation}
    \sigma_{\T_\alpha,\T_\beta,\epsilon}(\T):= \T_{\alpha}+\epsilon \T +\left(\T_\beta-\T_\alpha-\epsilon^2\right)
    \xi\left(\T/\epsilon
    \right),
    \label{eq:sigma}
\end{equation}
where $\xi(v)$ is a smooth function such that $\xi(v)=0$ for $v\leq 0$,  $\xi(v)=1$ for $v\geq 1$ and $\xi(v)$ is strictly monotone increasing for $v\in(0,1)$. For example, we can take
$$
\xi(v):= \frac{\chi(v)}{\chi(v)+\chi(1-v)},
$$
where
$$
\chi(v):=
\left\{
    \begin{array}{rcl}
    \exp(-1/v) &\mbox{for}& v>0,\\
    0 &\mbox{for}& v \le 0,
    \end{array}\right.
$$
takes values in the interval $[0,1)$ and is strictly monotone increasing for $v>0$. 
One can check that $\sigma_{\T_\alpha,\T_\beta,\epsilon}(\T)$ defined by~\eqref{eq:sigma}
is $C^{\infty}$-smooth in all three parameters, strictly monotone increasing in $\T$ as long as $\epsilon^2<\T_\beta-\T_\alpha$, 
linear with slope $\epsilon$ for $\T\le 0$ and for $\T\ge \epsilon$:
\begin{equation}
\sigma_{\T_\alpha,\T_\beta,\epsilon}(\T)=
\left\{ \begin{array}{rll}
\T_\alpha+\epsilon\T& ~\mbox{ if } & \T\leq 0,\\
\T_\beta+\epsilon(\T-\epsilon)& ~\mbox{ if } & \T\geq \epsilon,
\end{array}\right.
\nonumber
\end{equation}
and satisfies 
\begin{equation}
\label{eq:aux0}
\sigma_{\T_\alpha,\T_\beta,\epsilon}(0)=\T_{\alpha}\quad\mbox{and}\quad\sigma_{\T_\alpha,\T_\beta,\epsilon}(\epsilon)=\T_{\beta}.
\end{equation} 
In other words, for $\epsilon$ small compared to $\T_\beta-\T_\alpha$, there is slow change for $\T \le 0$, rapid change for $\T\in(0,\epsilon)$, and slow change thereafter.
In the limit $\epsilon=0$, the reparameterisation function~\eqref{eq:sigma}  has a  jump discontinuity at $\T=0$.
 Most importantly, if $\Lambda(\T)$ is exponentially bi-asymptotically constant with decay coefficient $\rho>0$, then $\tilde{\Lambda}(\T)$ is also exponentially bi-asymptotically constant with decay coefficient $\epsilon \rho$,  so the results
in Sections~\ref{sec:compactdyns} and~\ref{sec:TrackingProof} apply to $\tilde{\Lambda}(\T)$.

For the reparametrised input $\tilde{\Lambda}$ in~\eqref{eq:lambdaconstruct}, we write the unique pullback attractor 
from Proposition~\ref{prop:invsete-}(a) as $x^{[r,\tilde\Lambda]}(\T,e^-)$ to indicate that, in addition to $r$, it depends on 
$\epsilon$ and $\T_\alpha < \T_\beta$ through $\tilde{\Lambda}$. 
We fix the rate parameter $r = r^* > 0$ and show that there is a choice of the  parameters $\epsilon$ and $\T_{\alpha} < \T_{\beta}$  in $\tilde{\Lambda}$ such that the ensuing $\tilde\Lambda(\T)$ 
gives R-tipping at  this $r=r^*$.

By the argument in  Theorem~\ref{thm:tracking}(b),  solution $x^{[r^*,\tilde\Lambda]}(\T,e^-)$ exists and $\delta$-close tracks the moving sink $e(\tilde\Lambda(\T))$ for all  $\T \le 0$ if $\epsilon$ is small enough. Similarly, by the argument in  Theorem~\ref{thm:trackingthresholds}(b), a regular R-tipping threshold
$\Theta^{[r^*,\tilde\Lambda]}(\T)$ anchored by $\eta^+$ at infinity exists and $\delta$-close\,\footnote{The notion of Hausdorff distance $d_H$ is discussed in Appendix~\ref{sec:A1}.} tracks the moving regular threshold $\theta(\tilde\Lambda(\T))$
for all  $\T \ge \epsilon$ if $\epsilon$ is small enough.
This means that there is an $\epsilon_1 > 0$ such that if $0<\epsilon<\epsilon_1$ then
\begin{equation}
\label{eq:d1}
d\left(x^{[r^*,\tilde\Lambda]}(\T,e^-),e(\tilde{\Lambda}(\T))\right)<\frac{1}{3}\delta\quad\mbox{for all}\quad \T\leq 0\quad\mbox{and}\quad (\T_\alpha,\T_\beta)\in{\cal{N}}, 
\end{equation}
and
\begin{equation}
\label{eq:dH}
d_H\left(\Theta^{[r^*,\tilde\Lambda]}(\T),\theta(\tilde\Lambda(\T))\right)<\frac{1}{3}\delta\quad\mbox{for all}\quad \T\geq \epsilon
\quad\mbox{and}\quad (\T_\alpha,\T_\beta)\in{\cal{N}}.
\end{equation}
Furthermore, local continuity of $x^{[r^*,\tilde{\Lambda}]}(\T,e^-)$ on varying time and the  three parameters in $\tilde{\Lambda}$ means that there is an  $\epsilon_2 > 0$ such that if $0<\epsilon<\epsilon_2$ then
\begin{equation}
\label{eq:d2}
d\left(x^{[r^*,\tilde\Lambda]}(0,e^-),x^{[r^*,\tilde{\Lambda}]}(\epsilon,e^-)\right)<\frac{1}{3}\delta\quad\mbox{for all}\quad (\T_\alpha,\T_\beta)\in{\cal{N}}.
\end{equation}
We chose $0 < \epsilon < \min\{\epsilon_1,\epsilon_2 \}$.

 We now examine the signed distance\,\footnote{The notion of signed distance $d_s$ is discussed in Appendix~\ref{sec:A3}. } between $x^{[r^*,\tilde\Lambda]}(\T,e^-)$ and $\Theta^{[r^*,\tilde\Lambda]}(\T)$ at time $\T=\epsilon$, its dependence on the two remaining parameters $\T_\alpha < \T_\beta$, and choose $(\T_\alpha,\T_\beta)\in{\cal{N}}$ that give R-tipping.
Recall the triangle inequality  $d(a,b)\le d(a,c) + d(c,b)$ for points $a,b,c\in \R^n$, and also note that $|d_s(a,S)-d_s(a',S)|\le d(a,a')$ and $|d_s(a,S)-d_s(a,S')|\leq d_H(S,S')$ for any codimension one sets $S,S'$, and points $a,a'$ in some convex neighbourhood of $S$ and $S'$, respectively, where $d_s(a,S)$ and $d_s(a,S')$ are defined. Using these inequalities together with~\eqref{eq:d1} and~\eqref{eq:d2},
note that
\begin{align}
\begin{split}
& \left|d_s\left(x^{[r^*,\tilde\Lambda]}(\epsilon,e^-),\theta(\tilde\Lambda(\epsilon)\right) - d_s\left(e(\tilde\Lambda(0)),\theta(\tilde\Lambda(\epsilon)\right)\right| \\
& \le
d\left(x^{[r^*,\tilde\Lambda]}(\epsilon,e^-),e(\tilde\Lambda(0))\right)\\
&\le
d\left(x^{[r^*,\tilde\Lambda]}(\epsilon,e^-),x^{[r^*,\tilde\Lambda]}(0,e^-)\right)
+
d\left(x^{[r^*,\tilde\Lambda]}(0,e^-),e(\tilde\Lambda(0))\right) 
\\
&<  \frac{1}{3}\delta+\frac{1}{3}\delta = \frac{2}{3}\delta  \quad\mbox{for all}\quad (\T_\alpha,\T_\beta)\in{\cal{N}}. 
\end{split}
\label{eq:aux1}
\end{align}
 Similarly, using~\eqref{eq:dH}, note that
\begin{align}
\begin{split}
& \left|d_s\left(e(\tilde\Lambda(0)),\Theta^{[r^*,\tilde\Lambda]}(\epsilon)\right) - d_s\left(e(\tilde\Lambda(0)),\theta(\tilde\Lambda(\epsilon)\right)\right| \\
&\le
d_H\left(\Theta^{[r^*,\tilde\Lambda]}(\epsilon),\theta(\tilde\Lambda(\epsilon)\right) \\
&<  \frac{1}{3}\delta \quad\mbox{for all}\quad (\T_\alpha,\T_\beta)\in{\cal{N}}.
\end{split}
\label{eq:aux2}
\end{align}
The triangle inequality  $|a-b| \le |a-c| + |c-b|$ for $a,b,c\in \R$, together with~\eqref{eq:aux1} and~\eqref{eq:aux2}, gives
\begin{align}
\begin{split}
&\left|
d_s\left(x^{[r^*,\tilde\Lambda]}(\epsilon,e^-),\Theta^{[r^*,\tilde\Lambda]}(\epsilon)\right) 
- 
d_s\left(e(\tilde\Lambda(0)),\theta(\tilde\Lambda(\epsilon)\right)
\right|\\
&\le
d_H\left(\Theta^{[r^*,\tilde\Lambda]}(\epsilon),\theta(\tilde\Lambda(\epsilon))\right)
+
d\left(x^{[r^*,\tilde\Lambda]}(\epsilon,e^-),e(\tilde\Lambda(0))\right)
\\
& < \frac{1}{3}\delta + \frac{2}{3}\delta = \delta\quad\mbox{for all}\quad (\T_\alpha,\T_\beta)\in{\cal{N}}.
\end{split}
\label{eq:aux3}
\end{align}
\begin{figure}[t]
    \centering
    \includegraphics[width=14.5cm]{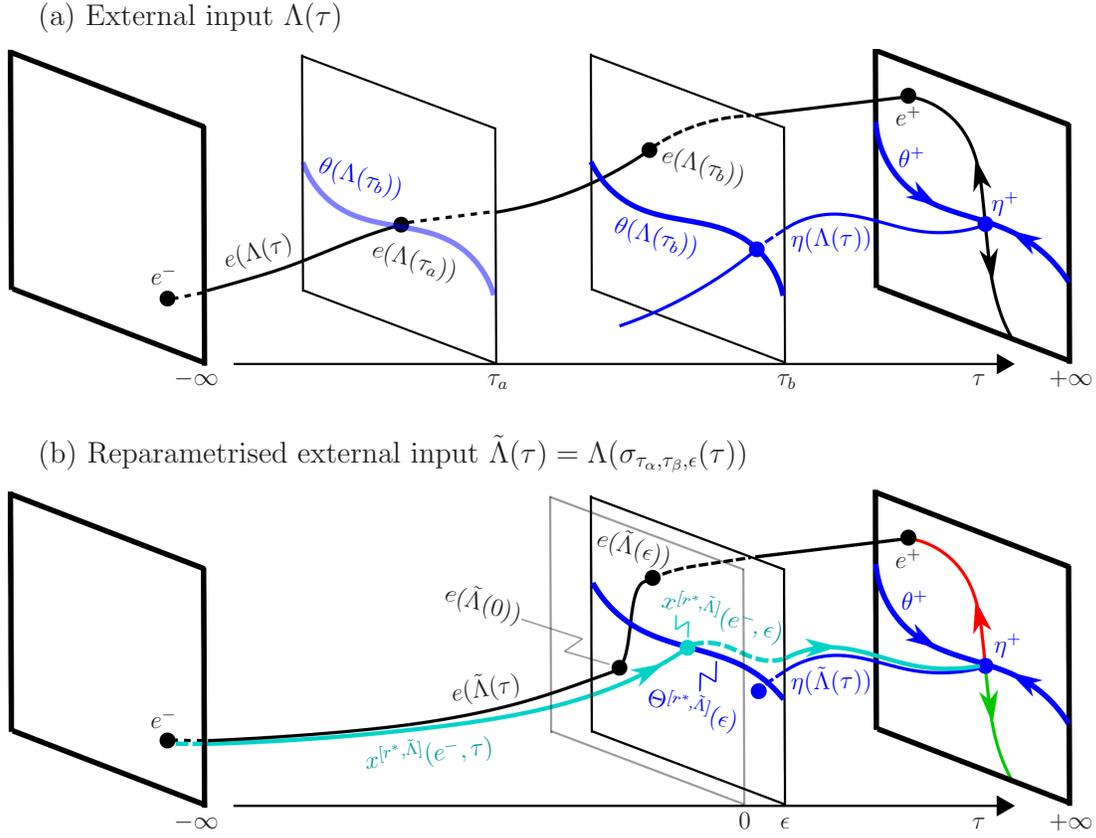}
    \caption{Construction of the time reparametrization in the proof of Theorem~\ref{thm:Rtip} for $x\in\R^2$. (a) Forward threshold instability of the moving sink $e(\Lambda(\T_1))$ due to crossing the moving regular threshold $\theta(\Lambda(\T_2))$
    for $(\T_1,\T_2)=(\T_a,\T_b)$. (b) For some fixed $r=r^*>0$, there is a reparametrisation $\tilde{\Lambda}(\T) = \Lambda(\sigma_{\T_\alpha,\T_\beta,\epsilon}(\T))$ such that the (cyan) pullback attractor $x^{[r,\tilde{\Lambda}]}(e^-,\T)$ enters a regular R-tipping threshold $\Theta^{[r,\tilde{\Lambda}]}(\T)$ (see the snapshot at time $\T=\epsilon$) for a suitable choice of $\epsilon$ and $(\T_\alpha,\T_\beta)=(\T_a^*,\T_b^*)$ shown in Figure~\ref{fig:N}. The pullback attractor then tracks $\eta(\tilde{\Lambda}(\tau))$ and limits to the regular equilibrium R-tipping edge state $\eta^+$. For non-degenerate R-tipping, the pullback attractor switches (red/green) edge tail on crossing $r^*$.
    }
    \label{fig:threshold}
\end{figure}
Finally, note from~\eqref{eq:aux0} that
$$
d_s\left(e(\tilde\Lambda(0)),\theta(\tilde\Lambda(\epsilon)\right)  =\Delta_\Lambda(\T_\alpha,\T_\beta),
$$
and use~\eqref{eq:aux3} to arrive at
\begin{align}
\label{eq:aux4}
\begin{split}
&\Delta_\Lambda(\T_\alpha,\T_\beta) - \delta < d_s\left(x^{[r^*,\tilde\Lambda]}(\epsilon,e^-),\Theta^{[r^*,\tilde\Lambda]}(\epsilon)\right)<\Delta_\Lambda(\T_\alpha,\T_\beta) + \delta\\&\mbox{for all}\quad (\T_\alpha,\T_\beta)\in{\cal{N}}.
\end{split}
\end{align}
For $(\T_\alpha,\T_\beta)=(\T_{a}^+,\T_{b}^+)$, it follows from~\eqref{eq:Delta_delta} and~\eqref{eq:aux4} that
$$
0 < d_s\left(x^{[r^*,\tilde\Lambda]}(\epsilon,e^-),\Theta^{[r_c,\tilde\Lambda]}(\epsilon)\right) <2\delta.
$$
The same argument applied to  $(\T_\alpha,\T_\beta)=(\T_{a}^-,\T_{b}^-)$ gives
$$
-2\delta < d_s\left(x^{[r^*,\tilde\Lambda]}(\epsilon,e^-),\Theta^{[r^*,\tilde\Lambda]}(\epsilon)\right) < 0.
$$
Now consider pairs $(\T_\alpha,\T_\beta)$ on the line in $\mathcal{N}$ from $(\T_a^+,\T_b^+)$ to $(\T_a^-,\T_b^-)$; see the green line in Figure~\ref{fig:N}. Noting that 
$$
d_s\left(x^{[r^*,\tilde\Lambda]}(\epsilon,e^-),\Theta^{[r^*,\tilde\Lambda]}(\epsilon)\right),
$$
is continuous on this line, the intermediate value theorem 
 guarantees a choice of  $(\T_{\alpha},\T_\beta)= (\T_{\alpha}^*,\T_\beta^*)$ on this line such that
$$
d_s\left(x^{[r^*,\tilde\Lambda]}(\epsilon,e^-),\Theta^{[r^*,\tilde\Lambda]}(\epsilon)\right)=0.
$$
It then follows from the properties of $\Theta^{[r]}(\T)$ in Definition~\ref{def:rtipthres} that
$$
x^{[r^*,\tilde\Lambda]}(\T,e^-)\rightarrow \eta^+\;\;\mbox{as}\;\; \T\to +\infty,
$$
for the chosen $0 < \epsilon < \min\{\epsilon_1,\epsilon_2 \}$ and  $(\T_{\alpha},\T_\beta)= (\T_{\alpha}^*,\T_\beta^*)\in{\cal{N}}$; see Figure~\ref{fig:threshold} for an illustration of this.
Hence we conclude there is R-tipping for this $\tilde{\Lambda}(\T)$  at $r=r^*$.
\qed

\end{document}